\documentclass{article}

\usepackage{geometry}
\usepackage{graphicx} 
\usepackage{epstopdf} 

\usepackage[colorlinks, hypertexnames=false]{hyperref}
\hypersetup{linkcolor=blue, citecolor=red} 

\usepackage{bm}
\usepackage{amsmath} 
\usepackage{amssymb} 
\usepackage{stmaryrd} 
\usepackage{multirow}

\usepackage{amsthm} 
\usepackage{mathtools} 

\usepackage[dvipsnames]{xcolor}
\definecolor{red}{rgb}{1.00,0.00,0.00}
{\numberwithin{equation}{section}
\setlength{\parindent}{1em}

\newtheorem{theorem}{Theorem}[section]
\newtheorem{lemma}{Lemma}[section]
\newtheorem{remark}{Remark}[section]

\newcommand{\normmm}[1]{{\left\vert\kern-0.25ex\left\vert
\kern-0.25ex\left\vert #1
    \right\vert\kern-0.25ex\right\vert\kern-0.25ex\right\vert}}

\geometry{left=3cm,right=3cm,top=2.5cm,bottom=2cm}

\begin{document}           

\title{Locking free staggered DG method for the Biot system of poroelasticity on general polygonal meshes}
\author{Lina Zhao\footnotemark[1]\qquad
\;Eric Chung\footnotemark[2]\qquad
\;Eun-Jae Park\footnotemark[3]}
\renewcommand{\thefootnote}{\fnsymbol{footnote}}
\footnotetext[1]{Department of Mathematics,The Chinese University
of Hong Kong, Hong Kong SAR, China. ({lzhao@math.cuhk.edu.hk}).}
\footnotetext[2]{Department of Mathematics, The Chinese University
of Hong Kong, Hong Kong SAR, China. ({tschung@math.cuhk.edu.hk}).}
\footnotetext[3]{School of Mathematics and Computing
(Computational Science and Engineering), Yonsei University, Seoul
03722, Korea. ({ejpark@yonsei.ac.kr}).} \maketitle

\begin{abstract}
 In this paper we propose and analyze a staggered discontinuous Galerkin method for a five-field formulation of the Biot system of poroelasticity on general polygonal meshes. Elasticity is equipped with stress-displacement-rotation formulation with weak stress symmetry for arbitrary polynomial orders, which extends the piecewise constant approximation developed in (L. Zhao and E.-J. Park, SIAM J. Sci. Comput. 42:A2158–A2181,2020). The proposed method is locking free and can handle highly distorted grids possibly including hanging nodes, which is desirable for practical applications. We prove the convergence estimates for the semi-discrete scheme and fully discrete scheme for all the variables in their natural norms. In particular, the stability and convergence analysis do not need a uniformly positive storativity coefficient. Moreover, to reduce the size of the global system, we propose a five-field formulation based fixed stress splitting scheme, where the linear convergence of the scheme is proved. Several numerical experiments are carried out to confirm the optimal convergence rates and the locking-free property of the proposed method.
\end{abstract}


\pagestyle{myheadings} \thispagestyle{plain}
\markboth{ZhaoChung} {Locking free SDG for poroelasticity}

\section{Introduction}

The Biot system of poroelasticity \cite{Biot41,Showalter00} models fluid flow within deformable porous media, and has been extensively studied in the literature due to its broad range of applications such as geosciences, carbon sequestration and biomedical applications. The system is a fully coupled system, consisting of an equilibrium equation for the solid and a mass balance equation for the fluid. A great amount of effort has been devoted to developing efficient numerical schemes for the Biot system. The two-field displacement-pressure formulation is studied by numerous numerical methods such as finite difference method \cite{Gaspar03}, finite volume method \cite{Nordbotten16}, finite element methods \cite{Murad92,Murad94,Murad96,Rodrigo16}, coupling of HHO method with SWIP method \cite{Boffi16}, weak Galerkin method \cite{HuMu18} and HDG method \cite{Fu19}. The three-field displacement-pressure-Darcy velocity formulation with various couplings of continuous and discontinuous Galerkin methods, and mixed finite
element methods is studied in \cite{Phillips07,Phillips072,Yi13,Wheeler14,Oyarzua16,LeeMardal17,Yi17,Lee18}. Very recently, a stabilized hybrid mixed finite element method uniformly stable with respect to the physical parameters is proposed in \cite{Niu20}. Alternatively, fully-mixed formulations of the Biot system has also drawn great attention \cite{Korsawe05,Yi14,Lee16,Ambartsumyan20}. Typical difficulty encountered in the design of numerical schemes for the Biot system is to avoid locking and to get estimates independent of the storativity coefficient.

Therefore, the purpose of this paper is to develop and analyze a staggered discontinuous Galerkin (DG) method on general polygonal meshes for the Biot system of poroelasticity based on a five-field formulation.
As a new generation of numerical schemes, staggered DG method offers many salient features such as the easy treatment of general polygonal elements possibly including hanging nodes, the local and global mass conservation and superconvergence. All these features make it a good candidate for practical applications, thus it has been extensively studied for a wide range of partial differential equations arising from physical applications \cite{ChungEngquist06,ChungWave09,ChungLee11,ChungKim13,KimChungStokes13,Chung14,Cheung15,Chung16,ChungDu17,ChungQiu17,Du18,
LinaParkdiffusion18,ZhaoParkShin19,LinaParkelasticity20,ZhaoPark20,ZhaoChungLam20,LinaKim20,LinaEricPark20}.
Therefore, it is of great importance to develop a staggered DG method for the Biot system of poroelasticity. To this end, we will first extend staggered DG method developed in \cite{LinaParkelasticity20} with piecewise constant approximations to arbitrary polynomial orders, where the symmetry of stress is weakly imposed via the carefully designed approximation space for rotation. Then it is combined with staggered DG method proposed in \cite{LinaParkdiffusion18} to discretize the Biot system of poroelasticity.

The resulting formulation is fully mixed with stress tensor, the fluid flux, the rotation, the displacement and the pressure as the unknowns. At a first sight this appears to increase the computational burden by introducing additional variables compared to methods that only involve displacement and pressure. However, the introduction of these variables enables us to achieve all the variables of physical interest directly without resorting to postprocessing, which is actually more accurate. We emphasize that the rotation variable introduced here is a scalar field, which can reduce the size of the global system. Another advantage of our method is that no further stabilization or penalty term is needed in contrast to other DG methods. In addition, local conservation is preserved on the dual mesh. It is worth mentioning that the mass matrix is block diagonal, which allows local elimination of the stress, velocity and rotation. As such, we can exploit local elimination to reduce the global saddle point system to displacement-pressure system. Since this procedure is similar to that proposed in \cite{Ambartsumyan20}, where block diagonal matrices can be generated thanks to the vertex quadrature rule utilized, we will not repeat here. Instead, we will introduce fixed stress splitting scheme to reduce the size of the global system. Fixed stress splitting scheme is a popular method for iteratively solving the Biot system, but has not been addressed for the formulation we exploited. Thus, we will propose a fixed stress splitting scheme for our formulation and prove the linear convergence of the given scheme. The main ingredient is to decompose the global system into two subproblem, then we can solve the flow problem and the mechanics problem separately. In particular, as mentioned before, the mass matrix is block diagonal, thus we can also apply local elimination for each problem, consequently, we can obtain a system which is solely based on displacement and pressure.

We analyze the unique solvability, stability and convergence estimates for the semidiscrete continuous-in-time and the fully discrete methods. Note that a different approach from \cite{LeeKim16} is pursued to prove the inf-sup condition for the elasticity equation. The optimal convergence rates for all the variables in their natural norms are analyzed. In addition, the error estimates are independent of the anisotropy of the permeability of the porous media. More importantly, the stability and convergence estimates are independent of the storativity coefficient $c_0$ and are valid for $c_0=0$. Moreover, like the linear elasticity \cite{LinaParkelasticity20}, the error bounds are robust for nearly incompressible materials. The numerical experiments including cantilever bracket benchmark problem presented verify the optimal convergence rates and locking free property of the proposed method.

The rest of the paper is organized as follows. In the next section, we describe the model problem and the associated weak formulation. Then in section~\ref{sec:discrete}, we derive the discrete formulation for the Biot system of poroelasticity and prove the unique solvability. Convergence analysis for semi-discrete and fully discrete methods is presented in section~\ref{sec:error}. In section~\ref{sec:fixed}, a fixed stress splitting scheme is introduced and analyzed. Several numerical experiments are carried out in section~\ref{sec:numerical} to confirm the proposed theories. Finally, a conclusion is given.


\section{Model problem and weak formulation}\label{sec2}

In this section we introduce the poroelasticity system and its fully mixed weak formulation, which lays foundation for the derivation of our staggered DG scheme. Let $\Omega\subset \mathbb{R}^2$ be a simply connected bounded domain, occupied by a poroelastic media saturated with fluid. Then
the stress-strain constitutive relationship for the poroelastic body is given by
\begin{align}
A\sigma_e=\epsilon(u),\label{eq:constitutive}
\end{align}
where $\epsilon(u)=\frac{1}{2}(\nabla u+\nabla u^T)$ and $\sigma_e=2\mu \epsilon(u)+\lambda \nabla\cdot uI$. Here $A$ is a symmetric, bounded and uniformly positive definite linear operator representing the compliance tensor, $\sigma_e$ is the elastic strain, $u$ is the solid displacement.

In the case of a homogeneous and isotropic body, we have
\begin{align*}
A\sigma = \frac{1}{2\mu} (\sigma-\frac{\lambda}{2\mu+2\lambda}tr(\sigma)I),
\end{align*}
where $I$ is the $2\times 2$ identity matrix and $\mu>0$, $\lambda\geq 0$ are the Lam\'{e} coefficients. In this case the elastic stress is $\sigma_e=2\mu \epsilon(u)+\lambda \nabla\cdot u I$.
The poroelastic stress, which includes the effect of the fluid pressure $p$, is given as
\begin{align}
\sigma=\sigma_e-\alpha pI,\label{eq:sigmap}
\end{align}
where $0<\alpha\leq 1$ is the Biot-Willis constant.

Given a vector field $f$ representing the body forces and a source term $q$, the quasi-static Biot system that governs the fluid flow within the poroelastic media is given as follows:
\begin{align}\label{eq:model}
\begin{alignedat}{2}
-\nabla\cdot \sigma &=f \quad && {\rm in~} \Omega\times [0,T],\\
K^{-1}z+\nabla p&=0\quad && {\rm in~}\Omega\times [0,T],\\
\frac{\partial}{\partial t}(c_0p+\alpha \nabla\cdot u)+\nabla\cdot z&=q\quad && {\rm in~} \Omega\times [0,T],
\end{alignedat}
\end{align}
where $z$ is the Darcy velocity, $c_0\geq 0$ is a mass storativity coefficient, $T>0$ is a finite time and $K$ is a symmetric and positive definite tensor representing the permeability of the porous media divided by the fluid viscosity, i.e., there exist two strictly positive real numbers $K_1$ and $K_2$ satisfying for almost every $x\in \Omega$ and all $\xi\in \mathbb{R}^2$ such that $|\xi|=1$
\begin{equation*}
    0<K_1\leq K(x)\xi\cdot \xi\leq K_2.
\end{equation*}
Further, we introduce the global anisotropy ratio
\begin{align}
\varrho_B = \frac{K_2}{K_1}\label{eq:aniso-ratio}.
\end{align}
For the sake of simplicity, we assume that the system satisfies homogeneous Dirichlet boundary conditions for $u$ and $p$, i.e., $u=0,p=0$ on $\partial \Omega$ and the initial condition $p(x,0)=p^0$.
Now we are ready to derive the weak formulation, we have from \eqref{eq:constitutive} and \eqref{eq:sigmap}
\begin{align*}
\nabla\cdot u=tr(\epsilon(u))=tr(A\sigma_e)=trA(\sigma+\alpha pI),
\end{align*}
which can be substituted in the third equation of \eqref{eq:model} to get
\begin{align*}
\partial_t(c_0p+\alpha trA(\sigma+\alpha pI))+\nabla\cdot z=q.
\end{align*}
In the weakly symmetric stress formulation, we allow for $\sigma$ to be non-symmetric and introduce the Lagrange multiplier $\gamma = \mbox{rot}(u)/2 $, $\mbox{rot}(u)=-\frac{\partial u_1}{\partial y}+\frac{\partial u_2}{\partial x}$. The constitutive equation \eqref{eq:constitutive} can be rewritten as
\begin{align*}
A(\sigma+\alpha pI) = \nabla u-\gamma \chi,
\end{align*}
where
\begin{align*}
\chi = \left(
         \begin{array}{cc}
           0 & -1 \\
           1 & 0 \\
         \end{array}
       \right).
\end{align*}
By the preceding arguments, we can propose the following mixed weak formulation for the Biot system of poroelasticity: Find $(\sigma,u,\gamma,z,p)\in L^2(\Omega)^{2\times 2}\times H^1_0(\Omega)^2\times L^2(\Omega)\times L^2(\Omega)^2\times H^1_0(\Omega)$ such that
\begin{align}
(A(\sigma+\alpha pI),\psi)-(\nabla u,\psi)+(\gamma,as(\psi))&=0\hspace{0.9cm} \forall \psi\in L^2(\Omega)^{2\times 2},\label{eq:weak1}\\
(\sigma,\nabla v)&=(f,v)\quad \forall v\in H^1_0(\Omega)^2,\label{eq:weak2}\\
(as(\sigma), \eta)&=0\hspace{1cm}  \forall \eta\in L^2(\Omega),\label{eq:weak3}\\
(K^{-1}z,\xi)+(\nabla p,\xi)&=0\hspace{1cm}  \forall \xi\in L^2(\Omega)^2,\label{eq:weak4}\\
c_0(\partial_tp,w)+\alpha(\partial_tA(\sigma+\alpha pI),wI)-(z,\nabla w)&=(q,w)\quad \forall w\in H^1_0(\Omega)\label{eq:weak5},
\end{align}
where we use $(tr A\psi,w)=(A\psi,wI)$ and $as(\psi)=-\psi_{12}+\psi_{21}$ for $\psi\in L^2(\Omega)^{2\times 2}$.

For the initial condition we can set $z^0=-K\nabla p^0$ which naturally satisfies \eqref{eq:weak4}. We can also determine $(\sigma^0,u^0,\gamma^0)$ by solving the elasticity problem \eqref{eq:weak1}-\eqref{eq:weak3} with $p^0$ given as initial data. We refer to the initial data obtained by this procedure as compatible initial data.

Before closing this section, we also introduce some notations that will be used later. Let $D\subset \mathbb{R}^2$. By $(\cdot,\cdot)_D$, we denote the scalar product in $L^2(D):(p,q)_D:=\int_D p \;q\;dx$. When $D$ coincides with $\Omega$, the subscript $\Omega$ will be dropped. We use the same notation for the scalar product in $L^2(D)^2$ and in $L^2(D)^{2\times 2 }$. More precisely, $(\xi,w)_D:=\sum_{i=1}^2 (\xi^i,w^i)$ for $\xi,w\in L^2(D)^2$ and $(\psi,\zeta)_D:=\sum_{i=1}^2\sum_{j=1}^2 (\psi^{i,j},\zeta^{i,j})_D$ for $\psi,\zeta\in L^2(D)^{2\times 2}$. The associated norm is denoted by $\|\cdot\|_{0,D}$.
Given an integer  $m\geq 0$ and $n\geq 1$, $W^{m,n}(D)$ and $W_0^{m,n}(D)$ denote the usual Sobolev space provided the norm
and semi-norm $\|v\|_{W^{m,n}(D)}=\{\sum_{|\ell|\leq m}\|D^\ell
v\|^n_{L^n(D)}\}^{1/n}$, $|v|_{W^{m,n}(D)}=\{\sum_{|\ell|=
m}\|D^\ell v\|^n_{L^n(D)}\}^{1/n}$. If $n=2$ we usually write
$H^m(D)=W^{m,2}(D)$ and $H_0^m(D)=W_0^{m,2}(D)$,
$\|v\|_{H^m(D)}=\|v\|_{W^{m,2}(D)}$ and $|v|_{H^m(D)}=|v|_{W^{m,2}(D)}$.
In addition, for an integer $l\geq 0$ we define spaces of vector valued functions such as $H^l(0,T;H^m(\Omega))$ and $L^{\infty}(0,T;H^m(\Omega))$ with the norms
\begin{align*}
\|\psi\|_{H^l(0,T;H^m(\Omega))}=\sum_{i=0}^l\Big(\int_0^{T} \|d_t^i\psi(t)\|_{H^m(\Omega)}^2\;dt\Big)^{1/2},\quad
\|\psi\|_{L^{\infty}(0,T;H^m(\Omega))}=\max_{0\leq t\leq T}\|\psi(t)\|_{H^m(\Omega)}.
\end{align*}
In the sequel, we use $C$ to denote a positive constant independent of the meshsize and of the anisotropy ratio $ \varrho_B$ in \eqref{eq:aniso-ratio}, which may have different values at different occurrences. In addition, for any $\psi\in L^2(\Omega)^{2\times 2}$, $(A\psi,\psi)^{1/2}$ is a norm equivalent to $\|\psi\|_0$, which will be denoted by $\|A^{1/2}\psi\|_0$ throughout the paper.

\section{Staggered DG discretization}\label{sec:discrete}

In this section, we present the staggered DG discretization for the Biot system of poroelasticity. Specifically, a fully mixed formulation will be used, and the symmetry of stress is imposed weakly via the introduction of a lagrange multiplier.

To begin with, we introduce the construction of staggered meshes, which lays foundation for the construction of finite element spaces suitable for our setting. Let $\mathcal{T}_u$ be a shape-regular family of matching meshes covering $\Omega$ exactly, where hanging nodes are allowed. Note that the shape-regularity for general polygonal meshes can be described as follows: (1) Every element $M\in \mathcal{T}_u$ is star-shaped with respect to a ball of radius $\geq \rho_B h_E$, where $\rho_B$ is a positive constant and $h_E$ is the diameter of $M$; (2) For every element $M\in \mathcal{T}_u$ and every edge $e\subset \partial M$, it satisfies $h_e\geq \rho_e h_M$, where $\rho_e$ is a positive constant and $h_e$ represents the length of edge $e$. For an element $M\in \mathcal{T}_u$, we select an interior point $\nu$ and connect it to all the vertices of $M$. The resulting sub-triangulation is denoted as $\mathcal{T}_h$. We remark that $\nu$ is an arbitrary point interior to $M$ and for simplicity we choose it as center point. The union of all the edges in the primal partition $\mathcal{T}_u$ is denoted as $\mathcal{F}_u$. We use $\mathcal{F}_u^0$ to represent the subset of $\mathcal{F}_u$, that is the set of edges in $\mathcal{F}_u$ that does not lie on $\partial \Omega$. In addition, the edges generated during the subdivision process due to the connection of the interior point $\nu$ to all the vertices are called dual edges, which is denoted as $\mathcal{F}_p$. In addition, we define $\mathcal{F}=\mathcal{F}_u\cap \mathcal{F}_p$ and $\mathcal{F}=\mathcal{F}_u^0\cap \mathcal{F}_p$. For each triangle $\tau\in \mathcal{T}_h$, we let $h_\tau$ be the diameter of $\tau$ and $h=\max\{h_\tau,\tau\in \mathcal{T}_h\}$. Next, we introduce the dual mesh. For each interior edge $e\in \mathcal{F}_u^0$, we use $D(e)$ to represent the dual mesh, which is the union of the two triangles in $\mathcal{T}_h$ sharing the common edge $e$. For each edge $e\in \mathcal{F}_u\backslash \mathcal{F}_u^0$, we use $D(e)$ to denote the triangle in $\mathcal{T}_h$ having the edge $e$, see Figure~\ref{fig:mesh} for an illustration. For each edge $e$, we define a unit normal vector $n_e$ as follows: If $e$ is an interior edge, then $n_e$ is the unit normal vector of $e$ pointing towards the outside of $\Omega$. If $e$ is an interior edge, we then fix $n_e$ as one of the two possible unit normal vectors on $e$. When there is no ambiguity, we use $n$ instead of $n_e$ to simplify the notation. Let $k\geq 0$ be the order of approximation, for every $\tau\in \mathcal{T}_h$ and $e\in \mathcal{F}$, we define $P^k(\tau)$ and $P^k(e)$ as the spaces of polynomials of degree less than or equal to $k$ on $\tau$ and $e$, respectively.

\begin{figure}[t]
\centering
\includegraphics[width=0.65\textwidth]{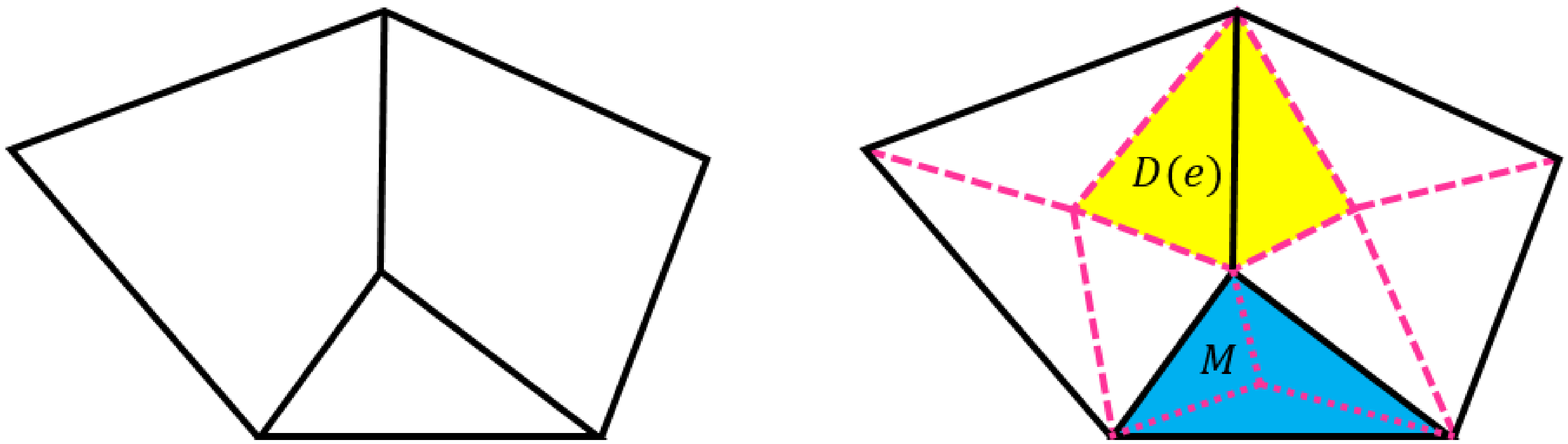}
\caption{Schematic of the primal mesh and the dual mesh. Solid lines represent edges in $\mathcal{F}_u$ and dashed lines represent edges in $\mathcal{F}_p$.}
\label{fig:mesh}
\end{figure}

For a scalar, vector, or tensor function $v$ that is double-valued on an interior edge $e\in \mathcal{F}^0$, its jump and average on $e$ are defined as
\begin{align*}
[v]=v\mid_{\tau_1}-v\mid_{\tau_2},\quad \{v\}=\frac{1}{2}(v\mid_{\tau_1}+v\mid_{\tau_2}),
\end{align*}
where $\tau_1$ and $\tau_2$ are the two triangles belonging to $\mathcal{T}_h$ which share the common edge $e$. We set $[v]=v\mid_{\tau_1}$ and $\{v\}=v\mid_{\tau_1}$ on boundary edges.

Now we are ready to introduce the finite element spaces that will be used for the construction of our method. First, the locally $H^{1}(\Omega)-$conforming
SDG space $S_h$ is defined by
\begin{equation*}
S_{h}:=\{w\in L^2(\Omega): w\mid_{\tau}\in P^{k}(\tau),[w]\mid_e=0 \;\forall e\in \mathcal{F}_u; w\mid_{\partial \Omega}=0\},
\end{equation*}
which is equipped by the following norm
\begin{align*}
\|w\|_Z^2=\sum_{\tau\in \mathcal{T}_h}\|\nabla w\|_{0,\tau}^2+\sum_{e\in \mathcal{F}_p}h_e^{-1}\|[w]\|_{0,e}^2.
\end{align*}
For a function $v=(v_1,v_2)\in [S_h]^2$ we define $\|v\|_h^2=\|v_{1}\|_Z^2+\|v_{2}\|_Z^2$.
%

Next, the locally $H(\mbox{div};\Omega)$-conforming
SDG space $\Sigma_h$ is defined by
\begin{equation*}
\Sigma_{h}:=\{\xi\in L^2(\Omega)^2: \xi\mid_{\tau} \in P^{k}(\tau)^2\;\forall \tau \in \mathcal{T}_{h},[\xi \cdot n]\mid_e=0\;\forall e\in \mathcal{F}_p\}.
\end{equation*}

Finally, the locally $H^{1}(\Omega)-$conforming finite element space is defined as
\begin{equation*}
M_{h}:=\{\eta\in L^2(\Omega): \eta\mid_{\tau}\in P^{k}(\tau)\;\forall \tau\in \mathcal{T}_h, [\eta]\mid_e=0\;\forall e\in \mathcal{F}_p\}.
\end{equation*}

We can derive the staggered DG discretization for the Biot system of poroelasticity by following \cite{LinaParkdiffusion18,LinaParkelasticity20}, which reads as follows: Find $(\sigma_h,u_h,\gamma_h,z_h,p_h)\in [\Sigma_h]^2\times [S_h]^2\times M_h\times \Sigma_h\times S_h$ such that
\begin{align}
(A(\sigma_h+\alpha p_hI),\psi)-B_h^*(u_h,\psi)+(\gamma_h,as(\psi))&=0\hspace{0.9cm} \forall \psi\in [\Sigma_h]^2,\label{eq:semi-discrete1}\\
B_h(\sigma_h,v)&=(f,v)\quad \forall v\in [S_h]^2,\label{eq:semi-discrete2}\\
(as(\sigma_h), \eta)&=0\hspace{1cm}\forall \eta\in M_h,\label{eq:semi-discrete3}\\
(K^{-1}z_h,\xi)+b_h^*(p_h,\xi)&=0\hspace{1.05cm}\forall \xi\in \Sigma_h,\label{eq:semi-discrete4}\\
c_0(\partial_tp_h,w)+\alpha(\partial_tA(\sigma_h+\alpha p_hI),wI)-b_h(z_h,w)&=(q,w)\quad \forall w\in S_h,\label{eq:semi-discrete5}
\end{align}
where
\begin{align*}
B_h(\psi,v) &= -\sum_{e\in \mathcal{F}_p}(\psi n,[v])_e+(\psi,\nabla v)\quad \forall (\psi,v)\in [\Sigma_h]^2\times [S_h]^2,\\
B_h^*(v,\psi)& = \sum_{e\in \mathcal{F}_u^0}(v,[\psi n])_e-(v,\nabla\cdot\psi)\quad \forall (\psi,v)\in [\Sigma_h]^2\times [S_h]^2
\end{align*}
and
\begin{align*}
b_h(\xi,w)&=-\sum_{e\in \mathcal{F}_p}(\xi\cdot n,[w])_e+(\xi,\nabla w)\quad \forall (\xi,w)\in \Sigma_h\times S_h,\\
b_h^*(w,\xi)&=\sum_{e\in \mathcal{F}_u^0} (w,[\xi\cdot n])_e-(w,\nabla\cdot \xi)\quad \forall (\xi,w)\in \Sigma_h\times S_h.
\end{align*}
The bilinear forms defined above satisfy the following adjoint properties
\begin{align*}
b_h(\xi,w) &= b_h^*(w,\xi)\quad \forall (w,\xi)\in S_h\times \Sigma_h,\\
B_h(\psi,v) & = B_h^*(v,\psi)\quad \forall (\psi,v)\in [\Sigma_h]^2\times [S_h]^2.
\end{align*}
In addition, the following inf-sup condition holds; cf. \cite{ChungWave09}.
\begin{align}
\|w\|_Z\leq C \sup_{\xi\in \Sigma_h}\frac{b_h^*(w,\xi)}{\|\xi\|_0}\quad \forall w\in S_h\label{eq:inf-sup-bh}
\end{align}
and
\begin{align}
\|v\|_h\leq C\sup_{\psi \in [\Sigma_h]^2} \frac{-B_h^*(v,\psi)}{\|\psi\|_0} \quad \forall v\in [S_h]^2.\label{eq:inf-sup-psi1}
\end{align}

\begin{remark}
{\rm (local conservation on the dual mesh). Taking $v=1$ in \eqref{eq:semi-discrete2} to be identically one over the dual mesh $D(e), \forall e\in \mathcal{F}_u$, we have
\begin{align*}
-(\sigma_h n_D,1)_{\partial D(e)}=(f,1)_{D(e)},
\end{align*}
where $n_D$ is the outward unit normal vector of $D(e)$. This property is the  crux for a posteriori estimates based on equilibrated fluxes; cf. \cite{LinaParkdiffusion18,ZhaoParkShin19}.

}
\end{remark}

Now we prove the inf-sup condition, which is the key to proving the stability and convergence estimates of the proposed method. The methodology exploited here is different from that of \cite{LeeKim16}. Indeed the norm used here bypasses the construction of divergence free functions in proving the convergence estimates.
\begin{lemma}
The following inf-sup condition holds
\begin{align}
\sup_{\psi_h\in [\Sigma_h]^2}\frac{-B_h^*(v_h,\psi_h)+(\eta_h,as(\psi_h))}{\|\psi_h\|_0}\geq C (\|v_h\|_h^2+\|\eta_h\|_0^2)^{\frac{1}{2}},\quad \forall (v_h,\eta_h)\in [S_h]^2\times M_h.\label{eq:inf-sup}
\end{align}

\end{lemma}

\begin{proof}


%
%
%
%
%
%
%
%
%

To prove the inf-sup condition \eqref{eq:inf-sup}, we apply Fortin interpolation operator (cf. \cite{Braess07}), namely, we need to build an interpolation operator $J_h: H^1(\Omega)^{2\times 2}\rightarrow [\Sigma_h]^2$ such that it satisfies for any $\psi\in H^1(\Omega)^{2\times 2}$
\begin{align*}
-B_h^*(v_h,\psi-J_h\psi)+(\eta_h, as(\psi-J_h\psi))&=0\quad \forall (v_h,\eta_h)\in [S_h]^2\times M_h,\\
\|J_h\psi\|_0&\leq C \|\psi\|_0.
\end{align*}
To this end we will first build $\psi_1\in [\Sigma_h]^2$ such that it satisfies
\begin{equation}
\begin{split}
B_h^*(v_h,\psi_1-\psi)&=0\quad \forall v_h\in[S_h]^2, \\
\|\psi_1\|_0&\leq C \|\psi\|_0,
\end{split}
\label{eq:inf-sup1}
\end{equation}
then we correct $\psi_1$ by using $\psi_2\in [\Sigma_h]^2$ which satisfies $B_h^*(v_h,\psi_2)=0,\;\forall v_h\in [S_h]^2$, in addition, there holds
\begin{align}
(as(\psi-\psi_2),\eta_h)&=(as(\psi_1),\eta_h)\quad \forall \eta_h \in M_h,\label{eq:fortin2}\\
\|\psi_2\|_0&\leq C\|\psi-\psi_1\|_0.\label{eq:fortin-ineq}
\end{align}
Since $B_h^*(\cdot,\cdot)$ satisfies the inf-sup condition \eqref{eq:inf-sup-psi1}, which means there exists
%
$\psi_1\in [\Sigma_h]^2$ that satisfies
\begin{equation*}
\begin{split}
B_h^*(v_h,\psi-\psi_1)&=0\quad \forall v_h\in [S_h]^2,\\
\|\psi_1\|_0&\leq  C\|\psi\|_0.
\end{split}
\end{equation*}
Next, we correct $\psi_1$ in the following way.
We define the following spaces for each primal element $M\in \mathcal{T}_u$:
\begin{align*}
V_h^k(M)&=\{\varsigma\mid_M\in H^1_0(M)^2:\varsigma\mid_\tau\in P^k(\tau)^2,\tau\subset M\},\\
Q_h^k(M)&=\{q\mid_M\in L^2(M):q\mid_\tau\in P^k(\tau),\tau\subset M, \int_Mq \;dx=0\}.
\end{align*}
For each primal element $M\in \mathcal{T}_u$, we consider the
stable approximation $V_h^{k+1}(M)\times Q_h^k(M)$ for the Stokes element such that
\begin{align*}
\mbox{curl} V_h^{k+1}(M)\subset \Sigma_h(M),
\end{align*}
where $\Sigma_h(M)$ denotes $\Sigma_h$ restricted to the primal element $M$. We then solve for $(\psi_h^*\mid_M, p_h^*\mid_M)\in V_h^{k+1}(M)\times Q_h^k(M)$ such that
\begin{align*}
( 2\mu \epsilon(\psi_h^*),\epsilon(\phi_h))_M+( p_h^*,\nabla\cdot\phi_h)_M&=(f,\phi_h)_M\quad \forall \phi_h\in V_h^{k+1}(M),\\
(q_h,\nabla \cdot \psi_h^*)_M &= (as(\psi-\psi_1),q_h)_M\quad \forall q_h\in Q_h^k(M).
\end{align*}
Since $\psi_h^*=0$ on $\partial M$ it is easy to check that the second equation also satisfies for a constant function.
Here we remark that the existing stable pair that fits into the above spaces can be Taylor-Hood element (cf. \cite{Brezzi91}).

Let $\psi_2\mid_M=-\mbox{curl}\psi_h^*\mid_M$, then we can check that
\begin{align*}
(as(\psi-\psi_2),\eta_h)_M&=(as(\psi_1),\eta_h)\quad \forall \eta_h\in M_h(M),\\
\|\psi_2\|_{0,M}&\leq C \|\psi-\psi_1\|_{0,M},
\end{align*}
which satisfies \eqref{eq:fortin2} by summing over all the primal elements. Here $M_h(M)$ denotes the space $M_h$ restricted to each primal element $M\in \mathcal{T}_u$. In addition, it also holds
\begin{align*}
B_h^*(v_h,\psi_2)=0\quad \forall v_h\in [S_h]^2.
\end{align*}
Therefore, we have built $\psi_1$ and $\psi_2$ which satisfy \eqref{eq:inf-sup1}-\eqref{eq:fortin-ineq}. Hence, by the well-known arguments from \cite{Braess07} we can infer that the existence of $\psi_1$ and $\psi_2$ leads to the inf-sup condition.

%

%
%
%
%
%
%
%

\end{proof}

\begin{remark}
{\rm
In the above proof, we employ Taylor-Hood stable pair, in fact we can also exploit Scott-Vogelius element. In which case, we need to be careful about the singular nodes, indeed if singular nodes exists then one needs to enforce additional restriction on those nodes (cf. \cite{Guzman18}), which is naturally satisfied for our choice of space $M_h$.

We also emphasize that we choose locally $H^1$-conforming space for $\gamma_h$ to enforce the symmetry of stress, which is different from that of \cite{LeeKim16}. In fact, we can also exploit fully discontinuous space for $\gamma_h$, which then leads to strongly symmetric stress. In this case, the final system will be singular for rectangular mesh if one chooses the center point as $\nu$. One can perturb the position of $\nu$ to avoid singularity. We want to develop a numerical scheme which is stable regardless of the position of $\nu$, thus we choose locally $H^1$-conforming space for $\gamma_h$.
}
\end{remark}

In order to prove the well-posedness of \eqref{eq:semi-discrete1}-\eqref{eq:semi-discrete5}, we will make use of the following theorem; cf. \cite{Showalter13}.
\begin{theorem}\label{thm:monotone}

Let the linear, symmetric and monotone operators $\mathcal{N}$ be given for the real vector space $E$ to its dual algebraic dual $E^*$, and let $E_b'$ be the Hilbert space which is the dual of $E$ with the seminorm
\begin{align*}
|x|_b=(\mathcal{N}x(x))^{\frac{1}{2}},\quad \forall x\in E.
\end{align*}
Let $\mathcal{M}\subset E\times E_b'$ be a relation with domain $D=\{x\in E:\mathcal{M}(x)\neq 0\}$. Assume $\mathcal{M}$ is monotone and $Rg(\mathcal{N}+\mathcal{M})=E_b'$. Then for each $x_0\in D$ and for each $\mathcal{F}\in W^{1,1}(0,T;E_b')$, there is a solution $x$ of
\begin{align}
\frac{d}{dt}(\mathcal{N}x(t))+\mathcal{M}(x(t))\ni \mathcal{F}(t), \quad \forall\; 0<t<T\label{eq:mono-form}
\end{align}
with
\begin{align*}
\mathcal{N}x\in W^{1,\infty}(0,T;E_b'),\;x(t)\in D,\; \forall \;0\leq t\leq T, \quad \mbox{and}\quad \mathcal{N}x(0)=\mathcal{N}x_0.
\end{align*}

\end{theorem}

\begin{theorem}

There exists a unique solution to \eqref{eq:semi-discrete1}-\eqref{eq:semi-discrete5}.

\end{theorem}

\begin{proof}

In order to fit \eqref{eq:semi-discrete1}-\eqref{eq:semi-discrete5} in the form of Theorem~\ref{thm:monotone}, we consider a slightly modified formulation, with \eqref{eq:semi-discrete1} differentiated in time and the new variables $\dot{u}_h$ and $\dot{\gamma_h}$ representing $d_tu_h$ and $d_t\gamma_h$, respectively:
\begin{align}
(\partial_tA(\sigma_h+\alpha p_hI),\psi)-B_h^*(\dot{u}_h,\psi)+(\dot{\gamma}_h,as(\psi))&=0\quad \forall \psi\in [\Sigma_h]^2.\label{eq:difft}
\end{align}
We now introduce the operators
\begin{align*}
&(A_{\sigma\sigma}\sigma_h,\psi) = (A\sigma_h,\psi),\quad (A_{\sigma p}\sigma_h,w)=(A\sigma_h,\alpha wI),\quad (A_{\sigma u} \sigma_h,v)=B_h(\sigma_h,v),\\
&(A_{\sigma \gamma}\sigma_h,\eta) = (as(\sigma_h), \eta),\quad (A_{zz}z_h,\xi) = (K^{-1}z_h,\xi),\quad (A_{zp}z_h,w) = b_h(z_h,w),\\
& (A_{pp} p_h,w) = c_0(p_h,w)+\alpha(A\alpha p_hI,wI).
\end{align*}
We have a system in the form of \eqref{eq:mono-form}, where
\begin{align*}
\dot{x}=\left(
          \begin{array}{c}
            \sigma_h \\
            \dot{u}_h \\
            \dot{\gamma}_h \\
            z_h \\
            p_h \\
          \end{array}
        \right), \;
\mathcal{N}= \left(
               \begin{array}{ccccc}
                 A_{\sigma \sigma} & 0 & 0 & 0 & A_{\sigma p}^T \\
                 0 & 0 & 0 & 0 & 0 \\
                 0 & 0 & 0 & 0 & 0 \\
                 0 & 0 & 0 & 0 & 0 \\
                 A_{\sigma p} & 0 & 0 & 0 & A_{pp} \\
               \end{array}
             \right),\;
\mathcal{M}=\left(
              \begin{array}{ccccc}
                0 & -A_{\sigma u}^T & A_{\sigma \gamma}^T & 0 & 0 \\
                A_{\sigma u} &0 & 0 & 0 & 0 \\
                A_{\sigma \gamma} &0 & 0 & 0 & 0 \\
                0 & 0 & 0 & A_{zz} & A_{zp}^T \\
                0 & 0 & 0 & -A_{zp} & 0 \\
              \end{array}
            \right)
\end{align*}
and $ \mathcal{F}=(0,f,0,0,q)^T$.

The dual space $E_b'$ is $L^2(\Omega)^{2\times 2}\times 0\times 0\times 0\times L^2(\Omega)$, and the condition $\mathcal{F}\in W^{1,1}(0,T;E_b')$ in Theorem~\ref{thm:monotone} allows for non-zero source terms only in the equations with time derivatives, which means $f=0$.
We can reduce our problem to a system with $f=0$ by solving for each $t\in (0,T]$ an elasticity problem with a source term $f$, cf. \cite{Showalter10}.
\begin{align*}
\begin{alignedat}{2}
(A\sigma_h^f,\psi)-B_h^*(\dot{u}_h^f,\psi)+(\dot{\gamma}^f_h,as(\psi))&=0\quad && \forall \psi\in [\Sigma_h]^2,\\
B_h(\sigma_h^f,v)&=(f,v)\quad &&\forall v\in [S_h]^2,\\
(as(\sigma_h^f), \eta)&=0\quad && \forall \eta\in M_h.
\end{alignedat}
\end{align*}
Subtracting the above equations from \eqref{eq:semi-discrete1}-\eqref{eq:semi-discrete5}, we can obtain
\begin{align}
(A((\sigma_h-\dot{\sigma}_h^f)+\alpha p_hI),\psi)-B_h^*(u_h-\dot{u}_h^f,\psi)+(\gamma_h-\dot{\gamma}_h^f,as(\psi))&
=(A(\sigma_h^f-\partial_t\sigma_h^f),\psi),\\
B_h(\sigma_h-\sigma_h^f,v)&=0,\\
(as(\sigma_h-\sigma_h^f), \eta)&=0,\\
(K^{-1}z_h,\xi)+b_h^*(p_h,\xi)&=0,\\
c_0(\partial_tp_h,w)+\alpha(\partial_tA((\sigma_h-\sigma_h^f)+\alpha p_hI),wI)-b_h(z_h,w)&=(q,w)-(\alpha \partial_tA\sigma_h^f,wI),
\end{align}
which results in a problem with a modified right hand side $\mathcal{F}=(A_{\sigma \sigma} (\sigma_h^f-\partial_t\sigma_h^f),0,0,0,q-A_{\sigma p} \partial_t \sigma_h^f)^T$.

The range condition $Rg(\mathcal{N}+\mathcal{M})=E_b'$ can be verified by showing that the square finite dimensional homogeneous system: Find $(\hat{\sigma}_h,\hat{u}_h,\hat{\gamma}_h,\hat{z}_h,\hat{p}_h)\in [\Sigma_h]^2\times [S_h]^2\times M_h\times \Sigma_h\times S_h$ such that
\begin{align}
(A(\hat{\sigma}_h+\alpha \hat{p}_hI),\psi)-B_h^*(\hat{u}_h,\psi)+(\hat{\gamma}_h,as(\psi))&=0\quad \forall \psi\in [\Sigma_h]^2,\label{eq:homo1}\\
B_h(\hat{\sigma}_h,v)&=0\quad \forall v\in [S_h]^2,\\
(as(\hat{\sigma}_h), \eta)&=0\quad \forall \eta\in M_h,\\
(K^{-1}\hat{z}_h,\xi)+b_h^*(\hat{p}_h,\xi)&=0\quad \forall \xi\in \Sigma_h,\\
c_0(\partial_t\hat{p}_h,w)+\alpha(\partial_tA(\hat{\sigma}_h+\alpha \hat{p}_hI),wI)-b_h(\hat{z}_h,w)&=0\quad \forall w\in S_h,\label{eq:homo5}
\end{align}
has only zero solution. Taking $(\psi,v,\eta,\xi,w)=(\hat{\sigma}_h,\hat{u}_h,\hat{\gamma}_h,\hat{z}_h,\hat{p}_h)$ in \eqref{eq:homo1}-\eqref{eq:homo5} and summing up the resulting equations yield
\begin{align*}
\|A^{\frac{1}{2}}(\hat{\sigma}_h+\alpha \hat{p}_hI)\|_0^2+c_0\|\hat{p}_h\|_0^2+\|K^{-\frac{1}{2}}\hat{z}_h\|_0^2=0,
\end{align*}
which gives $\hat{\sigma}_h+\alpha \hat{p}_hI=0$ and $\hat{z}_h=0$. By the inf-sup condition \eqref{eq:inf-sup-bh}, we have
\begin{align*}
\|\hat{p}_h\|_Z\leq C \|\hat{z}_h\|_0,
\end{align*}
thereby $\hat{p}_h=0$ by using the discrete Poincar\'{e}-Friedrichs inequality (cf. \cite{Brenner03}), consequently, $\hat{\sigma}_h=0$. The inf-sup condition \eqref{eq:inf-sup} and \eqref{eq:homo1} lead to
\begin{align*}
\|\hat{u}_h\|_h+\|\hat{\gamma}_h\|_0\leq C \|A^{\frac{1}{2}}(\hat{\sigma}_h+\alpha \hat{p}_hI)\|_0,
\end{align*}
which yields $\hat{u}_h=0$ and $\hat{\gamma}_h=0$.

Finally, we need compatible initial data $\dot{x}_0\in D$, i.e., $\mathcal{M}\dot{x}_0\in E_b'$. Let us first consider the initial data $x^0=(\sigma_h^0,u_h^0,\gamma_h^0,z_h^0,p_h^0)$ for the discrete formulation \eqref{eq:semi-discrete1}-\eqref{eq:semi-discrete5}. We take $x^0$ to be the elliptic projection of the initial data $\tilde{x}_0=(\sigma^0,u^0,\gamma^0,z^0,p^0)$ for the weak formulation \eqref{eq:weak1}-\eqref{eq:weak5}, which is constructed from $p^0$ by the procedure described at the end of section~\ref{sec2}. With the reduction to a problem with $f=0$, the construction satisfies $(\mathcal{N}+\mathcal{M})\tilde{x}^0\in E_b'$. Then by
\begin{align*}
(\mathcal{N}+\mathcal{M})x^0=(\mathcal{N}+\mathcal{M})\tilde{x}^0,
\end{align*}
we have $\mathcal{M}x^0=(\mathcal{N}+\mathcal{M})\tilde{x}^0-\mathcal{N}x^0\in E_b'$. For the initial data of the differentiated problem \eqref{eq:difft}, \eqref{eq:semi-discrete2}-\eqref{eq:semi-discrete5}, we simply take $\dot{x}^0=(\sigma_h^0,0,0,z_h^0,p_h^0)$, which also satisfies $\mathcal{M}\dot{x}^0\in E_b'$. Here $u_h^0$ and $\gamma_h^0$ are not needed for the differentiated problem, but will be used to recover the solution of the original problem.

Now all conditions of Theorem~\ref{thm:monotone} are satisfied and we conclude the existence of a solution of \eqref{eq:difft}, \eqref{eq:semi-discrete2}-\eqref{eq:semi-discrete5} with $\sigma_h(0)=\sigma_h^0$ and $p_h(0)=p_h^0$. By taking $t\rightarrow 0$ in \eqref{eq:semi-discrete4} and using that $z_h^0$ and $p_h^0$ satisfy \eqref{eq:semi-discrete4} at $t=0$, we also infer that $z_h(0)=z_h^0$.

Next, we recover the solution of the original problem. Let us define
\begin{align*}
u_h(t)=u_h^0+\int_0^t \dot{u}_h\;ds,\quad \gamma_h(t)=\gamma_h^0+\int_0^t \dot{\gamma}_h\;ds\quad \forall t\in (0,T].
\end{align*}
By construction, $u_h(0)=u_h^0$ and $\gamma_h(0)=\gamma_h^0$. Integrating \eqref{eq:difft} from $0$ to any $t\in (0,T]$ and using that $\sigma_h^0,u_h^0,\gamma_h^0$ satisfy \eqref{eq:semi-discrete1} at $t=0$, we conclude that \eqref{eq:semi-discrete1} holds for all $t$. This completes the existence proof. Uniqueness follows from the stability bound given in Theorem~\ref{thm:stability} in the next section.

\end{proof}

%
%

Before closing this section, we introduce some interpolation operators which will play an important role for later analysis.
 There exists the interpolation operator $I_h: H^1(\Omega)\rightarrow S_h$ satisfying for any $\omega\in H^1(\Omega)$
\begin{equation}
\begin{split}
(I_h \omega-\omega,\psi)_e&=0 \quad \forall \psi\in P^k(e), e\in \mathcal{F}_{u},\\
(I_h\omega-\omega,\psi)_\tau&=0 \quad \forall \psi\in P^{k-1}(\tau), \tau\in \mathcal{T}_h.
\end{split}
\label{def:Ih}
\end{equation}
Similarly, there also exists an operator $\Pi_h : [H^\delta(\Omega)]^2\rightarrow \Sigma_h,\delta>1/2$ such that for any $\varphi\in [H^\delta(\Omega)]^2$
\begin{equation}
\begin{split}
((\Pi_h \varphi-\varphi)\cdot n,\phi)_e&=0 \quad \forall \phi\in P^k(e), e\in \mathcal{F}_{p},\\
(\Pi_h\varphi-\varphi,\phi)_\tau&=0 \quad \forall \phi\in P^{k-1}(\tau)^2, \tau\in \mathcal{T}_h.
\end{split}
\label{def:Jh}
\end{equation}
The definition of the interpolation operators implies that
\begin{align}
b_h(z-\Pi_hz,w)&=0 \quad \forall w\in S_h,\label{eq:Pih}\\
b_h^*(p-I_hp,v)&=0 \quad \forall v\in \Sigma_h.\label{eq:Ih}
\end{align}
%
In addition for $\Pi_h\sigma=(\Pi_h\sigma_1,\Pi_h\sigma_2)\in [\Sigma_h]^2$ and $I_hu=(I_hu_1,I_hu_2)\in [S_h]^2$, we also have
\begin{align}
B_h(\sigma-\Pi_h\sigma,v)&=0\quad \forall v\in [S_h]^2,\label{eq:Pih2}\\
B_h^*(u-I_hu, \psi)&=0\quad \forall \psi\in [\Sigma_h]^2.\label{eq:Ih2}
\end{align}
It is easy to check that $I_h$ and $\Pi_h$ are polynomial preserving operators, which satisfy the following interpolation error estimates (cf. \cite{Ciarlet02,ChungWave09})
\begin{equation}
\begin{split}
\|\varphi-\Pi_h\varphi\|_0&\leq C h^{k+1}\|\varphi\|_{H^{k+1}(\Omega)}\quad \forall \varphi\in H^{k+1}(\Omega)^2,\\
\|\omega-I_h\omega\|_0&\leq C h^{k+1}\|\omega\|_{H^{k+1}(\Omega)}\quad \forall \omega\in H^{k+1}(\Omega).
\end{split}
\label{eq:interpolation1}
\end{equation}
In addition, we take $\pi_h$ to be the standard nodal interpolation operator, which satisfies
\begin{align}
\|\pi_h \zeta-\zeta\|_0\leq C h^{k+1}\|\zeta\|_{H^{k+1}(\Omega)}\quad \forall \zeta\in H^{k+1}(\Omega).\label{eq:interpolation2}
\end{align}

\section{Error analysis}\label{sec:error}

In this section, we prove the convergence estimates for the semi-discrete scheme and fully discrete scheme with backward Euler in time. In particular we will show that the stability and convergence error estimates are independent of the anisotropy ratio $ \varrho_B$ and of the storativity coefficient $c_0$, and are valid for $c_0=0$.

\subsection{Error analysis for semi-discrete scheme}

In this subsection we prove the stability and convergence estimates for the semi-discrete scheme. To begin, we consider the following stability estimate.
\begin{theorem}\label{thm:stability}
(stability).
Let $(\sigma_h,u_h,\gamma_h,z_h,p_h)$ be the numerical solution of \eqref{eq:semi-discrete1}-\eqref{eq:semi-discrete5}, then we have
\begin{equation*}
\begin{split}
&\|\sigma_h\|_{L^\infty(0,T;L^2(\Omega))}^2+\|p_h\|_{L^\infty(0,T;L^2(\Omega))}^2
+\|u_h\|_{L^\infty(0,T;L^2(\Omega))}^2+\|K^{-\frac{1}{2}}z_h\|_{L^\infty(0,T;L^2(\Omega))}^2\\
&\;+\|\gamma_h\|_{L^\infty(0,T;L^2(\Omega))}^2+ \|K^{-\frac{1}{2}}z_h\|_{L^2(0,T;L^2(\Omega))}^2+\|u_h\|_{L^2(0,T;L^2(\Omega))}^2
+\|p_h\|_{L^2(0,T;L^2(\Omega))}^2\\
&\;+\|\sigma_h\|_{L^2(0,T;L^2(\Omega))}^2+\|\gamma_h\|_{L^2(0,T;L^2(\Omega))}^2\leq
C\Big(\|f\|_{L^\infty(0,T;L^2(\Omega))}^2\\
&\;+\|f\|_{H^1(0,T;L^2(\Omega))}^2+\|q\|_{H^1(0,T;L^2(\Omega))}^2
+\|q\|_{L^\infty(0,T;L^2(\Omega))}^2
+\|p^0\|_{H^1(\Omega)}^2\Big).
\end{split}
\end{equation*}

\end{theorem}

\begin{proof}
Differentiating \eqref{eq:semi-discrete1} in time, taking $\psi=\sigma_h,v=\partial_tu_h,\eta=\partial_t\gamma_h, \xi=z_h,w=p_h$ in \eqref{eq:semi-discrete1}-\eqref{eq:semi-discrete5} and summing up the resulting equations yield
%
\begin{align*}
\frac{1}{2}\frac{d}{dt}\Big(\|A^{\frac{1}{2}}(\sigma_h+\alpha p_hI)\|_0^2+c_0\|p_h\|_0^2\Big)+\|K^{-\frac{1}{2}}z_h\|_0^2=(f,\partial_t u_h)+(q,p_h).
\end{align*}
Integrating over time implies
\begin{align*}
&\frac{1}{2}\Big(\|A^{\frac{1}{2}}(\sigma_h+\alpha p_hI)(t)\|_0^2+c_0\|p_h(t)\|_0^2\Big)+\int_0^t\|K^{-\frac{1}{2}}z_h\|_0^2\;ds=-\int_0^t (\partial_tf,u_h)\;ds\\
&\;+\int_0^t (q,p_h)\;ds
+\frac{1}{2}\Big(\|A^{\frac{1}{2}}(\sigma_h+\alpha p_hI)(0)\|_0^2+c_0\|p_h^0\|_0^2\Big)+(f,u_h)(t)-(f,u_h)(0).
\end{align*}
The Cauchy-Schwarz inequality and Young's inequality lead to
\begin{equation}
\begin{split}
&\|A^{\frac{1}{2}}(\sigma_h+\alpha p_hI)(t)\|_0^2+c_0\|p_h(t)\|_0^2+2\int_0^t\|K^{-\frac{1}{2}}z_h\|_0^2\;ds\leq \epsilon_1\Big(\|u_h(t)\|_0^2\\
&\;+\int_0^t (\|p_h\|_0^2+\|u_h\|_0^2)\;ds\Big)
+\frac{1}{2\epsilon_1}\Big(\|f(t)\|_0^2+\int_0^t(\|q\|_0^2+\|\partial_t f\|_0^2)\;ds\Big)\\
&\;+\|A^{\frac{1}{2}}(\sigma_h+\alpha p_hI)(0)\|_0^2+c_0\|p_h^0\|_0^2+\|u_h^0\|_0^2+\|f(0)\|_0^2.
\end{split}
\label{eq:estimate1}
\end{equation}
By \eqref{eq:semi-discrete1} and the inf-sup condition \eqref{eq:inf-sup}, we have
\begin{align*}
\|u_h\|_h+\|\gamma_h\|_0\leq C \sup_{\psi\in [\Sigma_h]^2}\frac{-B_h^*(u_h,\psi)+(as(\psi),\gamma_h)}{\|\psi\|_0}\leq C \|A^{\frac{1}{2}}(\sigma_h+\alpha p_hI)\|_0.
\end{align*}
%
%
%
Then an appeal to the discrete Poincar\'{e}-Friedrichs inequalities gives
\begin{align}
\|u_h\|_0\leq C\|A^{\frac{1}{2}}(\sigma_h+\alpha p_hI)\|_0.\label{eq:uh}
\end{align}
Taking $\psi=\sigma_h$, $v=u_h$ and $\eta=\gamma_h$ in \eqref{eq:semi-discrete1}-\eqref{eq:semi-discrete3}, and summing up the resulting equations, we deduce that
\begin{align}
\|\sigma_h\|_0^2\leq C (\|p_h\|_0^2+\epsilon_2\|u_h\|_0^2+\frac{1}{\epsilon_2}\|f\|_0^2).\label{eq:sigma}
\end{align}
Thus, we can infer from \eqref{eq:uh} and \eqref{eq:sigma} that
\begin{align*}
\|u_h\|_0^2\leq C\|A^{\frac{1}{2}}(\sigma_h+\alpha p_hI)\|_0^2\leq C (\|\sigma_h\|_0^2+\|p_h\|_0^2)\leq C\Big(\|p_h\|_0^2+\epsilon_2\|u_h\|_0^2+\frac{1}{\epsilon_2}\|f\|_0^2\Big).
\end{align*}
Choosing $\epsilon_2$ small enough yields
\begin{align}
\|u_h\|_0^2\leq C(\|p_h\|_0^2+\frac{1}{\epsilon_2}\|f\|_0^2).\label{eq:ep2}
\end{align}
On the other hand, we have from \eqref{eq:semi-discrete4} and the inf-sup condition \eqref{eq:inf-sup-bh}
\begin{align}
\|p_h\|_Z\leq C \sup_{\xi\in \Sigma_h}\frac{b_h^*(p_h,\xi)}{\|\xi\|_0}=C \sup_{\xi\in \Sigma_h}\frac{(K^{-1}z_h,\xi)}{\|\xi\|_0}\leq C \|K^{-1}z_h\|_0\label{eq:phZ}.
\end{align}
Combining the discrete Poincar\'{e}-Friedrichs inequalities, \eqref{eq:ep2} and \eqref{eq:phZ} lead to
\begin{align}
\int_0^t (\|u_h\|_0^2+\|p_h\|_0^2)\;ds\leq C\int_0^t (\|K^{-1}z_h\|_0^2+\|f\|_0^2)\;ds\label{eq:uht}.
\end{align}
Therefore, choosing $\epsilon_1$ small enough, we can infer from \eqref{eq:estimate1}, \eqref{eq:phZ} and \eqref{eq:uht} that
\begin{equation}
\begin{split}
&\|A^{\frac{1}{2}}(\sigma_h+\alpha p_hI)(t)\|_0^2+c_0\|p_h(t)\|_0^2+\|u_h(t)\|_0^2
+\|\gamma_h(t)\|_0^2\\
&\;+\int_0^t \Big(\|K^{-\frac{1}{2}}z_h\|_0^2+\|u_h\|_0^2+\|p_h\|_0^2+\|\sigma_h\|_0^2+\|\gamma_h\|_0^2\Big)\;ds\leq
C\Big(\|f(t)\|_0^2\\
&\;+\int_0^t(\|f\|_0^2+\|q\|_0^2+\|\partial_t f\|_0^2)\;ds+\|p_h^0\|_0^2
+\|u_h^0\|_0^2+\|f(0)\|_0^2+\|\sigma_h^0\|_0^2\Big).
\end{split}
\label{eq:semi1}
\end{equation}
Next, differentiating \eqref{eq:semi-discrete1}-\eqref{eq:semi-discrete4} in time and combining them with \eqref{eq:semi-discrete5} with the choices $(\psi,v,\eta,\xi,w)=(\partial_t \sigma_h,\partial_t u_h,\partial_t\gamma_h,z_h,\partial_tp_h)$, we can obtain
\begin{equation}
\begin{split}
&2\int_0^t \Big(\|\partial_t A^{\frac{1}{2}}(\sigma_h+\alpha p_h I)\|_0^2+c_0\|\partial_tp_h\|_0^2\Big)\;ds+\|K^{-\frac{1}{2}}z_h(t)\|_0^2\leq \epsilon_1\Big(\|p_h(t)\|_0^2+\int_0^t \|\partial_t u_h\|_0^2\Big)\\
&+\frac{1}{\epsilon_1}\Big(\|q(t)\|_0^2+\int_0^t\|\partial_t f\|_0^2\;ds\Big)+\int_0^t (\|p_h\|_0^2+\|\partial_t q\|_0^2)\;ds+\|K^{-\frac{1}{2}}z_h^0\|_0^2+\|p_h^0\|_0^2+\|q(0)\|_0^2.
\end{split}
\label{eq:zh}
\end{equation}
Then by inf-sup condition \eqref{eq:inf-sup} and \eqref{eq:semi-discrete1} differentiated in time, we have
\begin{align}
\|\partial_tu_h\|_h+\|\partial_t \gamma_h\|_0\leq C \|\partial_tA^{\frac{1}{2}}(\sigma_h+\alpha p_hI)\|_0.\label{eq:partituh}
\end{align}
Choosing $\epsilon_1$ in \eqref{eq:zh} small enough yields
\begin{align*}
&\int_0^t \Big(\|\partial_t A^{\frac{1}{2}}(\sigma_h+\alpha p_h I)\|_0^2+c_0\|\partial_tp_h\|_0^2\Big)\;ds+\|K^{-\frac{1}{2}}z_h(t)\|_0^2+\|p_h(t)\|_0^2\\
&\leq C\Big(\int_0^t (\|p_h\|_0^2+\|\partial_t q\|_0^2+\|\partial_t f\|_0^2)\;ds+\|q(t)\|_0^2+\|K^{-\frac{1}{2}}z_h^0\|_0^2+\|p_h^0\|_0^2+\|q(0)\|_0^2\Big),
\end{align*}
which coupling with \eqref{eq:semi1} leads to
\begin{equation*}
\begin{split}
&\|\sigma_h(t)\|_0^2+\|p_h(t)\|_0^2+\|u_h(t)\|_0^2+\|K^{-\frac{1}{2}}z_h(t)\|_0^2
+\|\gamma_h(t)\|_0^2\\
&\;+\int_0^t \Big(\|K^{-\frac{1}{2}}z_h\|_0^2+\|u_h\|_0^2+\|p_h\|_0^2+\|\sigma_h\|_0^2
+\|\gamma_h\|_0^2\Big)\;ds\leq
C\Big(\|f(t)\|_0^2+\int_0^t\Big(\|f\|_0^2+\|q\|_0^2\\
&\;+\|\partial_t f\|_0^2+\|\partial_t q\|_0^2\Big)\;ds
+\|q(t)\|_0^2+\|p_h^0\|_0^2
+\|u_h^0\|_0^2+\|f(0)\|_0^2+\|\sigma_h^0\|_0^2+\|K^{-\frac{1}{2}}z_h^0\|_0^2+\|q(0)\|_0^2\Big).
\end{split}
\end{equation*}
Since the discrete initial data $(\sigma_h^0,u_h^0, \gamma_h^0, z_h^0,p_h^0)$ is the elliptic projection of the initial data $(\sigma^0,u^0,\gamma^0, z^0,p^0)$, we have
\begin{align}
\|\sigma_h^0\|_0+\|u_h^0\|+\|\gamma_h^0\|_0+\|z_h^0\|_0+\|p_h^0\|_0\leq C (\|p^0\|_{H^1(\Omega)}+\|f(0)\|_0).\label{eq:initial}
\end{align}
Therefore, the proof is completed.

\end{proof}

Next, we will prove the convergence estimate. To simplify the notation, we let
\begin{align*}
\sigma-\sigma_h&=(\sigma-\Pi_h\sigma)+(\Pi_h\sigma-\sigma_h):=C_\sigma+D_\sigma,\\
u-u_h&=(u-I_hu)+(I_hu-u_h):=C_u+D_u\\
\gamma-\gamma_h&=(\gamma-\pi_h\gamma)+(\pi_h\gamma-\gamma_h):=C_\gamma+D_\gamma,\\
z-z_h&=(z-\Pi_h z)+(\Pi_h z-z_h):=C_z+D_z,\\
p-p_h&=(p-I_hp)+(I_hp-p_h):=C_p+D_p.
\end{align*}

\begin{theorem}
Let $(\sigma_h,u_h,\gamma_h,z_h,p_h)$ be the numerical solution of \eqref{eq:semi-discrete1}-\eqref{eq:semi-discrete5} and assume that the solution of \eqref{eq:weak1}-\eqref{eq:weak5} is sufficiently smooth, then we have the following error estimate
\begin{align*}
&\|\sigma-\sigma_h\|_{L^\infty(0,T;L^2(\Omega))}^2+\|p-p_h\|_{L^\infty(0,T;L^2(\Omega))}^2
+\|u-u_h\|_{L^\infty(0,T;L^2(\Omega))}^2+\|\gamma-\gamma_h\|_{L^\infty(0,T;L^2(\Omega))}^2\\
&\;+\|K^{-\frac{1}{2}}(z-z_h)\|_{L^\infty(0,T;L^2(\Omega))}^2
+\|\sigma-\sigma_h\|_{L^2(0,T;L^2(\Omega))}^2+\|p-p_h\|_{L^2(0,T;L^2(\Omega))}^2\\
&\;+\|u-u_h\|_{L^2(0,T;L^2(\Omega))}^2
+\|\sigma-\sigma_h\|_{L^2(0,T;L^2(\Omega))}^2
+\|K^{-\frac{1}{2}}(z-z_h)\|_{L^2(0,T;L^2(\Omega))}^2\\
&\leq C h^{2(k+1)}\Big(
\|u\|_{L^2(0,T;H^{k+1}(\Omega))}^2
+\|\gamma\|_{H^1(0,T;H^{k+1}(\Omega))}^2
+\|p\|_{H^1(0,T;H^{k+1}(\Omega))}^2+\|\sigma\|_{H^1(0,T;H^{k+1}(\Omega))}^2\\
&\;+\|K^{-\frac{1}{2}}z\|_{H^1(0,T;H^{k+1}(\Omega))}^2
+\|\gamma\|_{L^\infty(0,T;H^{k+1}(\Omega))}^2
+\|K^{-\frac{1}{2}}z\|_{L^\infty(0,T;H^{k+1}(\Omega))}^2+
\|\sigma\|_{L^\infty(0,T;H^{k+1}(\Omega))}^2\\
&\;+\|p\|_{L^\infty(0,T;H^{k+1}(\Omega))}^2
+\|u\|_{L^\infty(0,T;H^{k+1}(\Omega))}^2\Big).
\end{align*}

\end{theorem}

\begin{proof}

We obtain the following error equations by replacing $(\sigma,u,\gamma,z,p)$ by $(\sigma_h,u_h,\gamma_h,z_h,p_h)$ and performing integration by parts on \eqref{eq:semi-discrete1}-\eqref{eq:semi-discrete5}
\begin{align}
(A(D_\sigma+\alpha D_pI),\psi)- B_h^*(D_u,\psi)+(D_\gamma,as(\psi))&=-(A(C_\sigma+\alpha C_pI),\psi)\nonumber\\
&+ B_h^*(C_u,\psi)-(C_\gamma,as(\psi)),\label{eq:error1}\\
B_h(D_\sigma,v)&=-B_h(C_\sigma,v),\label{eq:error2}\\
(as(D_\sigma), \eta)&=-(as(C_\sigma), \eta),\label{eq:error3}\\
(K^{-1}D_z,\xi)+b_h^*(D_p,\xi)&=-(K^{-1}C_z,\xi)-b_h^*(C_p,\xi),\label{eq:error4}\\
c_0(\partial_tD_p,w)+\alpha(\partial_tA(D_\sigma+\alpha D_pI),wI)-b_h(D_z,w)&=-c_0(\partial_tC_p,w)\nonumber\\
&-\alpha(\partial_tA(C_\sigma+\alpha D_pI),wI)+b_h(C_z,w)\label{eq:error5}
\end{align}
for $(\psi,v,\eta,\xi,w)\in [\Sigma_h]^2\times [S_h]^2\times M_h\times \Sigma_h\times S_h$.
%
%

Differentiating \eqref{eq:error1} with time, taking $\psi=D_\sigma, v=\partial_tD_u,\eta=\partial_tD_\gamma,\xi=D_z,w=D_p$, and summing up the resulting equations yield
\begin{equation}
\begin{split}
&\frac{1}{2}\partial_t \Big(\|A^{\frac{1}{2}}(D_\sigma+\alpha D_pI)\|_0^2+c_0\|D_p\|_0^2\Big)+\|K^{-\frac{1}{2}}D_z\|_0^2
=-(\partial_tA(C_\sigma+\alpha C_pI),D_\sigma)\\
&\;+B_h^*(\partial_tC_u,D_\sigma)
-(\partial_tC_\gamma,as(D_\sigma))-B_h(C_\sigma,\partial_tD_u)
+(as(C_\sigma),\partial_t D_\gamma)
-(K^{-1}C_z,D_z)\\
&\;-b_h^*(C_p,D_z)
-c_0(\partial_tC_p,D_p)
-\alpha(\partial_tA(C_\sigma+\alpha C_pI),D_pI)+b_h(C_z,D_p).
\end{split}
\label{eq:error-sum}
\end{equation}
It follows from \eqref{eq:Pih} and \eqref{eq:Ih} that
\begin{align*}
B_h^*(\partial_tC_u,D_\sigma)=0,\quad -B_h(C_\sigma,\partial_tD_u)=0,
\quad b_h^*(C_p,D_z)=0,\quad b_h(C_z,D_p)=0.
\end{align*}
In addition, the Cauchy-Schwarz inequality yields
\begin{align*}
|-(\partial_tA(C_\sigma+\alpha C_pI),D_\sigma)|&\leq \|\partial_tA^{\frac{1}{2}}(C_\sigma+\alpha C_pI\|_0\|A^{\frac{1}{2}}D_\sigma\|_0,\\
|(\partial_tC_\gamma,as(D_\sigma))|&\leq \|\partial_tC_\gamma\|_0\|as(D_\sigma)\|_0,\\
|(K^{-1}C_z,D_z)|&\leq \|K^{-\frac{1}{2}}C_z\|_0\|K^{-\frac{1}{2}}D_z\|_0,\\
|c_0(\partial_tC_p,D_p)|&\leq c_0\|\partial_tC_p\|_0\|D_p\|_0,\\
|(\partial_tA(C_\sigma+\alpha C_pI),D_pI)|&\leq  \|\partial_tA^{\frac{1}{2}}(C_\sigma+\alpha C_pI)\|_0\|D_p\|_0.
\end{align*}
Therefore, we have from \eqref{eq:error-sum} that
\begin{align*}
&\frac{1}{2}\partial_t \Big(\|A^{\frac{1}{2}}(D_\sigma+D_pI)\|_0^2+c_0\|D_p\|_0^2\Big)+\|K^{-\frac{1}{2}}D_z\|_0^2\leq \|\partial_tA^{\frac{1}{2}}(C_\sigma+\alpha C_pI)\|_0\|A^{\frac{1}{2}}D_\sigma\|_0+\|\partial_tC_\gamma\|_0\|D_\sigma\|_0\\
&\;+(as(C_\sigma),\partial_t D_\gamma)+\|K^{-\frac{1}{2}}C_z\|_0\|K^{-\frac{1}{2}}D_z\|_0+c_0\|\partial_tC_p\|_0\|D_p\|_0
+\|\partial_tA^{\frac{1}{2}}(C_\sigma+\alpha C_pI)\|_0\|D_p\|_0.
\end{align*}
Integrating in time from $0$ to an arbitrary $t\in (0,T]$ results in
\begin{equation}
\begin{split}
&\frac{1}{2}\Big(\|A^{\frac{1}{2}}(D_\sigma+D_pI)(t)\|_0^2+c_0\|D_p(t)\|_0^2\Big)
+\int_0^t\|K^{-\frac{1}{2}}D_z\|_0^2\;ds\leq \frac{1}{2}\Big(\|A^{\frac{1}{2}}(D_\sigma+D_pI)(0)\|_0^2\\
&\;+c_0\|D_p(0)\|_0^2\Big)-\int_0^t(\partial_t as(C_\sigma),D_\gamma)\;ds+\frac{1}{\epsilon_2}\|C_\sigma(t)\|_0^2
+\epsilon_2\|D_\gamma(t)\|_0^2-(as(C_\sigma),D_\gamma)(0)\\
&\;+\epsilon_1\int_0^t \Big(\|D_p\|_0^2+\|K^{-\frac{1}{2}}D_z\|_0^2+\|D_\sigma\|_0^2\Big)\;ds
+\frac{1}{\epsilon_1}\int_0^t\Big(\|\partial_tC_\gamma\|_0^2
+\|K^{-\frac{1}{2}}C_z\|_0^2
+\|\partial_tC_p\|_0^2\\
&\;+\|\partial_tA^{\frac{1}{2}}(C_\sigma+\alpha C_pI)\|_0^2\Big)\;ds.
\end{split}
\label{eq:estimate1p}
\end{equation}
It follows from the inf-sup condition \eqref{eq:inf-sup} and \eqref{eq:error1} that
\begin{equation}
\begin{split}
\|D_u\|_h+\|D_\gamma\|_0&\leq C \sup_{\psi\in [\Sigma_h]^2}\frac{-B_h^*(D_u,\psi)+(as(\psi),D_\gamma)}{\|\psi\|_0}\\
&=C\sup_{\psi\in [\Sigma_h]^2}\frac{-(A(D_\sigma+\alpha D_pI),\psi)-(A(C_\sigma+\alpha C_pI),\psi)-(C_\gamma,as(\psi))}{\|\psi\|_0}\\
&\leq C\Big(\|A^{\frac{1}{2}}(D_\sigma+\alpha D_pI)\|_0+\|A^{\frac{1}{2}}(C_\sigma+\alpha C_pI)\|_0+\|C_\gamma\|_0\Big).
\end{split}
\label{eq:Du}
\end{equation}
Taking $\psi=D_\sigma, v=D_u, \eta=D_\gamma$ in \eqref{eq:error1}-\eqref{eq:error3}, and summing up the resulting equations yield
\begin{align*}
(A(D_\sigma+\alpha D_pI),D_\sigma)=-(A(C_\sigma+\alpha C_pI),D_\sigma)-(C_\gamma,as(D_\sigma))+(as(C_\sigma), D_\gamma),
\end{align*}
which together with Young's inequality implies
\begin{align*}
\|A^{\frac{1}{2}}D_\sigma\|_0^2\leq C\Big( \|D_p\|_0^2+\|C_\sigma\|_0^2+\|C_p\|_0^2+\|C_\gamma\|_0^2+
\frac{1}{\epsilon_0}\|C_\sigma\|_0^2+\epsilon_0\|D_\gamma\|_0^2\Big).
\end{align*}
Taking $\epsilon_0$ small enough and using \eqref{eq:Du}, we can obtain
\begin{align}
\|D_\sigma\|_0^2\leq C \Big(\|C_p\|_0^2+\|D_p\|_0^2+\|C_\sigma\|_0^2+\|C_\gamma\|_0^2\Big).\label{eq:Ds}
\end{align}
It follows from \eqref{eq:inf-sup-bh} and \eqref{eq:error4} that
\begin{align}
\|D_p\|_Z\leq C \sup_{\xi\in \Sigma_h}\frac{b_h^*(D_p,\xi)}{\|\xi\|_0}=C \sup_{\xi\in \Sigma_h}\frac{-(K^{-1}D_z,\xi)-(K^{-1}C_z,\xi)}{\|\xi\|_0}\leq C(\|K^{-1}D_z\|_0+\|K^{-1}C_z\|_0),\label{eq:Dph}
\end{align}
which together with the discrete Poincar\'{e}-Friedrichs inequalities implies
\begin{align*}
\int_0^t \|D_p\|_0\;ds\leq\int_0^t \|D_p\|_Z\;ds\leq \int_0^t(\|K^{-1}D_z\|_0+\|K^{-1}C_z\|_0)\;ds.
\end{align*}
An application of \eqref{eq:Ds} and \eqref{eq:Dph} leads to
\begin{align*}
\int_0^t \|D_\sigma\|_0\;ds
&\leq C\int_0^t\Big(\|K^{-1}D_z\|_0+\|K^{-1}C_z\|_0+\|C_\sigma\|_0+\|C_\gamma\|_0+\|C_p\|_0\Big)\;ds.
\end{align*}
Besides, we can infer from \eqref{eq:Du} and \eqref{eq:Ds} that
\begin{align*}
\|D_u\|_h+\|D_\gamma\|_0\leq C\Big(\|K^{-1}D_z\|_0+\|K^{-1}C_z\|_0+\|C_\sigma\|_0+\|C_\gamma\|_0+\|C_p\|_0\Big).
\end{align*}
Choosing $\epsilon_1$ and $\epsilon_2$ in \eqref{eq:estimate1p} small enough, we can infer from the preceding arguments %
\begin{equation}
\begin{split}
&\|A^{\frac{1}{2}}(D_\sigma+D_pI)(t)\|_0^2+c_0\|D_p(t)\|_0^2
+\int_0^t\Big(\|K^{-\frac{1}{2}}D_z\|_0^2+\|D_\sigma\|_0^2+\|D_p\|_0^2+\|D_\gamma\|_0^2\\
&\;+\|D_u\|_0^2\Big)\;ds
\leq \frac{1}{2}\Big(\|A^{\frac{1}{2}}(D_\sigma+D_pI)(0)\|_0^2+c_0\|D_p(0)\|_0^2\Big)
+\|C_\sigma(0)\|_0^2
+\|D_\gamma(0)\|_0^2+\|C_\sigma(t)\|_0^2
\\
&+\int_0^t \Big(\|C_\gamma\|_0^2+\|K^{-\frac{1}{2}}C_z\|_0^2
+\|\partial_tC_p\|_0^2
+\|C_\sigma\|_0^2+\|\partial_t C_\sigma\|_0^2+\|C_p\|_0^2+\|\partial_tC_\gamma\|_0^2\Big)\;ds.
\end{split}
\label{eq:bound1}
\end{equation}
Differentiating \eqref{eq:error1}-\eqref{eq:error4} in time and setting $(\psi,v,\eta,\xi,w)=(\partial_t D_\sigma,\partial_t D_u,\partial_t D_\gamma, D_z,\partial_t D_p)$, we can obtain
\begin{align*}
&\|\partial_tA^{\frac{1}{2}}(D_\sigma+\alpha D_p I)\|_0^2+c_0\|\partial_t D_p\|_0^2+\frac{1}{2}\partial_t\|K^{-\frac{1}{2}}D_z\|_0^2=-(\partial_tA(C_\sigma+\alpha C_pI),\partial_t D_\sigma+\alpha \partial_t D_pI)\\
&-(\partial_tC_\gamma,as(\partial_t D_\sigma))+(\partial_tas(C_\sigma), \partial_t D_\gamma)-(\partial_tK^{-1}C_z,D_z)-c_0(\partial_tC_p,\partial_t D_p).
\end{align*}
%
Integrating over time  and using $as(\partial_t D_\sigma)=as(\partial_t (D_\sigma+\alpha D_pI ))$ yield
\begin{equation}
\begin{split}
&\int_0^t \|\partial_tA^{\frac{1}{2}}(D_\sigma+\alpha D_p I)\|_0^2\;ds+c_0\int_0^t \|\partial_t D_p\|_0^2\;ds+\frac{1}{2}\|K^{-\frac{1}{2}}D_z(t)\|_0^2\\
&\leq \epsilon_1\int_0^t \|\partial_tA^{\frac{1}{2}}(D_\sigma+\alpha D_p I)\|_0^2\;ds +\frac{1}{4\epsilon_1}\Big(\int_0^t \|\partial_tA^{\frac{1}{2}}(C_\sigma+\alpha C_pI)\|_0^2\;ds+\int_0^t \|\partial_tC_\gamma\|_0^2\;ds\Big)\\
&\;+\frac{1}{4\epsilon_2}\int_0^t \|\partial_t C_\sigma\|_0^2\;ds+\epsilon_2\int_0^t\|\partial_tD_\gamma\|_0^2\;ds+\int_0^t \|K^{-\frac{1}{2}}\partial_t C_z\|_0^2\;ds+\int_0^t \|K^{-\frac{1}{2}}D_z\|_0^2\;ds\\
&\;+\frac{c_0}{2}\int_0^t\|\partial_t C_p\|_0^2\;ds+\frac{c_0}{2}\int_0^t \|\partial_t D_p\|_0^2\;ds+\frac{1}{2}\|K^{-\frac{1}{2}}D_z(0)\|_0^2.
\end{split}
\label{eq:estimate2}
\end{equation}
By inf-sup condition \eqref{eq:inf-sup} and \eqref{eq:error1} differentiated in time, we can obtain
\begin{align*}
\|\partial_tD_u\|_h+\|\partial_t D_\gamma\|_0\leq C\| \partial_tA^{\frac{1}{2}}(D_\sigma+\alpha D_p I)\|_0.
\end{align*}
Choosing $\epsilon_1$ and $\epsilon_2$ in \eqref{eq:estimate2} small enough leads to
\begin{align*}
&\int_0^t \|\partial_tA^{\frac{1}{2}}(D_\sigma+\alpha D_p I)\|_0^2\;ds+c_0\int_0^t \|\partial_t D_p\|_0^2\;ds+\|K^{-\frac{1}{2}}D_z(t)\|_0^2\\
&\;\leq
C\Big(\int_0^t \|\partial_tA^{\frac{1}{2}}(C_\sigma+\alpha C_pI)\|_0^2\;ds+\int_0^t \|\partial_tC_\gamma\|_0^2\;ds+\int_0^t \|\partial_t C_\sigma\|_0^2\;ds+\int_0^t \|K^{-\frac{1}{2}}\partial_t C_z\|_0^2\;ds\\
&\;+c_0\int_0^t\|\partial_t C_p\|_0^2\;ds+\int_0^t \|K^{-\frac{1}{2}}D_z\|_0^2\;ds+\|K^{-\frac{1}{2}}D_z(0)\|_0^2\Big),
\end{align*}
which can be combined with \eqref{eq:Du}, \eqref{eq:Ds}, \eqref{eq:Dph} and \eqref{eq:bound1} implies
\begin{equation*}
\begin{split}
&\|D_\sigma(t)\|_0^2+\|D_p(t)\|_0^2
+\|D_u(t)\|_0^2+\|D_\gamma(t)\|_0^2+\|K^{-\frac{1}{2}}D_z(t)\|_0^2\\
&\;+\int_0^t\Big(\|K^{-\frac{1}{2}}D_z\|_0^2+\|D_\sigma\|_0^2+\|D_p\|_0^2+\|D_\gamma\|_0^2
+\|D_u\|_0^2\Big)\;ds\\
&\leq C\Big(\|A^{\frac{1}{2}}(D_\sigma+D_pI)(0)\|_0^2+c_0\|D_p(0)\|_0^2
+\int_0^t \Big(\|C_\gamma\|_0^2+\|C_z\|_0^2+\|C_\sigma\|_0^2+\|C_p\|_0^2\\
&\;+\|\partial_tC_\gamma\|_0^2
+\|\partial_tC_p\|_0^2+\|\partial_tA^{\frac{1}{2}}(C_\sigma+\alpha C_pI)\|_0^2+\|\partial_t C_\sigma\|_0^2\Big)\;ds+\|K^{-\frac{1}{2}}D_z(0)\|_0^2\\
&\;+\|C_\sigma(0)\|_0^2+\|D_\gamma(0)\|_0^2
+\|K^{-\frac{1}{2}}\partial_t C_z(t)\|_0^2
+\|C_\sigma(t)\|_0^2+\|C_p(t)\|_0^2+\|C_\gamma(t)\|_0^2\Big).
\end{split}
\end{equation*}
The initial data satisfies
\begin{equation}
\begin{split}
&\|D_\sigma(0)\|_0+\|D_\gamma(0)\|_0+\|D_p(0)\|_0+\|K^{-\frac{1}{2}}D_z(0)\|_0\\
&\leq C\Big( \|C_\sigma(0)\|_0+\|C_\gamma(0)\|_0
+\|C_p(0)\|_0+\|K^{-\frac{1}{2}}C_z(0)\|_0\Big).
\end{split}
\label{eq:initial2}
\end{equation}
Therefore, the proof is completed by combining the preceding arguments and the interpolation error estimates \eqref{eq:interpolation1}-\eqref{eq:interpolation2}.

\end{proof}

\begin{remark}
{\em
(robustness of error estimates for nearly incompressible materials). We can observe from the above proof that we employ the coercivity of $A$, while the coercivity of $A$ on $L^2(\Omega)^{2\times 2}$ is not uniform in $\lambda$. Indeed,
\begin{align*}
c_\lambda\|\psi\|_0^2\leq \|A^{\frac{1}{2}}\psi\|_0^2\quad \forall \psi\in L^2(\Omega)^{2\times 2}
\end{align*}
holds with a constant $c_\lambda>0$ but $c_\lambda\rightarrow 0$ as $\lambda\rightarrow +\infty$. It means that our error estimates obtained using the coercivity of $A$ may have error bounds growing unboundedly as $\lambda\rightarrow +\infty$. In order to get a uniform bound, we can follow the framework given in \cite{LeeKim16,LinaParkelasticity20}, which is omitted for simplicity. In fact, proceeding analogously to \cite{LeeKim16,LinaParkelasticity20} we can check that our proposed method has uniformly bounded estimate.
}
\end{remark}

\subsection{Error analysis for fully discrete scheme}

In this subsection we analyze the convergence estimates for the fully discrete scheme. To this end we introduce a partition of the time interval $[0,T]$ into subintervals $[t_{n-1},t_n],1\leq n\leq N (N\;\mbox{is an integer})$ and denote the time step size by $\Delta t=\frac{T}{N}$. Using backward Euler scheme in time, we get the fully discrete staggered DG method as follows: Find $(\sigma_h,u_h,\gamma_h,z_h,p_h)\in [\Sigma_h]^2\times [S_h]^2\times M_h\times \Sigma_h\times S_h$ such that
\begin{align}
(A(\sigma_h^{n}+\alpha p_h^{n}I),\psi)-B_h^*(u_h^{n},\psi)+(\gamma_h^n,as(\psi))&=0,\label{eq:fully1}\\
B_h(\sigma_h^n,v)&=(f^n,v),\label{eq:fully2}\\
(as(\sigma_h^n), \eta)&=0,\label{eq:fully3}\\
(K^{-1}z_h^n,\xi)+b_h^*(p_h^n,\xi)&=0,\label{eq:fully4}\\
c_0(\frac{p_h^n-p_h^{n-1}}{\Delta t},w)+\alpha(\frac{A(\sigma_h^n+\alpha p_h^nI)-A(\sigma_h^{n-1}+\alpha p_h^{n-1}I)}{\Delta t},wI)-b_h(z_h^n,w)&=(q^n,w)\label{eq:fully5}
\end{align}
for $(\psi,v,\eta,\xi,w)\in [\Sigma_h]^2\times [S_h]^2\times M_h\times \Sigma_h\times S_h$.

For a Sobolev space $H$ on $\Omega$ with a norm $\|\cdot\|_H$, we define the discrete in time norms by
\begin{align*}
\|\psi\|_{l^2(0,T;H)}:=\Big(\sum_{n=1}^N\Delta t\|\psi\|_H^2\Big)^{1/2},\quad \|\psi\|_{l^{\infty}(0,T;H)}:=\max_{0\leq n\leq N}\|\psi\|_H
\end{align*}
and we let
\begin{align*}
\sigma^n-\sigma_h^n&=(\sigma^n-\Pi_h\sigma^n)+(\Pi_h\sigma^n-\sigma_h^n):=C_\sigma^n+D_\sigma^n,\\
u^n-u_h^n&=(u^n-I_hu^n)+(I_hu^n-u_h^n):=C_u^n+D_u^n\\
\gamma^n-\gamma_h^n&=(\gamma^n-\pi_h\gamma^n)+(\pi_h\gamma^n-\gamma_h^n):=C_\gamma^n+D_\gamma^n,\\
z^n-z_h^n&=(z^n-\Pi_h z^n)+(\Pi_h z^n-z_h^n):=C_z^n+D_z^n,\\
p^n-p_h^n&=(p^n-I_hp^n)+(I_hp^n-p_h^n):=C_p^n+D_p^n.
\end{align*}

\begin{theorem}(existence and uniqueness).
There exists a unique solution to the fully discrete scheme \eqref{eq:fully1}-\eqref{eq:fully5}.
%

\end{theorem}

\begin{proof}

Since \eqref{eq:fully1}-\eqref{eq:fully1} is a square linear system, uniqueness implies existence. It suffices to show the uniqueness. To this end, suppose $f^n,q^n$ and the $(n-1)$th time step solution vanishes. Then we have from \eqref{eq:fully1}-\eqref{eq:fully5}
\begin{align}
(A(\sigma_h^{n}+\alpha p_h^{n}I),\psi)-B_h^*(u_h^{n},\psi)+(\gamma_h^n,as(\psi))&=0\quad \forall \psi\in [\Sigma_h]^2,\label{eq:fully11}\\
B_h(\sigma_h^n,v)&=0\quad \forall v\in [S_h]^2,\label{eq:fully21}\\
(as(\sigma_h^n), \eta)&=0\quad \forall \eta\in M_h,\label{eq:fully31}\\
(K^{-1}z_h^n,\xi)+b_h^*(p_h^n,\xi)&=0\quad \forall \xi\in \Sigma_h,\label{eq:fully41}\\
c_0(p_h^n,w)+\alpha(A(\sigma_h^n+\alpha p_h^nI),wI)-\Delta tb_h(z_h^n,w)&=0\quad \forall w\in S_h\label{eq:fully51}.
\end{align}
Taking $\psi=\sigma_h^n, v=u_h^n, \eta=\gamma_h^n, \xi=z_h^n, w=p_h^n$ in \eqref{eq:fully11}-\eqref{eq:fully51} and summing up the resulting equations yield
\begin{align*}
\|A^{\frac{1}{2}}(\sigma_h^n+\alpha p_h^nI)\|_0^2+c_0\|p_h^n\|_0^2+\Delta t\|K^{-\frac{1}{2}}z_h^n\|_0^2=0,
\end{align*}
thereby $\sigma_h^n+\alpha p_h^nI=0$ and $z_h^n=0$.

We can infer from inf-sup condition \eqref{eq:inf-sup-bh} and \eqref{eq:fully41} that
\begin{align*}
\|p_h^n\|_Z\leq C\|K^{-1}z_h^n\|_0,
\end{align*}
which implies $p_h^n=0$. Hence $\sigma_h^n=0$ because of $\sigma_h^n+\alpha p_h^nI=0$.

It follows from \eqref{eq:inf-sup} and \eqref{eq:fully11} that
\begin{align*}
\|u_h^n\|_h+\|\gamma_h^n\|_0\leq C \|A^{\frac{1}{2}}(\sigma_h^n+\alpha p_h^nI)\|_0,
\end{align*}
which gives $u_h^n=0$ and $\gamma_h^n=0$. This completes the proof.

\end{proof}

\begin{theorem}
Let $(\sigma_h,u_h,\gamma_h,z_h,p_h)\in [\Sigma_h]^2\times [S_h]^2\times M_h\times \Sigma_h\times S_h$ be the numerical solution of the fully discrete scheme \eqref{eq:fully1}-\eqref{eq:fully5} and assume that the solution of \eqref{eq:weak1}-\eqref{eq:weak5} is sufficiently smooth, then we have the following convergence estimate
\begin{equation*}
\begin{split}
&\|K^{-\frac{1}{2}}(z-z_h)\|_{l^2(0,T;L^2(\Omega))}^2+\|\sigma-\sigma_h\|_{l^2(0,T;L^2(\Omega))}^2
+\|z-z_h\|_{l^2(0,T;L^2(\Omega))}^2+\|u-u_h\|_{l^2(0,T;L^2(\Omega))}^2\\
&\;+\|\sigma-\sigma_h\|_{l^2(0,T;L^2(\Omega))}^2
+\|p-p_h\|_{l^\infty(0,T;L^2(\Omega))}^2+\|K^{-\frac{1}{2}}(z-z_h)\|_{l^\infty(0,T;L^2(\Omega))}^2
+\|\sigma-\sigma_h\|_{l^\infty(0,T;L^2(\Omega))}^2\\
&\;+\|\gamma-\gamma_h\|_{l^\infty(0,T;L^2(\Omega))}^2
+\|u-u_h\|_{l^\infty(0,T;L^2(\Omega))}^2\leq C\Big(h^{2(k+1)}\Big(
\|p\|_{H^{1}(0,T;H^{k+1}(\Omega))}^2\\
&\;+\|\sigma\|_{H^{1}(0,T;H^{k+1}(\Omega))}^2
+\|K^{-\frac{1}{2}}z\|_{H^{1}(0,T;H^{k+1}(\Omega))}^2
+\|\gamma\|_{H^{1}(0,T;H^{k+1}(\Omega))}^2
+\|\sigma\|_{L^\infty(0,T;H^{k+1}(\Omega))}^2\\
&\;+\|p\|_{L^\infty(0,T;H^{k+1}(\Omega))}^2
+\|u\|_{L^2(0,T;H^{k+1}(\Omega))}^2+\|u\|_{L^\infty(0,T;H^{k+1}(\Omega))}^2
+\|\gamma\|_{L^\infty(0,T;H^{k+1}(\Omega))}^2\\
&\;+\|K^{-\frac{1}{2}}z\|_{L^\infty(0,T;H^{k+1}(\Omega))}^2\Big)
+\Delta t^2\Big(\|p\|_{H^2(0,T;L^2(\Omega))}^2
+\|\sigma\|_{H^2(0,T;L^2(\Omega))}^2\Big)\Big).
\end{split}
\end{equation*}

\end{theorem}

\begin{proof}

We have the following error equations by performing integration by parts on \eqref{eq:fully1}-\eqref{eq:fully5}
\begin{align}
&(A(\sigma^n-\sigma_h^{n}+\alpha (p^n-p_h^{n})I),\psi)-B_h^*(u^n-u_h^{n},\psi)+(\gamma^n-\gamma_h^n,as(\psi))=0,\label{error-full1}\\
&B_h(\sigma^n-\sigma_h^n,v)=0,\label{error-full2}\\
&(as(\sigma^n-\sigma_h^n), \eta)=0,\label{error-full3}\\
&(K^{-1}(z^n-z_h^n),\xi)+b_h^*(p^n-p_h^n,\xi)=0,\label{error-full4}\\
&\frac{\alpha}{\Delta t}(A(\sigma^n-\sigma_h^n+\alpha p^n-p_h^nI)-A(\sigma^{n-1}-\sigma_h^{n-1}+\alpha (p^{n-1}-p_h^{n-1}I)),wI)\nonumber\\
&\;+\frac{c_0}{\Delta t}(p^n-p_h^n-(p^{n-1}-p_h^{n-1}),w)-b_h(z^n-z_h^n,w)=c_0(\frac{p^n-p^{n-1}}{\Delta t}-p_{t}(:,t_n),w)\nonumber\\
&\;+\frac{\alpha}{\Delta t}(A(\sigma^n-\sigma^{n-1}+\alpha (p^n-p^{n-1})I)-\Delta tA(\sigma_{t}(:,t_n)+\alpha p_{t}(:,t_n)),wI)\label{error-full5}
\end{align}
for $(\psi,v,\eta,\xi,w)\in [\Sigma_h]^2\times [S_h]^2\times M_h\times \Sigma_h\times S_h$.

Consider \eqref{error-full1} at iterations $n$ and $n-1$ and taking the difference, then setting $\psi=D_\sigma^n, v=D_u^n, \eta=D_\gamma^n-D_\gamma^{n-1}, \xi=D_z^n, w=D_p^n$ in the difference equation and \eqref{error-full2}-\eqref{error-full5}, and summing up the resulting equations, we have
%
\begin{equation*}
\begin{split}
&(K^{-1}(z^n-z_h^n),D_z^n)+\frac{c_0}{\Delta t}(p^n-p_h^n-(p^{n-1}-p_h^{n-1}),D_p^n)\\
&\;+ \frac{1}{\Delta t}(A(\sigma^n-\sigma_h^n+\alpha (p^n-p_h^n)I)-A(\sigma^{n-1}-\sigma_h^{n-1}+\alpha (p^{n-1}-p_h^{n-1})I),D_\sigma^n+\alpha D_p^nI)\\
&=\frac{c_0}{\Delta t}(p^n-p^{n-1}-\Delta tp_{t}(:,t_n),D_p^n)+\frac{1}{\Delta t}(A(\sigma^n-\sigma^{n-1}+\alpha (p^n-p^{n-1})I)\\
&\;-\Delta tA(\sigma_{t}(:,t_n)+\alpha p_{t}(:,t_n)),\alpha D_p^nI)-\frac{1}{\Delta t}(C_\gamma^n-C_\gamma^{n-1},as(D_\sigma^n))+\frac{1}{\Delta t}(D_\gamma^n-D_\gamma^{n-1},as(C_\sigma^n)).
\end{split}
\end{equation*}
Using the inequality $(a-b,a)\geq \frac{|a|^2-|b|^2}{2}$ and the Cauchy-Schwarz inequality, we can obtain
\begin{equation}
\begin{split}
&\|K^{-\frac{1}{2}}D_z^n\|_0^2+\frac{c_0}{2\Delta t}\Big(\|D_p^n\|_0^2-\|D_p^{n-1}\|_0^2
+\|D_p^n-D_p^{n-1}\|_0^2\Big)\\
&\;+\frac{1}{2\Delta t}\Big(\|A^{\frac{1}{2}}(D_\sigma^n+\alpha D_p^n)\|_0^2-\|A^{\frac{1}{2}}(D_\sigma^{n-1}+\alpha D_p^{n-1}I)\|_0^2\Big)\\
&\leq C\Big(\|K^{-\frac{1}{2}}C_z^n\|_0\|K^{-\frac{1}{2}}D_z^n\|_0+\frac{c_0}{\Delta t}\|C_p^n-C_p^{n-1}\|\|D_p^n\|_0+\frac{1}{\Delta t}(D_\gamma^n-D_\gamma^{n-1},as(C_\sigma^n))\\
&\;+\frac{1}{\Delta t} \|A(C_\sigma^n+\alpha C_p^nI)-A(C_\sigma^{n-1}+\alpha C_p^{n-1}I)\|_0\Big(\|D_\sigma^n\|_0+\|D_p^n\|_0\Big)\\
&\;+\frac{1}{\Delta t}\|A(\sigma^n-\sigma^{n-1}+\alpha (p^n-p^{n-1})I)-\Delta tA(\sigma_{t}(:,t_n)+\alpha p_{t}(:,t_n)I)\|_0\|D_p^n\|_0\\
&\;+\frac{1}{\Delta t}\|C_\gamma^n-C_\gamma^{n-1}\|_0\|D_\sigma^n\|_0
+\frac{c_0}{\Delta t}\|p^n-p^{n-1}-\Delta tp_{t}(:,t_n)\|_0\|D_p^n\|_0\Big).
\end{split}
\label{eq:error}
\end{equation}
It follows from \eqref{eq:inf-sup-bh}, \eqref{error-full4} and the discrete Poincar\'{e}-Friedrichs inequality that
\begin{align}
\|D_p^n\|_0\leq C\|D_p^n\|_Z\leq C \sup_{\xi\in \Sigma_h}\frac{b_h^*(p_h^n-I_hp^n,\xi)}{\|\xi\|_0}\leq C \|K^{-1}(z^n-z_h^n)\|_0.\label{eq:Dpn}
\end{align}
Similarly, we can infer from \eqref{eq:inf-sup} and \eqref{error-full1} that
\begin{align}
\|D_u^n\|_h+\|D_\gamma^n\|_0\leq C ( \|A^{\frac{1}{2}}(\sigma^n-\sigma_h^{n}+\alpha (p^n-p_h^{n})I)\|_0+\|C_\gamma^n\|_0).\label{eq:uerror-full}
\end{align}
Besides, taking $(\psi,v,\eta)=(D_\sigma^n, D_u^n, D_\gamma^n)$ in \eqref{error-full1}-\eqref{error-full3} and summing up the resulting equations yield
\begin{align*}
(A(\sigma^n-\sigma_h^{n}+\alpha (p^n-p_h^{n})I),D_\sigma^n)+(C_\gamma^n,as(D_\sigma^n))-(as(C_\sigma^n),D_\gamma^n)=0,
\end{align*}
then an application of Young's inequality yields
\begin{equation}
\begin{split}
\|D_\sigma^n\|_0&\leq C\Big( \|C_p^n\|_0+\|D_p^n\|_0+\|C_\sigma^n\|_0+\|C_\gamma^n\|_0
+\epsilon_1\|D_\gamma^n\|_0+\frac{1}{\epsilon_1}\|C_\sigma^n\|_0\Big).
\end{split}
\end{equation}
Choosing $\epsilon_1$ small enough and combining with \eqref{eq:Dpn}, \eqref{eq:uerror-full} give
\begin{equation}
\begin{split}
\|D_\sigma^n\|_0&\leq C\Big( \|C_p^n\|_0+\|D_p^n\|_0+\|C_\sigma^n\|_0+\|C_\gamma^n\|_0\Big)\\
&\leq  C\Big(\|p^n-I_hp^n\|_0+ \|K^{-1}(z^n-z_h^n)\|_0+\|\Pi_h\sigma^n-\sigma^n\|_0+\|\gamma^n-\pi_h\gamma^n\|_0\Big).
\end{split}
\label{eq:errorsigma}
\end{equation}
%
On the other hand, we have from Taylor's expansion
\begin{align*}
p^n-p^{n-1}-\Delta tp_{t}(:,t_n)=-\int_{t_{n-1}}^{t_n}(t-t_{n-1})p_{tt}(:,t_n)\;ds,
\end{align*}
where the right hand side can be estimated by the Cauchy-Schwarz inequality
\begin{align*}
\|\frac{1}{\Delta t}\int_{t_{n-1}}^{t_n}(t-t_{n-1})p_{tt}(:,t_n)\;ds\|_0^2\leq \frac{\Delta t}{3}\int_{t_{n-1}}^{t_n}\|p_{tt}\|_0^2\;ds.
\end{align*}
Similarly, we have
\begin{align*}
\|\frac{1}{\Delta t}A(\sigma^n-\sigma^{n-1}-\Delta t \sigma_t(:,t_n))\|_0^2&\leq \frac{C\Delta t}{3}\int_{t_{n-1}}^{t_n}\|\sigma_{tt}\|_0^2\;ds,\\
\|\frac{1}{\Delta t}A((p^n-p^{n-1})I-\Delta t p_t(:,t_n)I)\|_0^2&\leq \frac{C\Delta t}{3}\int_{t_{n-1}}^{t_n}\|p_{tt}\|_0^2\;ds.
\end{align*}
Therefore, we can infer from \eqref{eq:error} and Young's inequalities that
\begin{equation}
\begin{split}
&\|K^{-\frac{1}{2}}D_z^n\|_0^2+\frac{c_0}{2\Delta t}\Big(\|D_p^n\|_0^2-\|D_p^{n-1}\|_0^2
+\|D_p^n-D_p^{n-1}\|_0^2\Big)+\frac{1}{2\Delta t}\Big(\|A^{\frac{1}{2}}(D_\sigma^n+\alpha D_p^nI)\|_0^2\\
&\;-\|A^{\frac{1}{2}}(D_\sigma^{n-1}+\alpha D_p^{n-1}I)\|_0^2\Big)\leq C\Big(\|K^{-\frac{1}{2}}C_z^n\|_0^2+\|C_\sigma^n\|_0^2+\|C_\gamma^n\|_0^2+\|C_p^n\|_0^2\\
&\;+\frac{1}{(\Delta t)^2}(\|C_p^n-C_p^{n-1}\|_0^2+\|A^{\frac{1}{2}}(C_\sigma^n-C_\sigma^{n-1})\|_0^2
+\|C_\gamma^n-C_\gamma^{n-1}\|_0^2)\\
&\;+\Delta t\int_{t_{n-1}}^{t_n}\Big(\|p_{tt}\|_0^2+\|\sigma_{tt}\|_0^2\Big)\;ds+\frac{1}{\Delta t}(D_\gamma^n-D_\gamma^{n-1},as(C_\sigma^n))\Big).
\end{split}
\end{equation}
Changing $n$ to $j$ in \eqref{eq:error} and summing over $j=1,\cdots,n$, we can obtain
\begin{equation}
\begin{split}
&2\Delta t\sum_{j=1}^n\|K^{-\frac{1}{2}}D_z^j\|_0^2+c_0\Big(\|D_p^n\|_0^2-\|D_p^{0}\|_0^2
+\sum_{j=1}^n\|D_p^n-D_p^{n-1}\|_0^2\Big)+\|A^{\frac{1}{2}}(D_\sigma^n+\alpha D_p^nI)\|_0^2\\
&\leq C\Big(\sum_{j=1}^n\Delta t(\|K^{-\frac{1}{2}}C_z^j\|_0^2+\|C_\sigma^j\|_0^2
+\|C_\gamma^j\|_0^2+\|C_p^j\|_0^2)
+\sum_{j=1}^n\frac{1}{\Delta t}\Big(\|C_p^j-C_p^{j-1}\|_0^2\\
&\;+\|C_\sigma^j-C_\sigma^{j-1}\|_0^2
+\|C_\gamma^{j}-C_\gamma^{j-1}\|_0^2\Big)+\Delta t^2\sum_{j=1}^n\int_{t_{j-1}}^{t_j}(\|p_{tt}\|_0^2+\|\sigma_{tt}\|_0^2)\;ds\\
&\;+\|A^{\frac{1}{2}}(D_\sigma^{0}+\alpha D_p^{0}I)\|_0^2
+\sum_{j=1}^{n}(D_\gamma^j-D_\gamma^{j-1},as(C_\sigma^j))\Big).
\end{split}
\end{equation}
It follows from the interpolation error estimates \eqref{eq:interpolation1}-\eqref{eq:interpolation2} that
\begin{align*}
\|C_z^j\|_0&\leq C h^{k+1} \|z^j\|_{H^{k+1}(\Omega)},\quad
\|C_\gamma^j\|_0\leq C h^{k+1}\|\gamma^j\|_{H^{k+1}(\Omega)},\\
\|C_\sigma^j\|_0&\leq C h^{k+1}\|\sigma^j\|_{H^{k+1}(\Omega)},\quad
\|C_p^j\|_0\leq C h^{k+1}\|p^j\|_{H^{k+1}(\Omega)}.
\end{align*}
The Cauchy-Schwarz inequality and the interpolation error estimates \eqref{eq:interpolation1} imply
\begin{equation}
\begin{split}
\frac{1}{\Delta t}\|C_p^j-C_p^{j-1}\|_0^2&=\frac{1}{\Delta t}\|(p^j-I_hp^j)-(p^{j-1}-I_h p^{j-1})\|_0^2=\frac{1}{\Delta t}\|\int_{t_{j-1}}^{t_j}(p_t-I_h p_t)\|_0^2\\
&\leq C h^{2(k+1)}\int_{t_{j-1}}^{t_j}\|p_t\|_{H^{k+1}(\Omega)}^2,\\
\frac{1}{\Delta t}\|C_\gamma^j-C_\gamma^{j-1}\|_0^2
&\leq C h^{2(k+1)}\int_{t_{j-1}}^{t_j}\|\gamma_t\|_{H^{k+1}(\Omega)}^2.
\end{split}
\label{eq:Cpj}
\end{equation}
Discrete integration by parts in time yields
\begin{align*}
\sum_{j=1}^n(D_\gamma^j-D_\gamma^{j-1},as(C_\sigma^j))=D_\gamma^nC_\sigma^n-D_\gamma^0C_\sigma^0
-\sum_{j=1}^n(D_\gamma^{j-1},C_\sigma^j-C_\sigma^{j-1}),
\end{align*}
where the last term can be estimated by
\begin{align*}
\|C_\sigma^j-C_\sigma^{j-1}\|_0^2
\leq C \Delta t h^{2(k+1)}\int_{t_{j-1}}^{t_j}\|\sigma_t\|_{H^{k+1}(\Omega)}^2.
\end{align*}
Thus, the preceding arguments lead to
\begin{equation*}
\begin{split}
&\Delta t\sum_{j=1}^n\|K^{-\frac{1}{2}}D_z^j\|_0^2+c_0\Big(\|D_p^n\|_0^2-\|D_p^{0}\|_0^2
+\sum_{j=1}^n\|D_p^n-D_p^{n-1}\|_0^2\Big)+\|A^{\frac{1}{2}}(D_\sigma^n+\alpha D_p^nI)\|_0^2\\
&\leq C\Big(
h^{2(k+1)}(\|\gamma\|_{H^{1}(0,T;H^{k+1}(\Omega))}^2
+\|p\|_{H^{1}(0,T;H^{k+1}(\Omega))}^2+\|z\|_{L^{2}(0,T;H^{k+1}(\Omega))}
+\|\sigma\|_{H^{1}(0,T;H^{k+1}(\Omega))}^2\\
&\;+\|\sigma\|_{L^{\infty}(0,T;H^{k+1}(\Omega))}^2)+\Delta t^2\sum_{j=1}^n\int_{t_{j-1}}^{t_j}(\|p_{tt}\|_0^2+\|\sigma_{tt}\|_0^2)\;ds
+\|A^{\frac{1}{2}}(D_\sigma^{0}+\alpha D_p^{0}I)\|_0^2+\|C_\sigma^0\|_0^2+\|D_\gamma^0\|_0^2\Big).
\end{split}
\end{equation*}
Considering \eqref{error-full1}-\eqref{error-full4} at iterations $n$ and $n-1$, and taking the difference, setting $\psi=D_\sigma^n-D_\sigma^{n-1}, v=D_u^n-D_u^{n-1}, \eta=D_\gamma^n-D_\gamma^{n-1},\xi=D_z^n,w=D_p^n-D_p^{n-1}$ and summing up the resulting equations yield
\begin{align*}
&\frac{1}{\Delta t}\|A^{\frac{1}{2}}(D_\sigma^n+\alpha D_p^nI)-A^{\frac{1}{2}}(D_\sigma^{n-1}+\alpha D_p^{n-1}I)\|_0^2\\
&\;+(K^{-1}(D_z^n-D_z^{n-1}),
D_z^n)+\frac{c_0}{\Delta t}(p^n-p_h^n-(p^{n-1}-p_h^{n-1}),D_p^n-D_p^{n-1})\\
&=c_0(\frac{p^n-p^{n-1}}{\Delta t}-p_{t}(:,t_n),D_p^n-D_p^{n-1})-(K^{-1}(C_z^n-C_z^{n-1}),D_z^n)\\
&\;+\frac{1}{\Delta t}(A(\sigma^n-\sigma^{n-1}+\alpha (p^n-p^{n-1})I)-\Delta tA(\sigma_{t}(:,t_n)+\alpha p_{t}(:,t_n)I),\alpha(D_p^n-D_p^{n-1})I)\\
&\;-\frac{1}{\Delta t}(A(C_\sigma^n+\alpha C_p^nI)-A(C_\sigma^{n-1}+\alpha C_p^{n-1}I),D_\sigma^n+\alpha D_p^nI-(D_\sigma^{n-1}+\alpha D_p^{n-1}I))\\
&\;-\frac{1}{\Delta t}(C_\gamma^n-C_\gamma^{n-1}),as(D_\sigma^n-
D_\sigma^{n-1}))+\frac{1}{\Delta t}(as(C_\sigma^n-C_\sigma^{n-1}),D_\gamma^n
-D_\gamma^{n-1}).
\end{align*}
We can infer from inf-sup condition \eqref{eq:inf-sup-bh} and \eqref{eq:inf-sup} that
\begin{align*}
\|I_hp^n-p_h^n-(I_hp^{n-1}-p_h^{n-1})\|_0&\leq C\|K^{-1}(z^n-z_h^n-(z^{n-1}-z_h^{n-1}))\|_0,\\
\|\pi_h\gamma^n-\gamma_h^n-(\pi_h\gamma^{n-1}-\gamma_h^{n-1})\|_0&\leq C \|A^{\frac{1}{2}}(D_\sigma^n+\alpha D_p^nI)-A^{\frac{1}{2}}(D_\sigma^{n-1}+\alpha D_p^{n-1}I)\|_0.
\end{align*}
%
Therefore, we have
\begin{align*}
&\frac{1}{\Delta t}\|A^{\frac{1}{2}}(D_\sigma^n+\alpha D_p^nI)-A^{\frac{1}{2}}(D_\sigma^{n-1}+\alpha D_p^{n-1}I)\|_0^2+\frac{1}{2}\Big(\|K^{-\frac{1}{2}}D_z^n\|_0^2\\
&\;-\|K^{-\frac{1}{2}}D_z^{n-1}\|_0^2+
\|K^{-\frac{1}{2}}D_z^n-K^{-\frac{1}{2}}D_z^{n-1}\|_0^2\Big)+\frac{c_0}{\Delta t}\|D_p^n-D_p^{n-1}\|_0^2\\
&\leq C\Big( \frac{\Delta t}{3}\int_{t_{n-1}}^{t_n}(\|p_{tt}\|_0^2+\|\sigma_{tt}\|_0^2)\;ds+\frac{1}{\Delta t}\Big(\|C_\sigma^n-C_\sigma^{n-1}\|_0^2\\
&\;+\|C_p^n-C_p^{n-1}\|_0^2
+\|C_\gamma^n-C_\gamma^{n-1}\|_0^2+\|K^{-\frac{1}{2}}(C_z^n-C_z^{n-1})\|_0^2\Big)+\Delta t\|D_z^n\|_0^2\Big).
\end{align*}
Changing $n$ to $j$ and summing over $j=1,\cdots,n$, we can obtain
\begin{align*}
&\sum_{j=1}^n\frac{1}{\Delta t}\|A^{\frac{1}{2}}(D_\sigma^j+\alpha D_p^jI)-A^{\frac{1}{2}}(D_\sigma^{j-1}+\alpha D_p^{j-1}I)\|_0^2+\|K^{-\frac{1}{2}}D_z^n\|_0^2+\sum_{j=1}^n\frac{c_0}{\Delta t}\|D_p^j-D_p^{j-1}\|_0^2\\
&\leq C\Big(\sum_{j=1}^n\frac{1}{\Delta t}\Big(\|C_\sigma^j-C_\sigma^{j-1}\|_0^2+\|C_p^j-C_p^{j-1}\|_0^2
+\|C_\gamma^j-C_\gamma^{j-1}\|_0^2+\|K^{-\frac{1}{2}}(C_z^j-C_z^{j-1})\|_0^2\Big)\\
&+\|K^{-\frac{1}{2}}(I_hz^{0}-z_h^{0})\|_0^2+ \frac{\Delta t}{3}\int_{0}^{t_n}(\|p_{tt}\|_0^2+\|\sigma_{tt}\|_0^2)\;ds+\sum_{j=1}^n\Delta t\|K^{-\frac{1}{2}}D_z^j\|_0^2\Big).
\end{align*}
Proceeding analogously to \eqref{eq:Cpj}, we can get
\begin{align*}
\frac{1}{\Delta t}\|C_\sigma^j-C_\sigma^{j-1}\|_0^2&\leq C h^{2(k+1)}\int_{t_{j-1}}^{t_j}\|\sigma_t\|_{H^{k+1}(\Omega)}^2,\\
\frac{1}{\Delta t}\|C_z^j-C_z^{j-1}\|_0^2&\leq C h^{2(k+1)}\int_{t_{j-1}}^{t_j}\|K^{-\frac{1}{2}}z_t\|_{H^{k+1}(\Omega)}^2.
\end{align*}
Combining the above estimates with the interpolation error estimates \eqref{eq:interpolation1}-\eqref{eq:interpolation2} completes the proof.

\end{proof}

\section{Fixed stress splitting scheme}\label{sec:fixed}
The global system \eqref{eq:fully1}-\eqref{eq:fully5} consists of five variables which is relatively large. In order to reduce the computational costs, we propose the following fixed stress splitting scheme inspired by \cite{Mikelic13}. This includes two steps: First, the flow problem is solved independently. Second, the mechanics problem is solved using updated pressure and flux. For
fixed $n, i \in  \mathbb{N}$, the detailed splitting scheme reads as follows:
%
%

\textbf{Step 1}: Given $(z_h^{n,i-1},p_h^{n,i-1},\sigma_h^{n,i-1})$, find $(z_h^{n,i},p_h^{n,i})$ such that
\begin{align}
&(K^{-1}z_h^{n,i},\xi)+b_h^*(p_h^{n,i},\xi)=0\quad \forall \xi\in \Sigma_h, \label{eq:split1}\\
&((c_0+\beta)\frac{p_h^{n,i}}{\Delta t},w)+\alpha(\frac{A(\alpha p_h^{n,i}I)}{\Delta t},wI)-b_h(z_h^{n,i},w)=(q^n,w)+(c_0\frac{p_h^{n-1}}{\Delta t},w) \nonumber\\
&+(\beta\frac{p_h^{n,i-1}}{\Delta t},w)\;-\alpha(\frac{A(\sigma_h^{n,i-1})-A(\sigma_h^{n-1}+\alpha p_h^{n-1}I)}{\Delta t},wI)\quad \forall w\in S_h.\label{eq:split2}
\end{align}

\textbf{Step 2}: Given $p_h^{n,i}$, find $(\sigma_h^{n,i}, u_h^{n,i},\gamma_h^{n,i})$ such that
\begin{align}
(A(\sigma_h^{n,i}),\psi)-B_h^*(u_h^{n,i},\psi)+(\gamma_h^{n,i},as(\psi))&=-(A(\alpha p_h^{n,i}I),\psi)\quad \forall \psi \in [\Sigma_h]^2,\label{eq:sigmas}\\
B_h(\sigma_h^{n,i},v)&=(f^n,v)\quad \forall v\in [S_h]^2,\label{eq:sigmas2}\\
(as(\sigma_h^{n,i}), \eta)&=0\quad \forall \eta\in M_h.\label{eq:sigmas3}
\end{align}
The initial guess for the iterations is chosen to be the solution at the last time step, i.e., $(\sigma_h^{n,0},p_h^{n,0})=(\sigma_h^{n-1},p_h^{n-1})$.

\begin{remark}
{\rm Since the mass matrix is block diagonal, we can apply local elimination for \eqref{eq:split1}-\eqref{eq:split2} and \eqref{eq:sigmas}-\eqref{eq:sigmas3}, respectively. In this way, we can get a reduced system which solely depends on the displacement and pressure.}

\end{remark}

\begin{theorem}(linear convergence of fixed stress splitting).
Let $(\sigma_h^n,u_h^n,\gamma_h^n,z_h^n,p_h^n)$ be the solution of \eqref{eq:fully1}-\eqref{eq:fully5} and let $(\sigma_h^{n,i}, u_h^{n,i},\gamma_h^{n,i},z_h^{n,i},p_h^{n,i})$ be the solution of \eqref{eq:split1}-\eqref{eq:sigmas3}. Then for all $\beta$ satisfying
\begin{align*}
\beta\geq \frac{\alpha^2}{2(\mu+\lambda)}.
\end{align*}
We have
\begin{align}
\|p_h^n-p_h^{n,i}\|_0^2\leq \frac{\frac{\beta}{2}}{c_0+\frac{\beta}{2}+\frac{\alpha^2}{\mu+\lambda}+\Delta tC_{p}^{-1} C_{inf}^{-1}K_1^{\frac{1}{2}}}
\|p_h^n-p_h^{n,i-1}\|_0^2,\quad \forall i\geq 1,\label{eq:phn-phni}
\end{align}
where $C_p$ is the Poincar\'{e} constant and $C_{inf}$ is the inf-sup constant in \eqref{eq:inf-sup-bh}.
\end{theorem}

\begin{proof}

Let $e_u^i=-u_h^{n,i}+u_h^n$, $e_\sigma^i=-\sigma_h^{n,i}+\sigma_h^n$, $e_\gamma^i=-\gamma_h^{n,i}+\gamma_h^n$, $e_p^i=-p_h^{n,i}+p_h^n$ and $e_z^i = -z_h^{n,i}+z_h^n$. Subtracting \eqref{eq:split1}-\eqref{eq:sigmas3} from \eqref{eq:fully1}-\eqref{eq:fully5}, we obtain
\begin{align*}
(A(e_\sigma^{i}+\alpha e_p^iI),\psi)-B_h^*(e_u^i,\psi)+(e_\gamma^i,as(\psi))&=0,\\
B_h(e_{\sigma}^i,v)&=0,\\
as(e_\sigma^i,\eta)&=0,\\
(K^{-1}e_z^i,\xi)+b_h^*(e_p^i,\xi)&=0,\\
\frac{c_0}{\Delta t}(e_p^i,w)-\frac{\beta}{\Delta t}(p_h^{n,i}-p_h^{n,i-1},w)+\frac{\alpha}{\Delta t}(A(e_\sigma^{i-1}+\alpha e_p^iI),wI)-b_h(e_z^i,w)&=0
\end{align*}
for $(\psi,v,\eta,\xi,w)\in [\Sigma_h]^2\times[S_h]^2\times M_h\times \Sigma_h\times S_h$.

Taking $\psi=e_\sigma^{i-1},v=e_u^{i},\eta=e_\gamma^i$, $\xi=e_z^i,w=e_p^i$ in the above equations, and summing up the resulting equations, we can infer that
%
\begin{align}
c_0\|e_p^i\|_0^2+\beta (e_p^i-e_p^{i-1},e_p^i)+\alpha^2(A(e_p^i I), e_p^i I)+(A(e_\sigma^{i-1}),e_\sigma^i)+\Delta t \|K^{-\frac{1}{2}} e_z^i\|_0^2=0.\label{eq:ep}
\end{align}

Consider \eqref{eq:sigmas}-\eqref{eq:sigmas2} at iterations $i$ and $i-1$ and take the difference, setting $\psi=\sigma_h^{n,i}-\sigma_h^{n,i-1}$ and summing up the resulting equations yield
\begin{align*}
(A(\sigma_h^{n,i}-\sigma_h^{n,i-1}),\sigma_h^{n,i}-\sigma_h^{n,i-1})=-\alpha (A((p_h^{n,i}-p_h^{n,i-1})I),\sigma_h^{n,i}-\sigma_h^{n,i-1}),
\end{align*}
which leads to
\begin{align*}
\|A^{\frac{1}{2}}(\sigma_h^{n,i}-\sigma_h^{n,i-1})\|_0^2&\leq \alpha \|A^{\frac{1}{2}}((p_h^{n,i}-p_h^{n,i-1})I)\|_0
\|A^{\frac{1}{2}}(\sigma_h^{n,i}-\sigma_h^{n,i-1})\|_0\\
&\leq \Big(\frac{\alpha^2}{2}\|A^{\frac{1}{2}}((p_h^{n,i}-p_h^{n,i-1})I)\|_0^2
+\frac{1}{2}\|A^{\frac{1}{2}}(\sigma_h^{n,i}-\sigma_h^{n,i-1})\|_0^2\Big).
\end{align*}
Thus, we have
\begin{equation}
\begin{split}
\|A^{1/2}(e_\sigma^i-e_\sigma^{i-1})\|_0&\leq \alpha\|A^{1/2}((p_h^{n,i}-p_h^{n,i-1})I)\|_0\leq \alpha\|A^{1/2}((e_p^i-e_p^{i-1})I)\|_0\\
&\leq \frac{\alpha}{(\lambda+\mu)^{1/2}}\|e_p^i-e_p^{i-1}\|_0.
\end{split}
\label{eq:sigmape}
\end{equation}
We can infer from \eqref{eq:ep}, the equality $(a-b,a)=\frac{a^2-b^2+(a-b)^2}{2}$ and the identity $
(A(e_\sigma^{i-1}),e_\sigma^i)=\frac{1}{4}\Big(\|A^{\frac{1}{2}}(e_\sigma^i+e_\sigma^{i-1})\|_0^2
-\|A^{\frac{1}{2}}(e_\sigma^i-e_\sigma^{i-1})\|_0^2\Big)$ that
%
\begin{align*}
&c_0\|e_p^i\|_0^2+\frac{\beta}{2} (\|e_p^i\|_0^2-\|e_p^{i-1}\|_0^2+\|e_p^i-e_p^{i-1}\|_0^2)+\alpha^2(A(e_p^i I), e_p^i I)\\
&\;+\frac{1}{4}\Big(\|A^{\frac{1}{2}}(e_\sigma^i+e_\sigma^{i-1})\|_0^2
-\|A^{\frac{1}{2}}(e_\sigma^i-e_\sigma^{i-1})\|_0^2\Big)+\Delta t \|K^{-\frac{1}{2}} e_z^i\|_0^2=0.
\end{align*}
Therefore, an appeal to \eqref{eq:sigmape} yields
\begin{align*}
&c_0\|e_p^i\|_0^2+\frac{\beta}{2} (\|e_p^i\|_0^2-\|e_p^{i-1}\|_0^2)
+(\frac{\beta}{2}-\frac{\alpha^2}{4(\lambda+\mu)})\|e_p^i-e_p^{i-1}\|_0^2 +\alpha^2(A(e_p^i I), e_p^i I)\\
&\;+\frac{1}{4}\|A^{\frac{1}{2}}(e_\sigma^i+e_\sigma^{i-1})\|_0^2
+\Delta t \|K^{-\frac{1}{2}} e_z^i\|_0^2\leq 0.
\end{align*}
It follows from the inf-sup condition \eqref{eq:inf-sup-bh} and
the discrete Poincar\'{e}-Friedrichs inequality that
\begin{align*}
\|e_p^i\|_0\leq C_{p} C_{inf} \|K^{-1}e_z^i\|_0,
\end{align*}
where $C_p$ is the Poincar\'{e} constant and $C_{inf}$ is the
inf-sup constant in \eqref{eq:inf-sup-bh}.
Thereby, if we require 
\begin{align*}
\beta\geq \frac{\alpha^2}{2(\lambda+\mu)},
\end{align*}
then there holds
\begin{equation}
\begin{split}
&c_0\|e_p^i\|_0^2+\frac{\beta}{2} (\|e_p^i\|_0^2-\|e_p^{i-1}\|_0^2)+\alpha^2(A(e_p^i I), e_p^i I)\\
&\;+\Delta tC_{p}^{-1}
C_{inf}^{-1}K_1^{\frac{1}{2}}\|e_p^i\|_0^2+\frac{1}{4}\|A^{\frac{1}{2}}(e_\sigma^i+e_\sigma^{i-1})\|_0^2
\leq 0,
\end{split}
\label{eq:final}
\end{equation}
which leads to the desired estimate \eqref{eq:phn-phni}.

%

\end{proof}

\section{Numerical experiments}\label{sec:numerical}
In this section we present several numerical experiments to verify the theoretical convergence rates and illustrate the behavior of the proposed method. We also present one benchmark example showing the locking-free property of the method. In the simulation presented below, we use the polynomial order $k=1$.

\subsection{Smooth solution test}\label{ex1}
In this example, we test the convergence of our method. To this end, we set $\Omega=(0,1)^2$, where the displacement and pressure are given by
\begin{align*}
u=
\left(
  \begin{array}{c}
   \sin(2\pi t) \sin(2\pi x)\sin(2\pi y)\\
    x\cos(t) \\
  \end{array}
\right),\quad p = \mbox{exp}(t)(\sin(\pi x)\cos(\pi y)+10).
\end{align*}
The corresponding $f$ and $q$ can be calculated by \eqref{eq:constitutive}-\eqref{eq:model}. We set $\alpha=1,\lambda=1,\mu=1$ and the simulation time is $T=0.01$ with time steps $\Delta t=h^2$. We remark that $\Delta t$ should be chosen small enough so that the time discretization error will not influence the convergence rates. For this example, we consider two values of $K$, i.e., $K=I$ and
\begin{align}
K=\left(
    \begin{array}{cc}
      \frac{1}{50} & 0 \\
      0 & 1 \\
    \end{array}
  \right)\label{eq:K}
\end{align}

We exploit square meshes and trapezoidal meshes for this example, see Figure~\ref{ex1-mesh} for an illustration. The convergence history against the number of degrees of freedom for rectangular meshes and trapezoidal meshes with various values of $c_0$ is depicted in Figure~\ref{ex1-con}. Optimal convergence rates matching the theoretical results can be obtained for various values of $c_0$. In addition, with different values of $c_0$, the accuracy for all the errors remain almost the same and the method is still valid for $c_0=0$. Further, the accuracy for $L^2$ errors of $\sigma,u,\gamma,z,p$ on rectangular meshes and trapezoidal meshes is slightly different. Figure~\ref{ex1-con-aniso} shows the convergence history on rectangular meshes for anisotropic $K$ defined in \eqref{eq:K} and similar performances can be observed, which illustrates the robustness of our method to the anisotropy ratio.

\begin{figure}[t]
\centering
\includegraphics[width=0.35\textwidth]{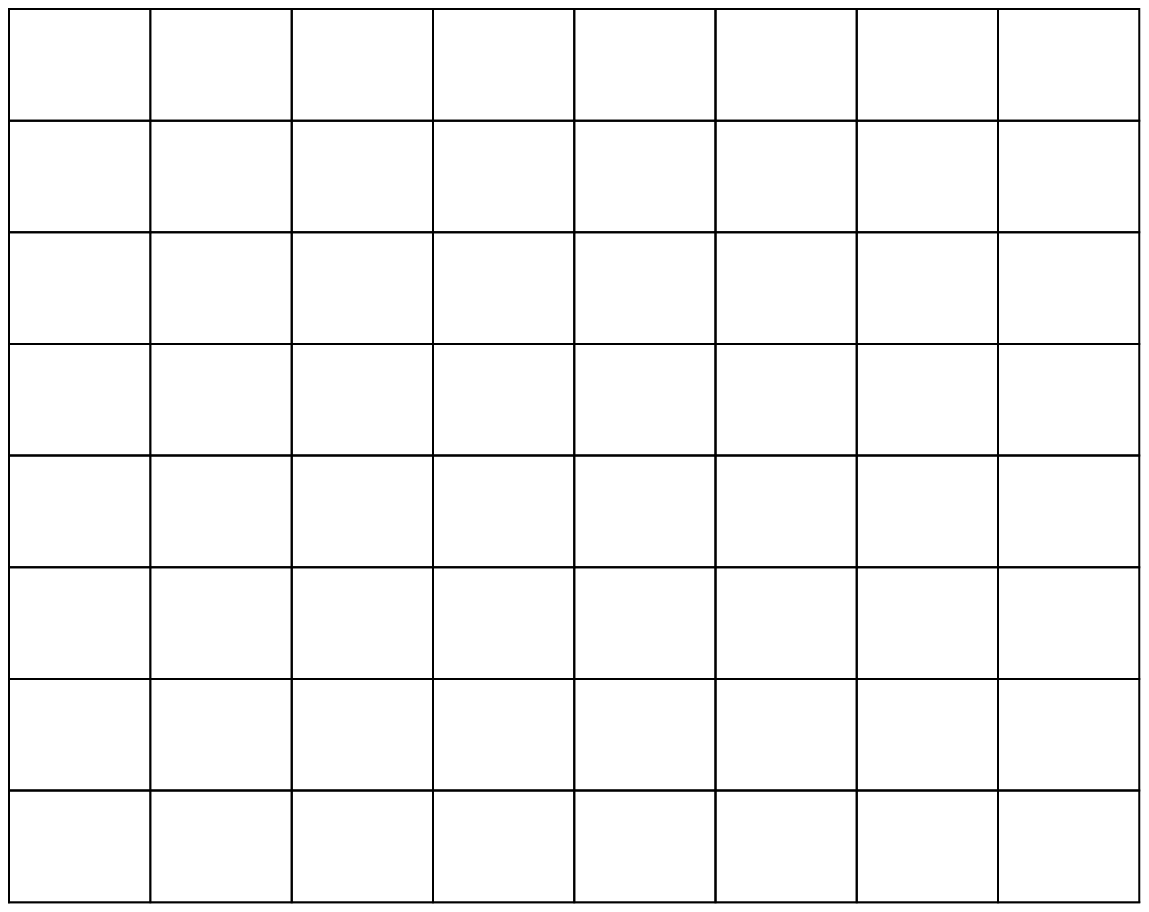}
\includegraphics[width=0.35\textwidth]{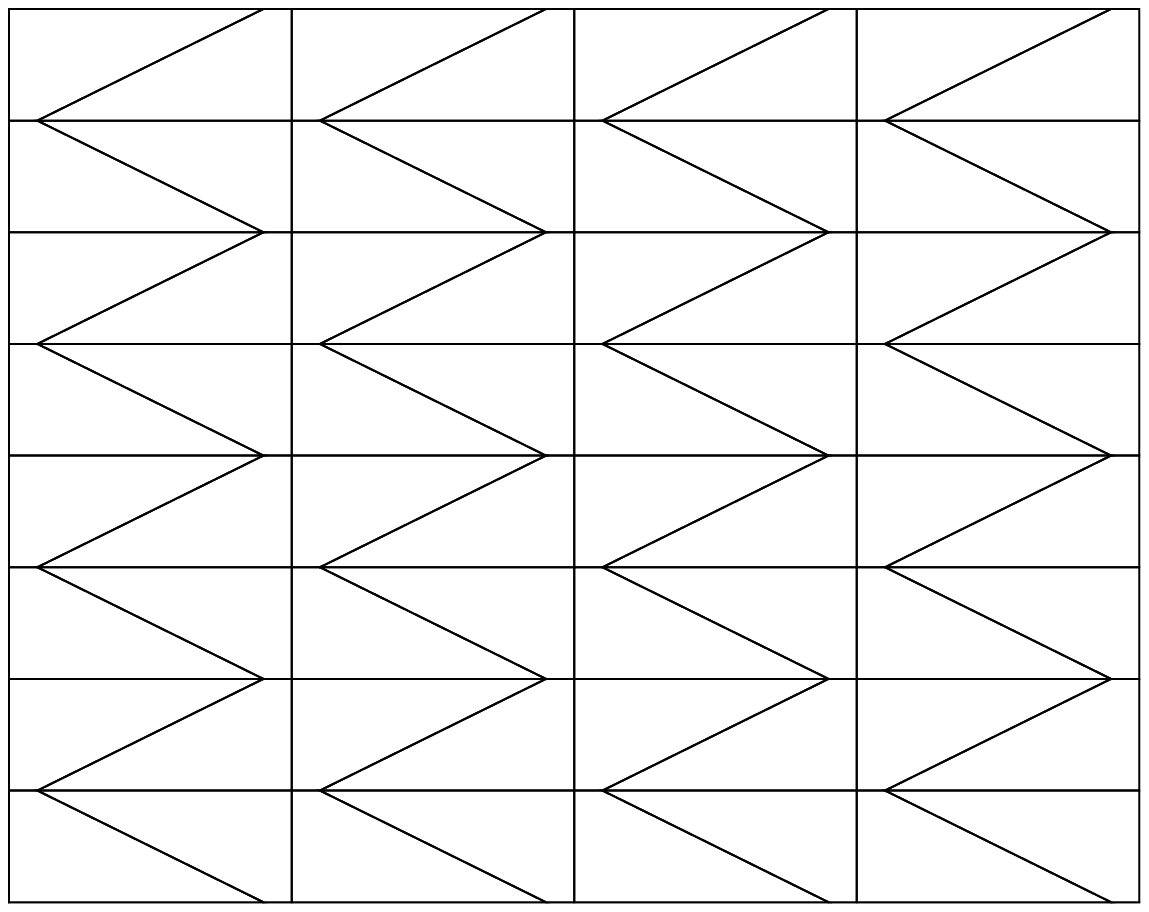}
\caption{Example~\ref{ex1}. Square mesh (left) and trapezoidal mesh (right).}
\label{ex1-mesh}
\end{figure}

\begin{figure}[t]
\centering
\includegraphics[width=0.32\textwidth]{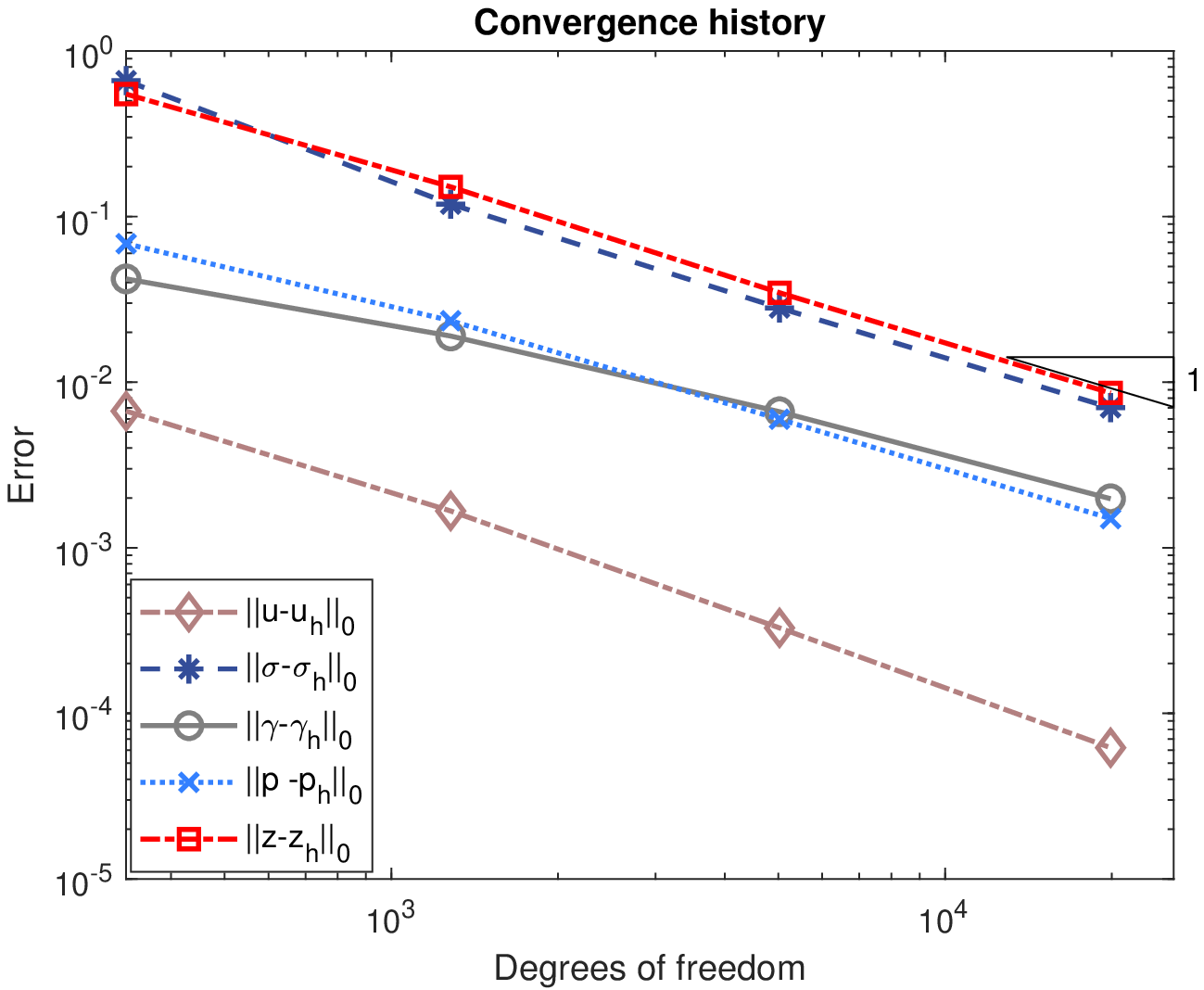}
\includegraphics[width=0.32\textwidth]{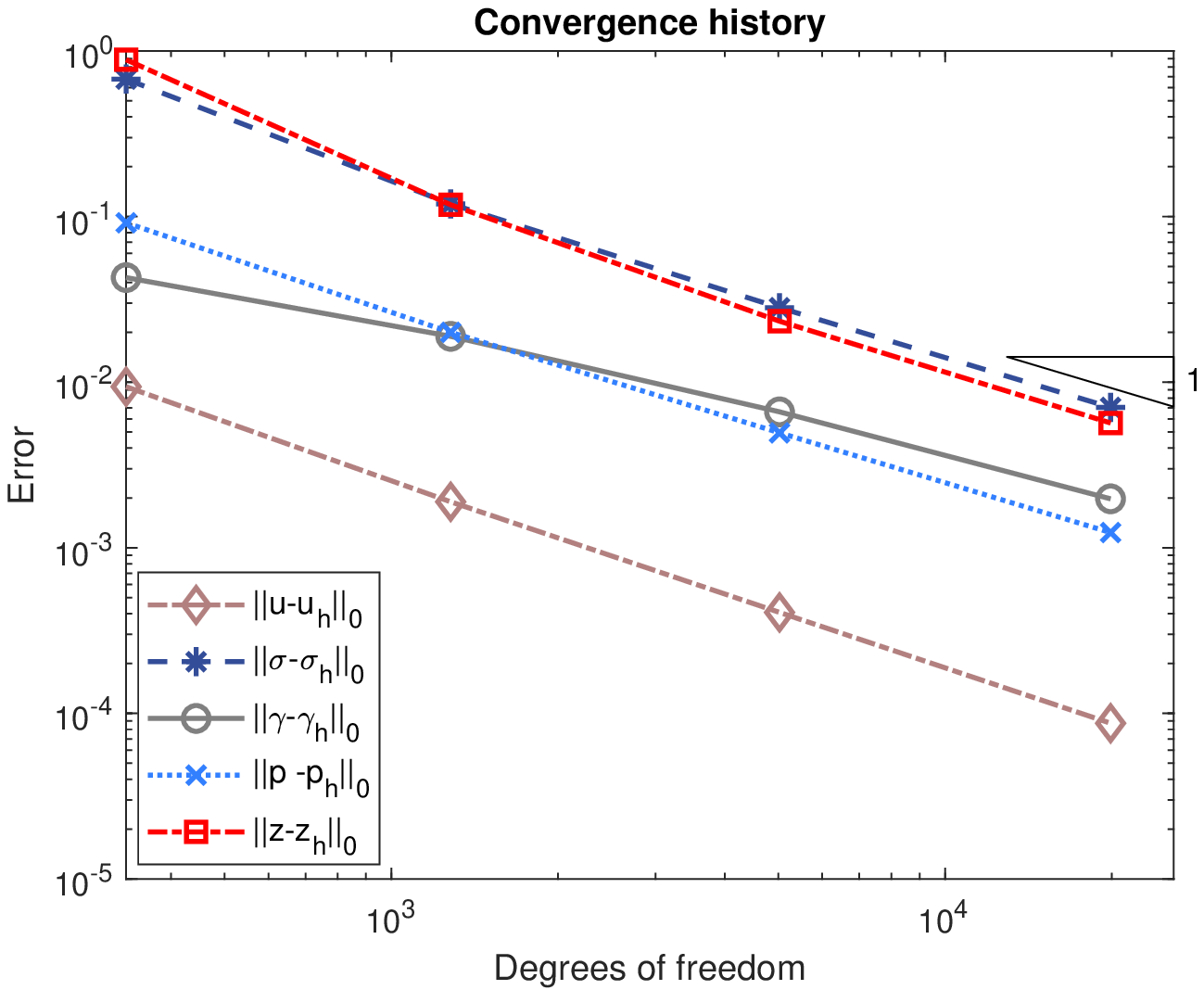}
\includegraphics[width=0.32\textwidth]{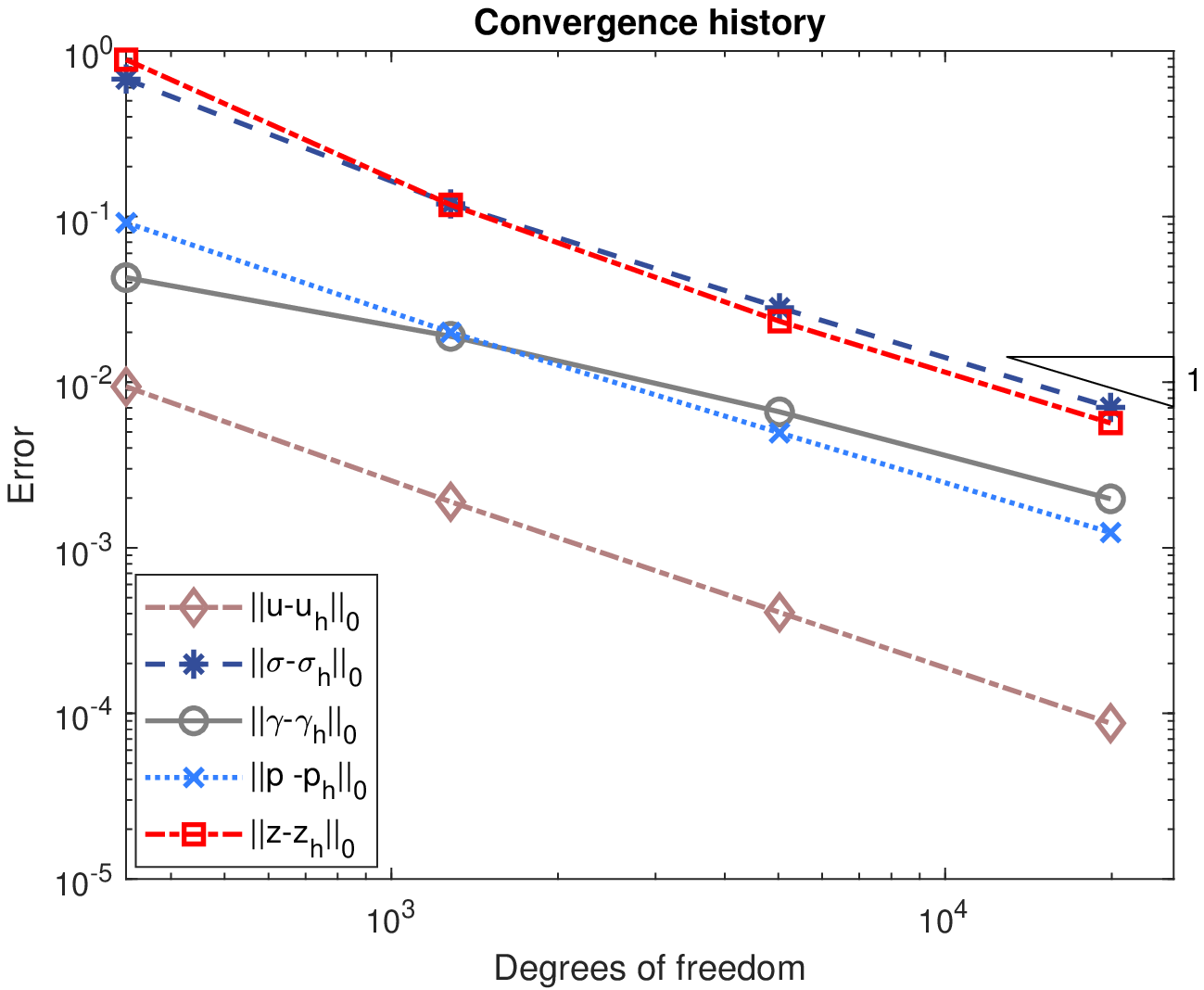}
\includegraphics[width=0.32\textwidth]{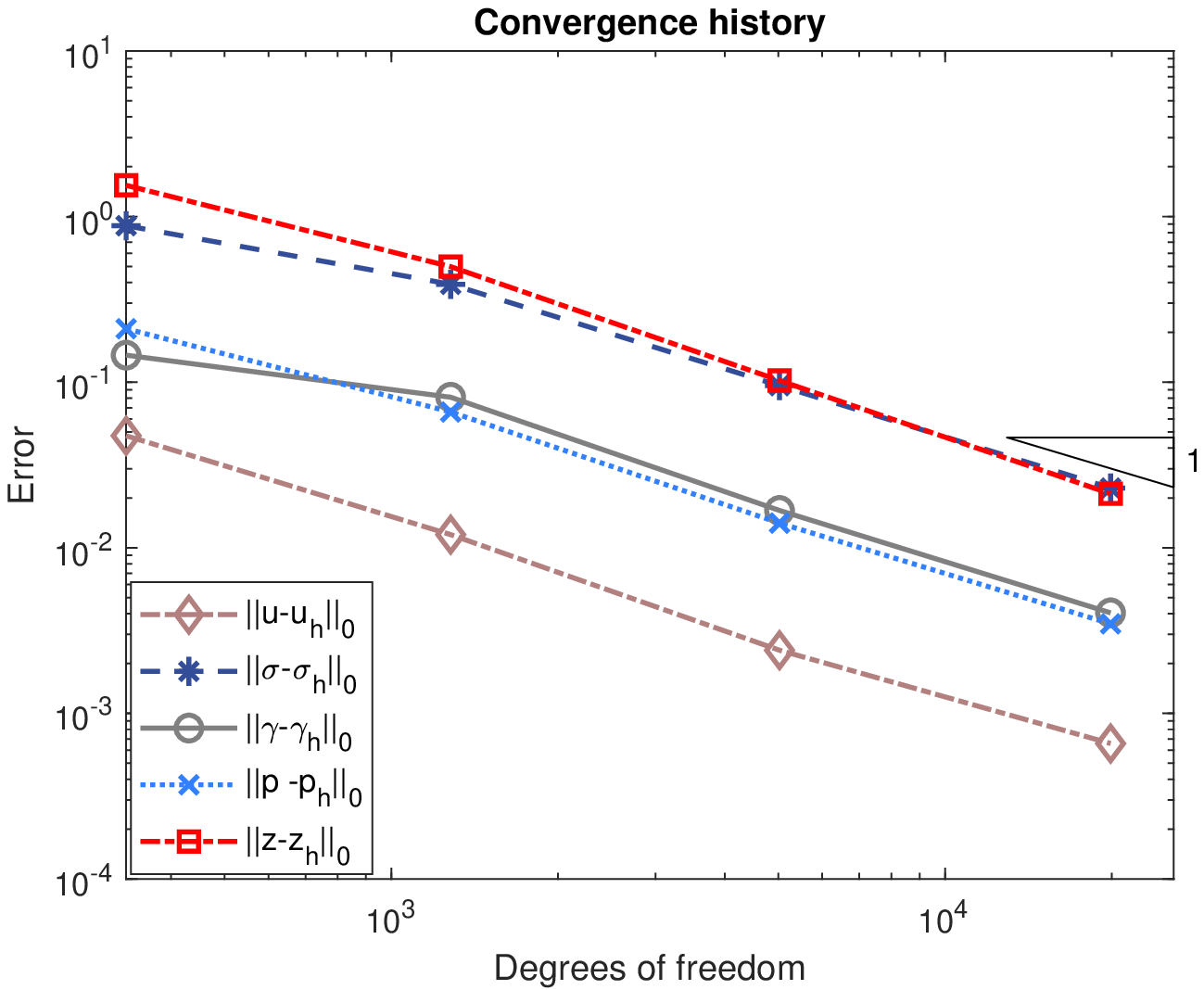}
\includegraphics[width=0.32\textwidth]{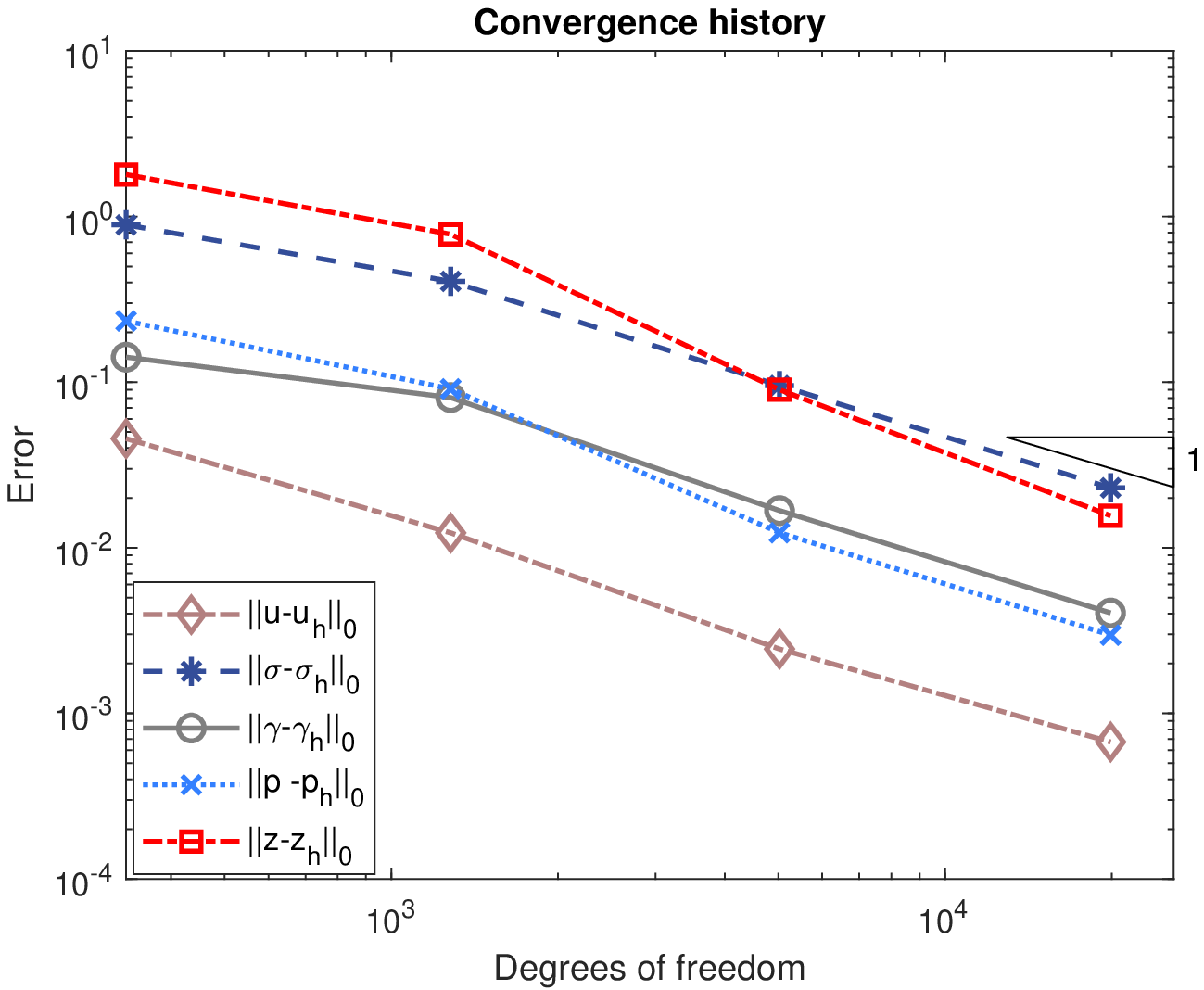}
\includegraphics[width=0.32\textwidth]{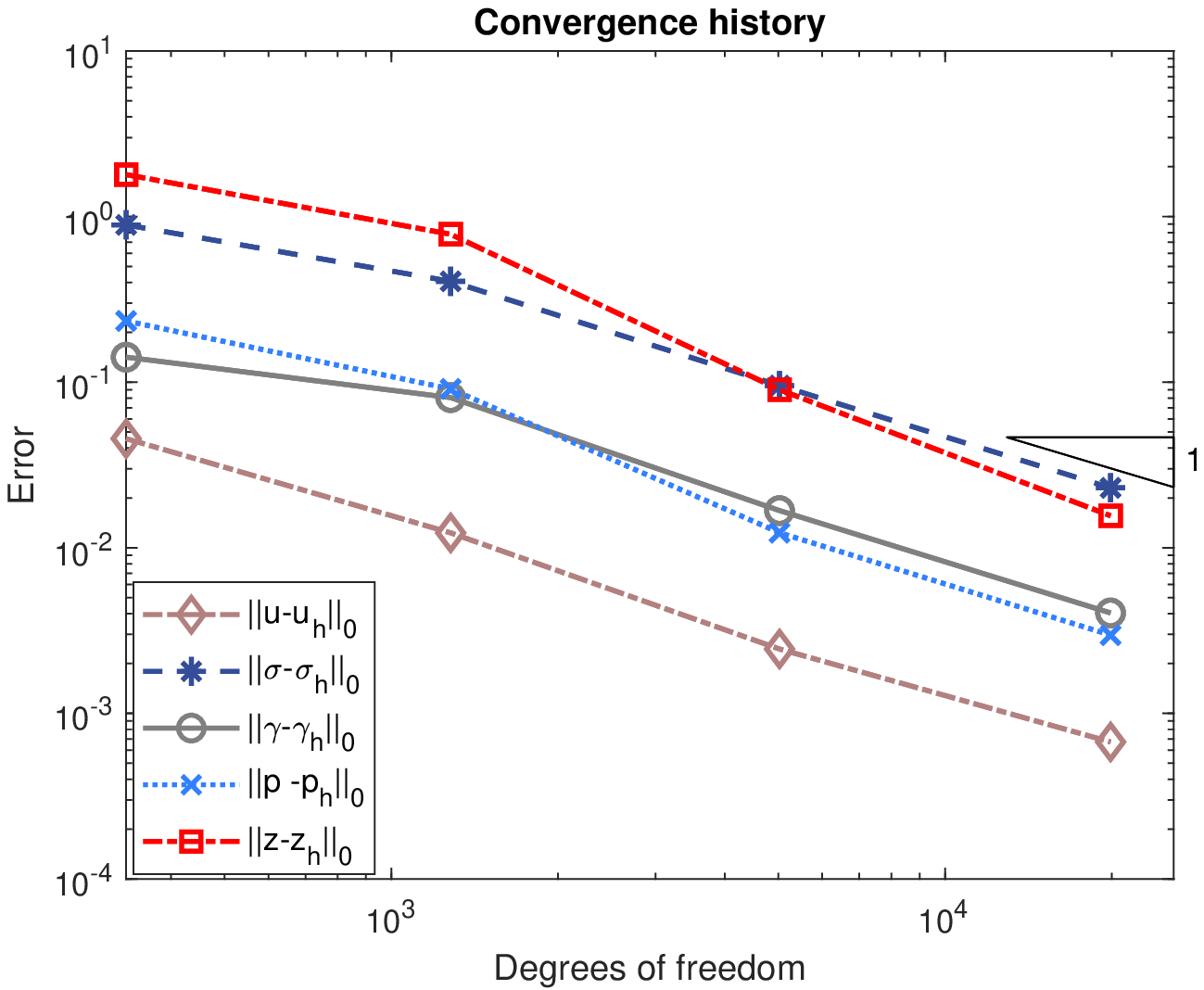}
\caption{Example~\ref{ex1}. Top plots: convergence history on square meshes for $c_0=1$ (left), $c_0=10^{-3}$ (middle) and $c_0=0$ (right) for $K=1$. Bottom plots: convergence history on trapezoidal meshes for $c_0=1$ (left), $c_0=10^{-3}$ (middle) and $c_0=0$ (right) for $K=1$.}
\label{ex1-con}
\end{figure}

\begin{figure}[t]
\centering
\includegraphics[width=0.32\textwidth]{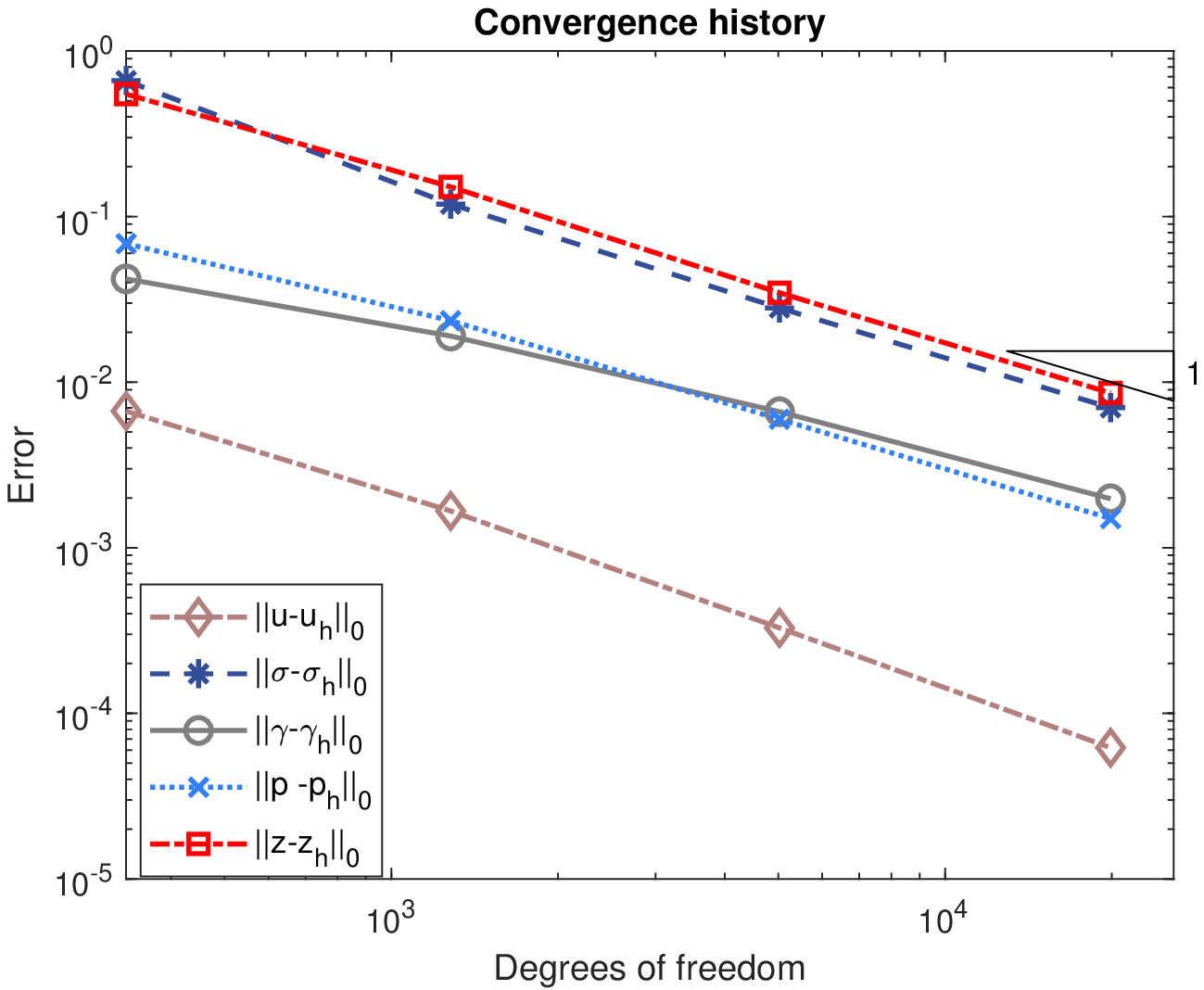}
\includegraphics[width=0.32\textwidth]{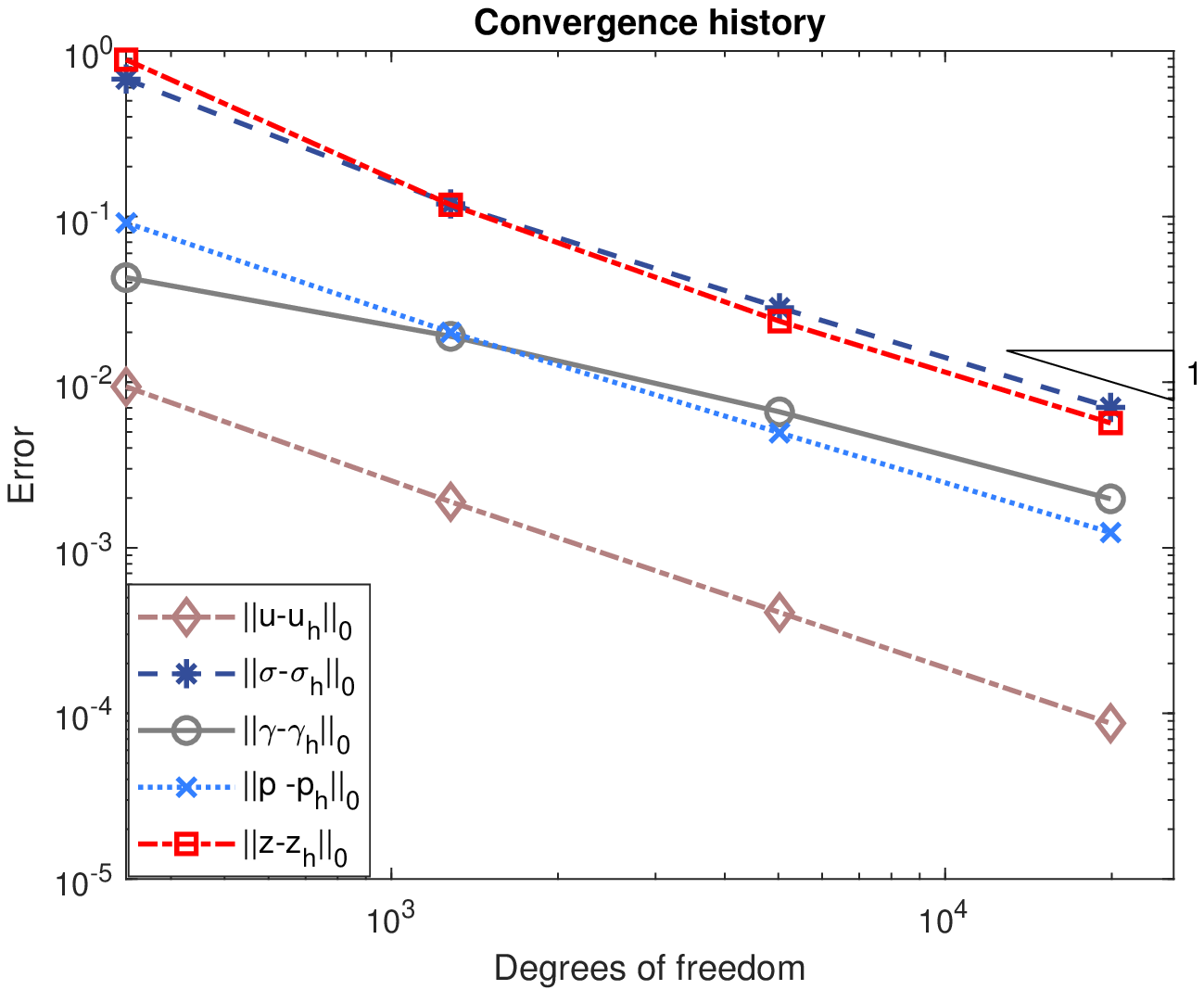}
\includegraphics[width=0.32\textwidth]{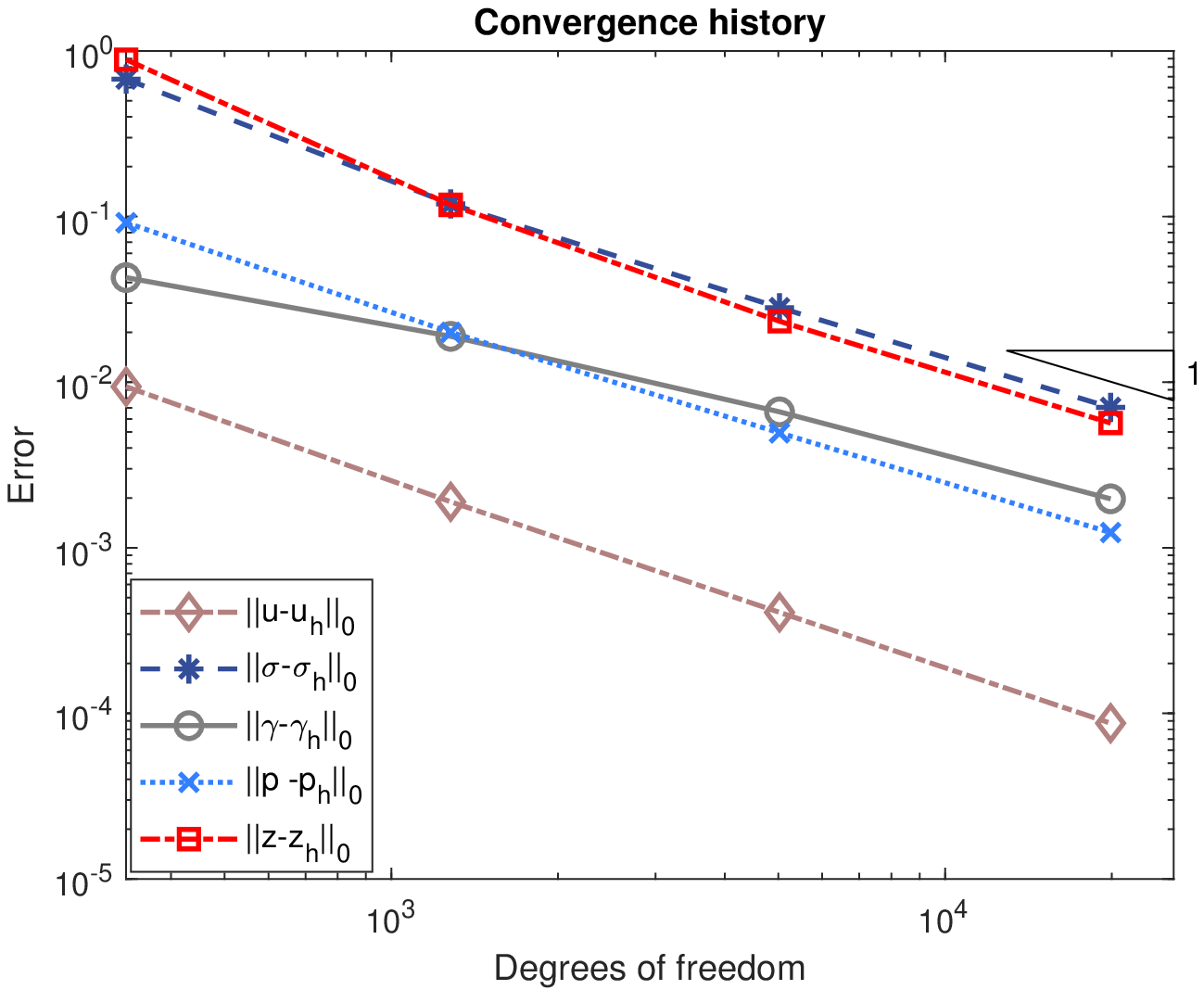}
\caption{Example~\ref{ex1}. Convergence history on square meshes for $c_0=1$ (left), $c_0=10^{-3}$ (middle) and $c_0=0$ (right) for anisotropic $K$.}
\label{ex1-con-aniso}
\end{figure}

\subsection{Cantilever bracket benchmark problem}\label{ex2}

The computational domain is the unit square. We impose no-flow boundary condition along all sides. The deformation is fixed at the left edge, and a downward traction is applied along the top. The bottom and right sides are traction-free, see Figure~\ref{ex2-boundary} for an illustration of the boundary condition. The time step for this example is $\Delta t =0.001$.

We use the same physical parameters as in \cite{Phillips09} as they typically induce locking:
\begin{align*}
E=10^{4},\quad\nu=0.4,\quad\alpha=0.93,\quad c_0=0,\quad K=10^{-6},
\end{align*}
where $E$ is the Young's modulus and $\nu$ is the Poisson ratio, and the Lam\'{e} parameters are computed by
\begin{align*}
\lambda = \frac{E \nu}{(1+\nu)(1-2\nu)}, \quad \mu = \frac{E}{2(1+\nu)}.
\end{align*}
Figure~\ref{ex2-pressure} shows that our proposed method yields a smooth pressure field, in contrast to the non-physical checkboard pattern that one obtains with continuous elasticity elements at the early time steps, see \cite{Phillips09}. In addition, Figure~\ref{ex2-pressure} also indicates that the pressure solution along different $x$-lines at time $T=0.005$. We can observe that it is free of oscillations and our solution agrees with the one obtained by DG-mixed discretizations \cite{Phillips09}.

\begin{figure}[t]
\centering
\includegraphics[width=0.5\textwidth]{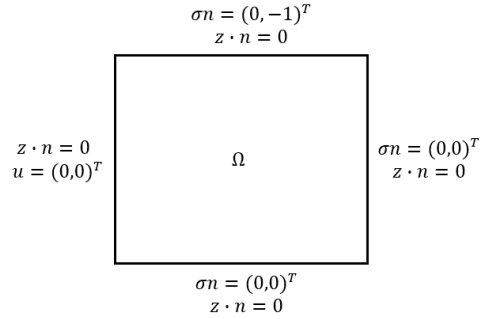}
\caption{Example~\ref{ex2}. Profile of boundary condition.}
\label{ex2-boundary}
\end{figure}

\begin{figure}[t]
\centering
\includegraphics[width=0.35\textwidth]{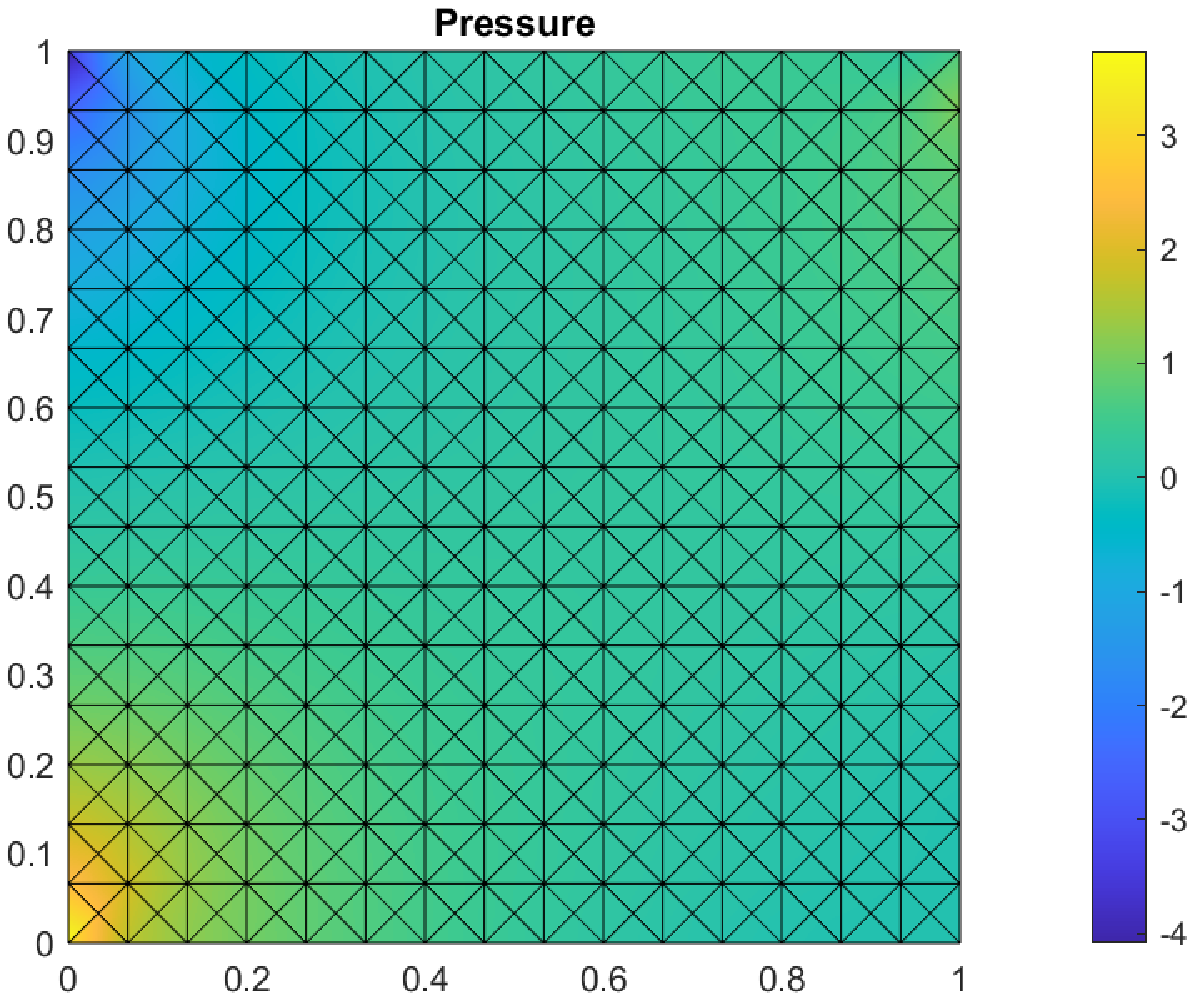}
\includegraphics[width=0.35\textwidth]{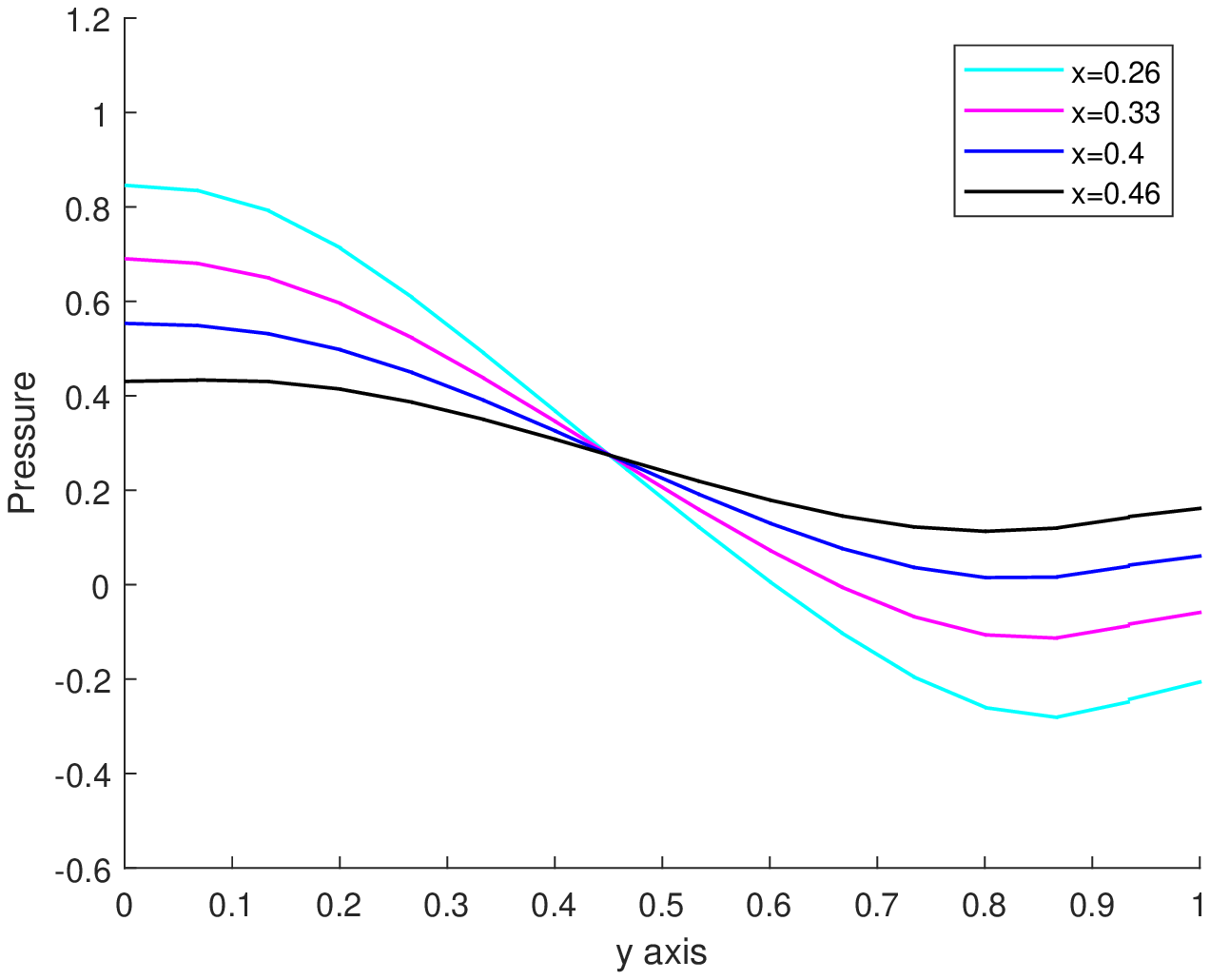}
\caption{Example~\ref{ex2}. Pressure field at $T=0.001$ (left) and pressure at different $x$-lines, $T=0.005$ (right).}
\label{ex2-pressure}
\end{figure}

\subsection{Anisotropic meshes}\label{ex3}

The computational domain is taken to be $\Omega=(0,1)^2$ and we set $\alpha=1,\lambda=1,\mu=1, K=I$. The simulation time is $T=0.01$ with time steps $\Delta t=h^2$. The exact displacement and pressure are given as
\begin{align*}
u=
\left(
  \begin{array}{c}
    \sin(2\pi t)\sin(2\pi y)\mbox{exp}(-x/\vartheta)\\
   \sin(2\pi t) \cos(2\pi y)\mbox{exp}(-x/\vartheta)\\
  \end{array}
\right),\quad p = \mbox{exp}(t)\mbox{ exp}(-x/\vartheta) x(1-x)^2y(1-y)^2.
\end{align*}
The used meshes are of Shishkin type (cf. \cite{Apel20}), see the example in Figure~\ref{ex3-mesh}. For a parameter $N\geq 2$ they are constructed by choosing a transition point parameter $ \delta\in (0,1)$ and generating a grid of points $(x^j,y^j)$ by
\begin{align*}
x^j& =
\begin{cases}
j\frac{2\delta}{N},\hspace{2.5cm} 0\leq j\leq \frac{N}{2},\; j\in \mathbb{N},\\
\delta+(j-\frac{N}{2})\frac{2(1-\delta)}{N}, \quad \frac{N}{2}<j\leq N,\; j\in \mathbb{N},
\end{cases}\\
y^i &= \frac{i}{N}\quad 0\leq i\leq N,\;  i\in \mathbb{N}.
\end{align*}
The transition point parameter $\delta$ is chosen as $\delta =\min\{\frac{1}{2},3\vartheta |\mbox{ln}(\vartheta)|\}$.
Here we choose $\vartheta=10^{-2}$ and the corresponding exact velocity and pressure at $t=0.01$ are displayed in Figures~\ref{ex3-sol} and \ref{ex3-mesh}, where the exponential boundary layer near $x=0$ is clearly visible. The layer has a width of $\mathcal{O}(\vartheta)$ and is present in the velocity and pressure solution. The convergence history against the number of degrees of freedom for various values of $c_0$ is shown in Figure~\ref{ex3-con} on square meshes and anisotropic meshes. We can observe that optimal convergence rates matching the theoretical results can be obtained, which indicates that our method can work well on anisotropic meshes and the method works for $c_0=0$. Further, the errors on anisotropic meshes are much smaller than that of rectangular meshes, which illustrates the superior performances of anisotropic meshes in particular for layered problem.

\begin{figure}[t]
\centering
\includegraphics[width=0.35\textwidth]{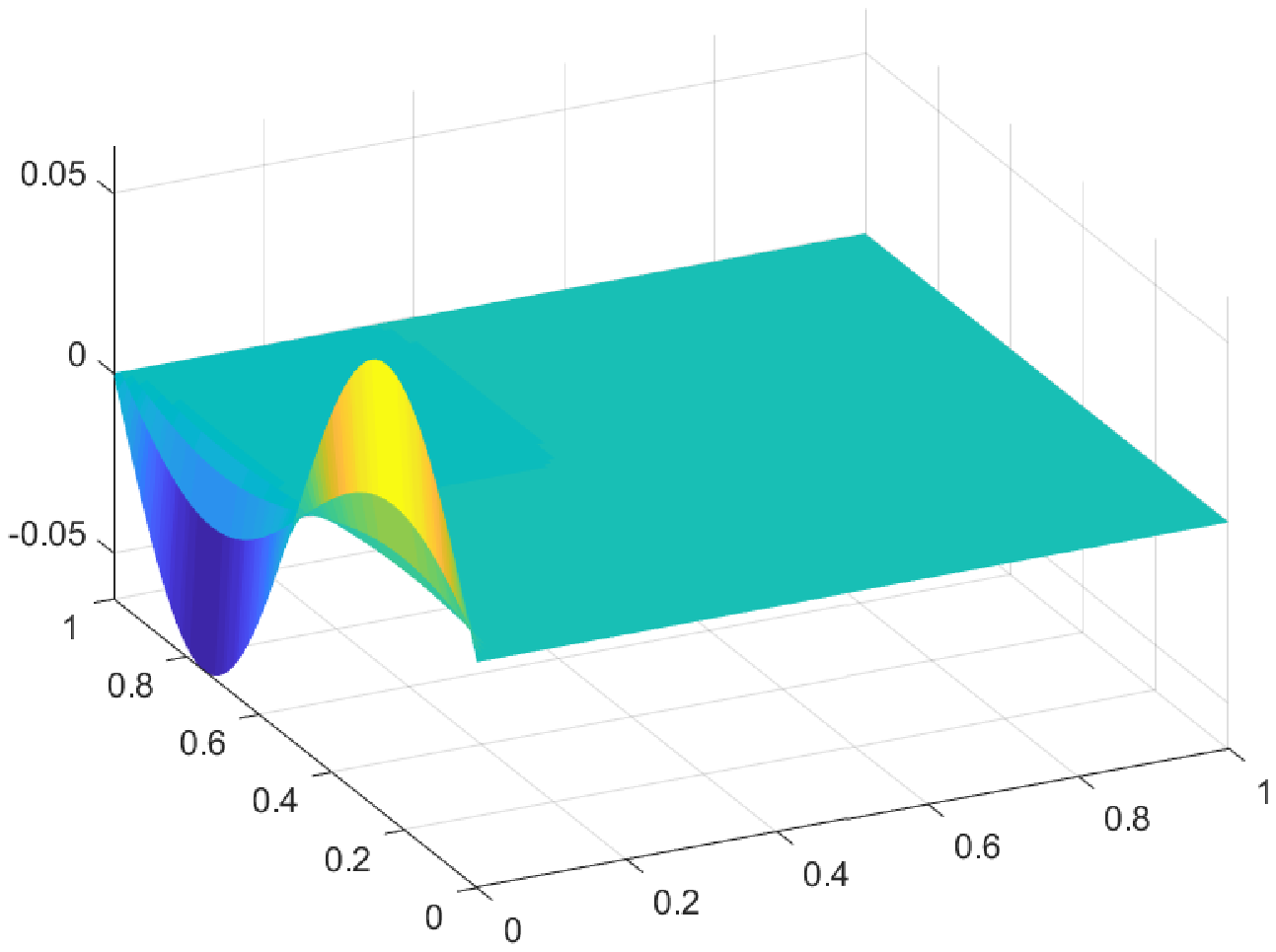}
\includegraphics[width=0.35\textwidth]{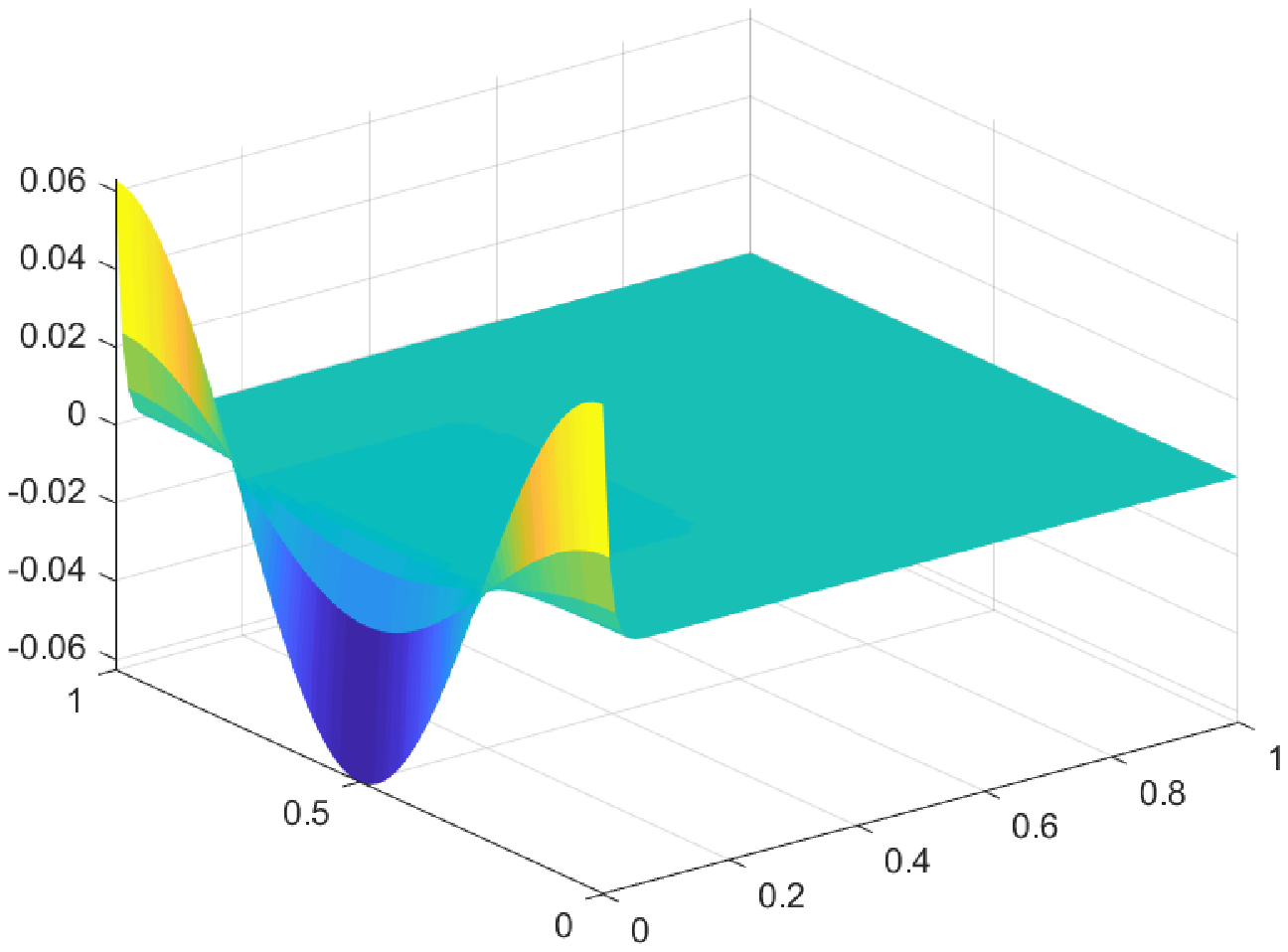}
\caption{Example~\ref{ex3}. Exact velocity: $u_1$ (left) and $u_2$ (right).}
\label{ex3-sol}
\end{figure}

\begin{figure}[t]
\centering
\includegraphics[width=0.35\textwidth]{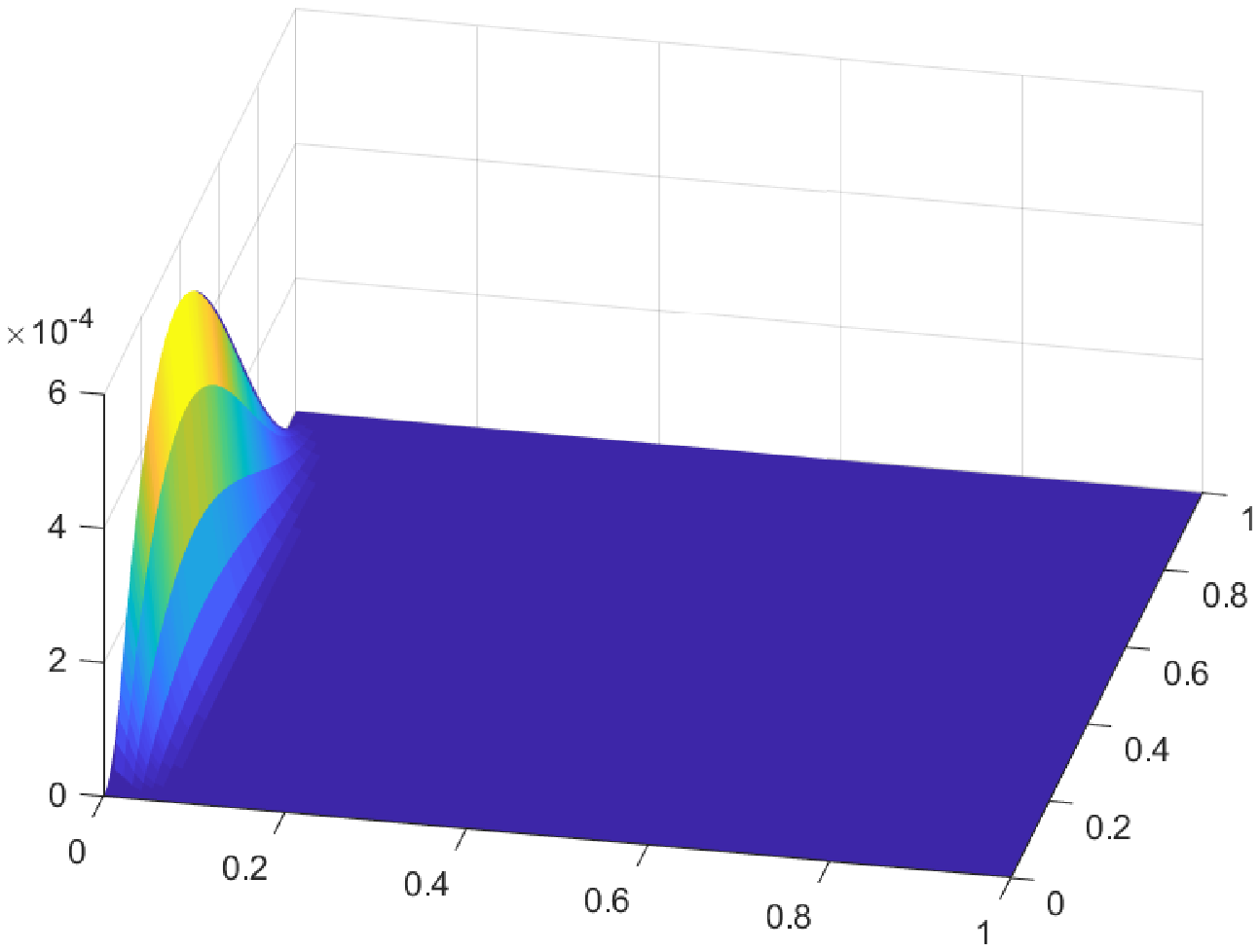}
\includegraphics[width=0.35\textwidth]{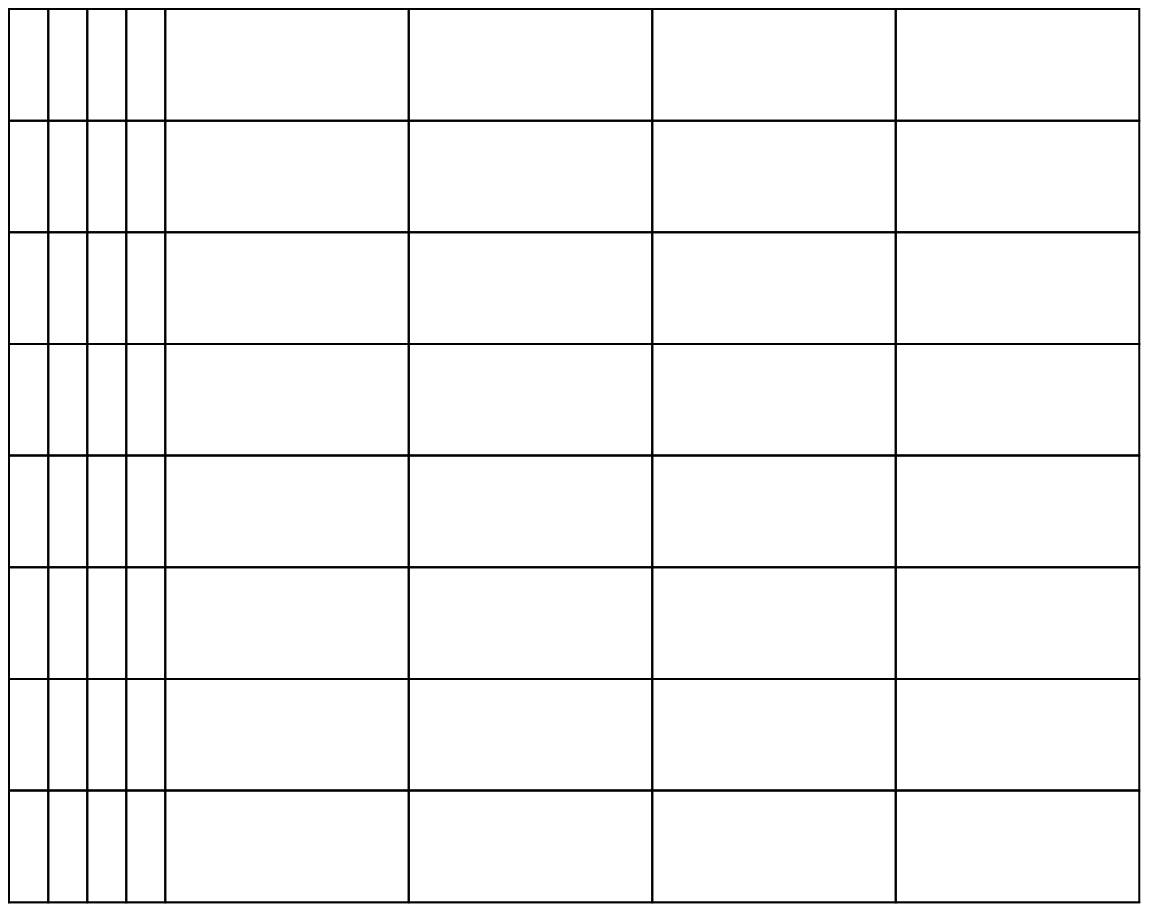}
\caption{Example~\ref{ex3}. Exact pressure (left) and example mesh (right) for $\vartheta=10^{-2}, N=8$ used in the simulations.}
\label{ex3-mesh}
\end{figure}

\begin{figure}[t]
\centering
\includegraphics[width=0.32\textwidth]{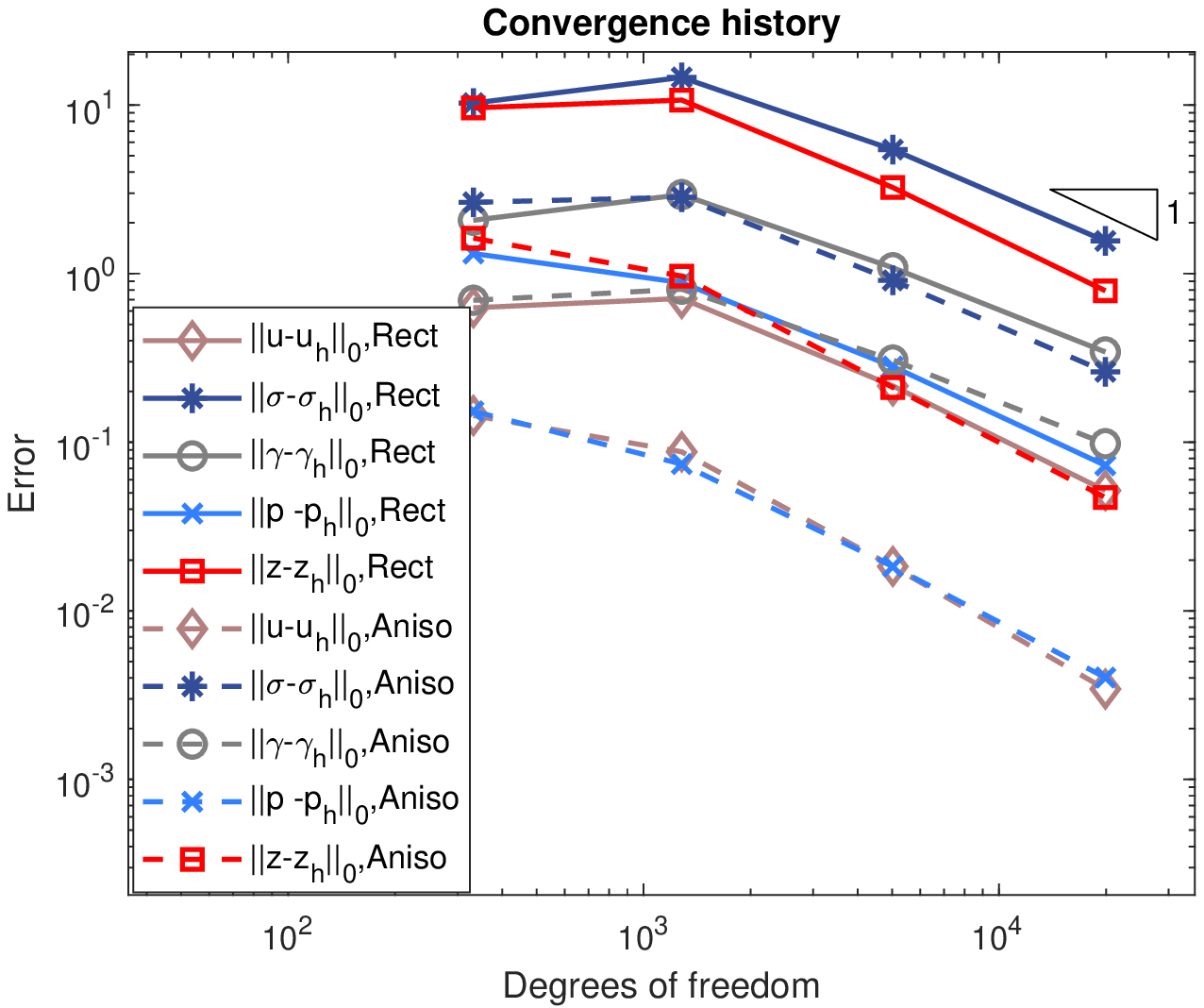}
\includegraphics[width=0.32\textwidth]{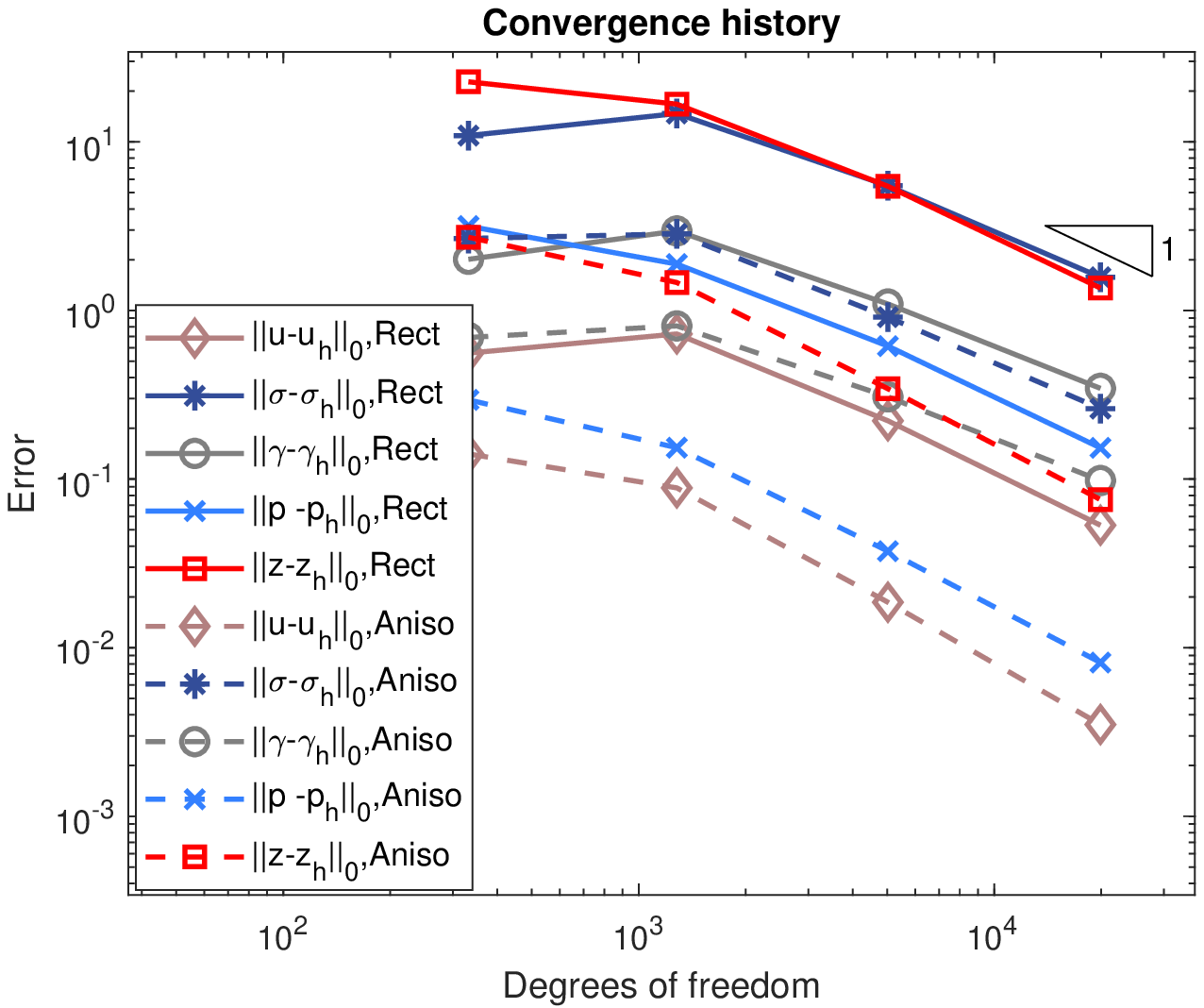}
\includegraphics[width=0.32\textwidth]{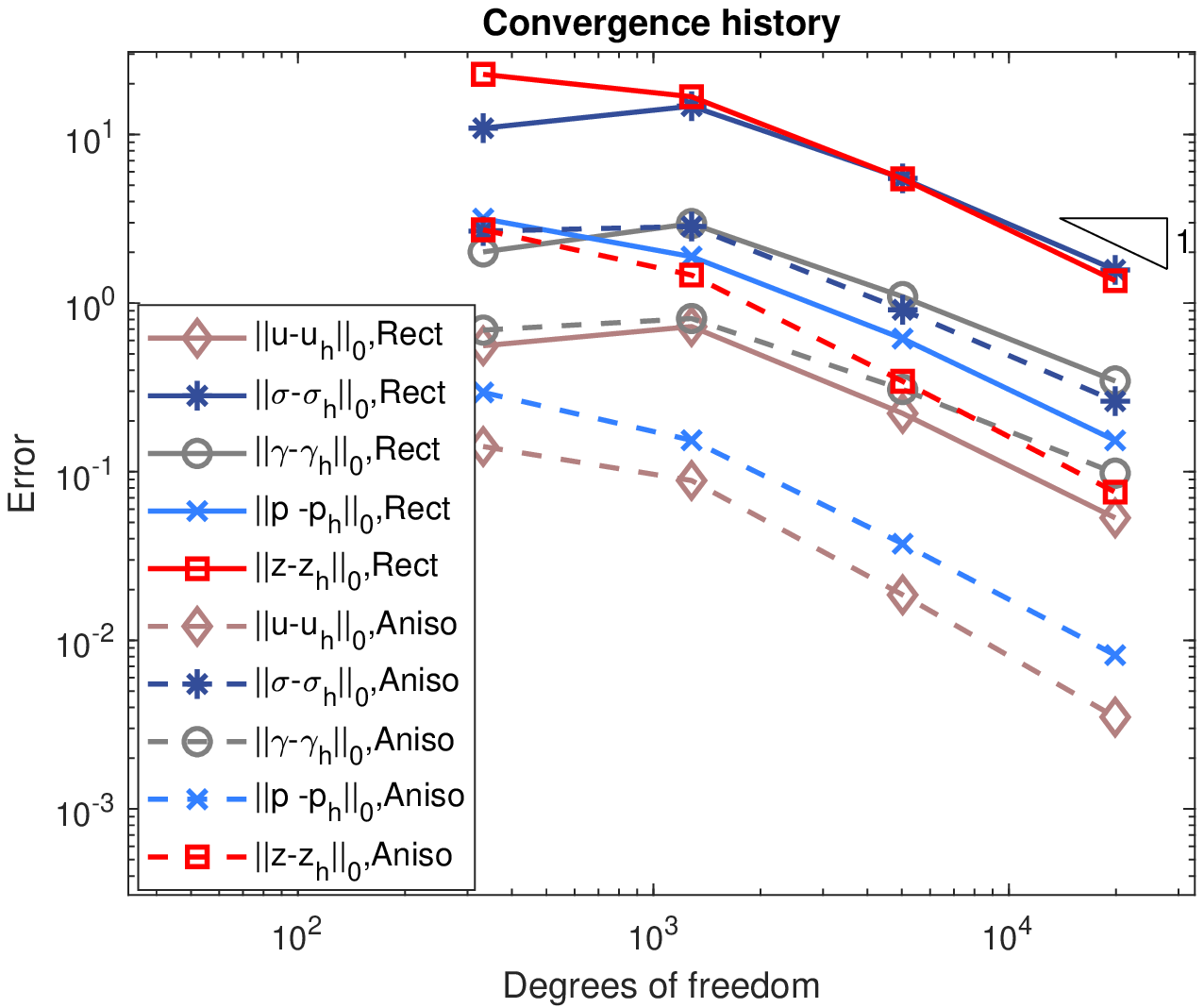}
\caption{Example~\ref{ex3}. Convergence history with uniform rectangular meshes (solid line) and anisotropic meshes (dashed line). $c_0=1$ (left), $c_0=10^{-3}$ (middle), and $c_0=0$ (right).}
\label{ex3-con}
\end{figure}

\subsection{A model problem of vertical compaction}\label{ex4}

In this example, we consider a more practical case motivated by the example given in \cite{Leake97}, where three sedimentary layers overlying an impermeable bedrock in a basin. The sediment stack totals 420m at the deepest point of the basin ($x = 0$m) but thins to 120m above the step ($x > 4000$m). The top two layers of the sequence are 20m thick each, see Figure~\ref{ex4-domain} for an illustration of the domain geometry. The first and third layers are aquifers; the middle layer is relatively impermeable to flow. The flow field initially is at steady state, but pumping from the lower aquifer reduces hydraulic head by 6m per year at the basin center. The head drop moves fluid away from the step. The fluid supply in the upper reservoir is limitless. The initial condition is set to zero and the boundary conditions for region A to E are described in Table~\ref{table:boundary}, in addition, we let $f=q=0$. The values for the parameters are given in Table~\ref{table:val}. The final simulation time is 10 years, where we take the simulation time to be $(0:360:360\times 10)$. Figure~\ref{ex4-sol} shows the plot for pressure and displacement at 2 years and 10 years.

\begin{figure}[t]
\centering
\includegraphics[width=0.4\textwidth]{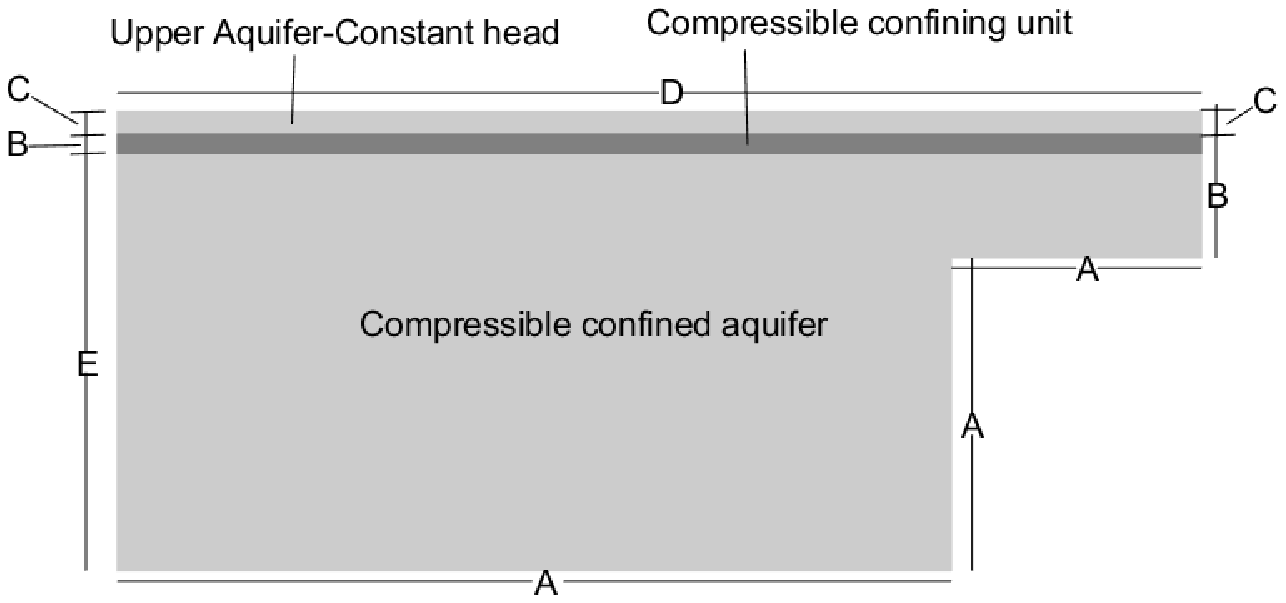}
\includegraphics[width=0.4\textwidth]{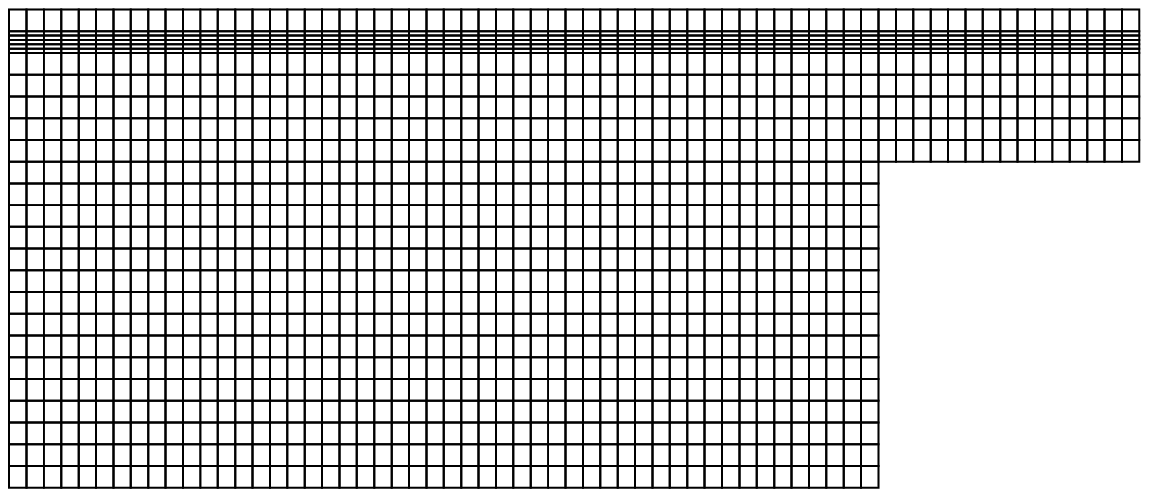}
\vskip -1cm
\caption{Example~\ref{ex4}. Model geometry showing boundary segments (left) and profile of the mesh used (right). Vertical exaggeration $\times 5$}
\label{ex4-domain}
\end{figure}

\begin{table}[t]
\centering
\begin{tabular}{l|l|l}
\hline
A & $z\cdot n=0$ & $u=(0,0)^T$ \\
B &$z\cdot n=0$ & $u_1=0$\\
C &$p=0$ & $u_1=0$\\
D &$p=0$& $\sigma n=(0,0)^T$\\
E &$p=H(t)$& $u_1=0$\\
\hline
\end{tabular}
\caption{Example~\ref{ex4}. Description of boundary conditions. $u=(u_1,u_2)^T$.}
\label{table:boundary}
\end{table}

\begin{table}[t]
\centering
\begin{tabular}{|l|l|l|}
\hline
Variable & Description & Value\\
\hline
$c_0$ & storage coefficient, aquifer layers & $10^{-6}\mbox{m}^{-1}$ \\
$c_0$ & storage coefficient, confining layers & $10^{-5}\mbox{m}^{-1}$ \\
$K$ &permeability, aquifer layers & 25 m/day \\
$K$ &permeability, confining layers & 0.01 m/day \\
$\alpha$ &Biot-Willis coefficient& 1\\
E & Young's modulus, aquifer layers & $8\cdot 10^8$ Pa\\
E & Young's modulus, confining layers & $8\cdot 10^7$ Pa\\
$\nu$ &Poisson ratio, all regions& 0.25\\
$H(t)$ &Declining head boundary& (6 m/year)$\cdot t$\\
\hline
\end{tabular}
\caption{Example~\ref{ex4}. Values of parameters.}
\label{table:val}
\end{table}

\begin{figure}[t]
\centering
\includegraphics[width=0.35\textwidth]{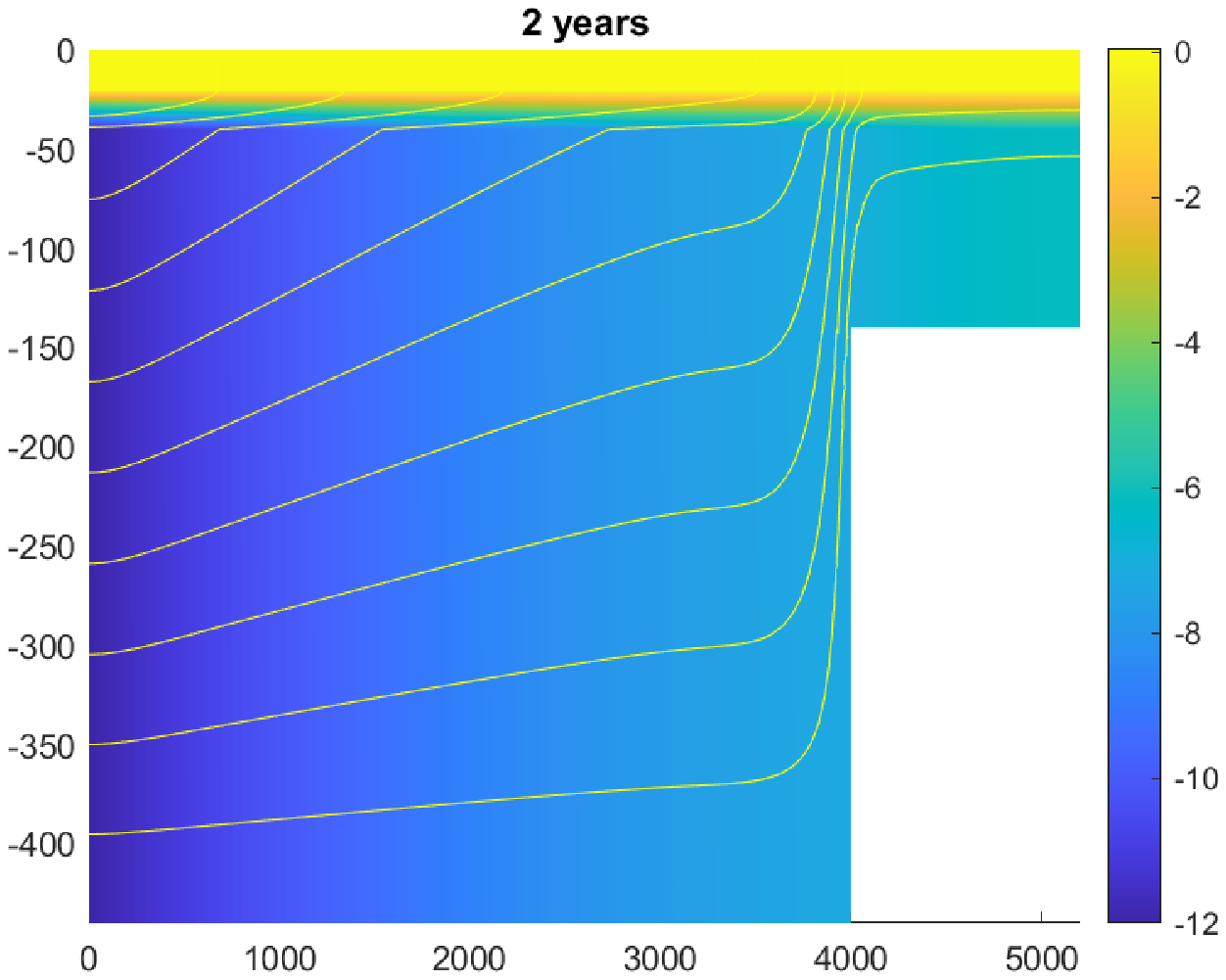}
\includegraphics[width=0.35\textwidth]{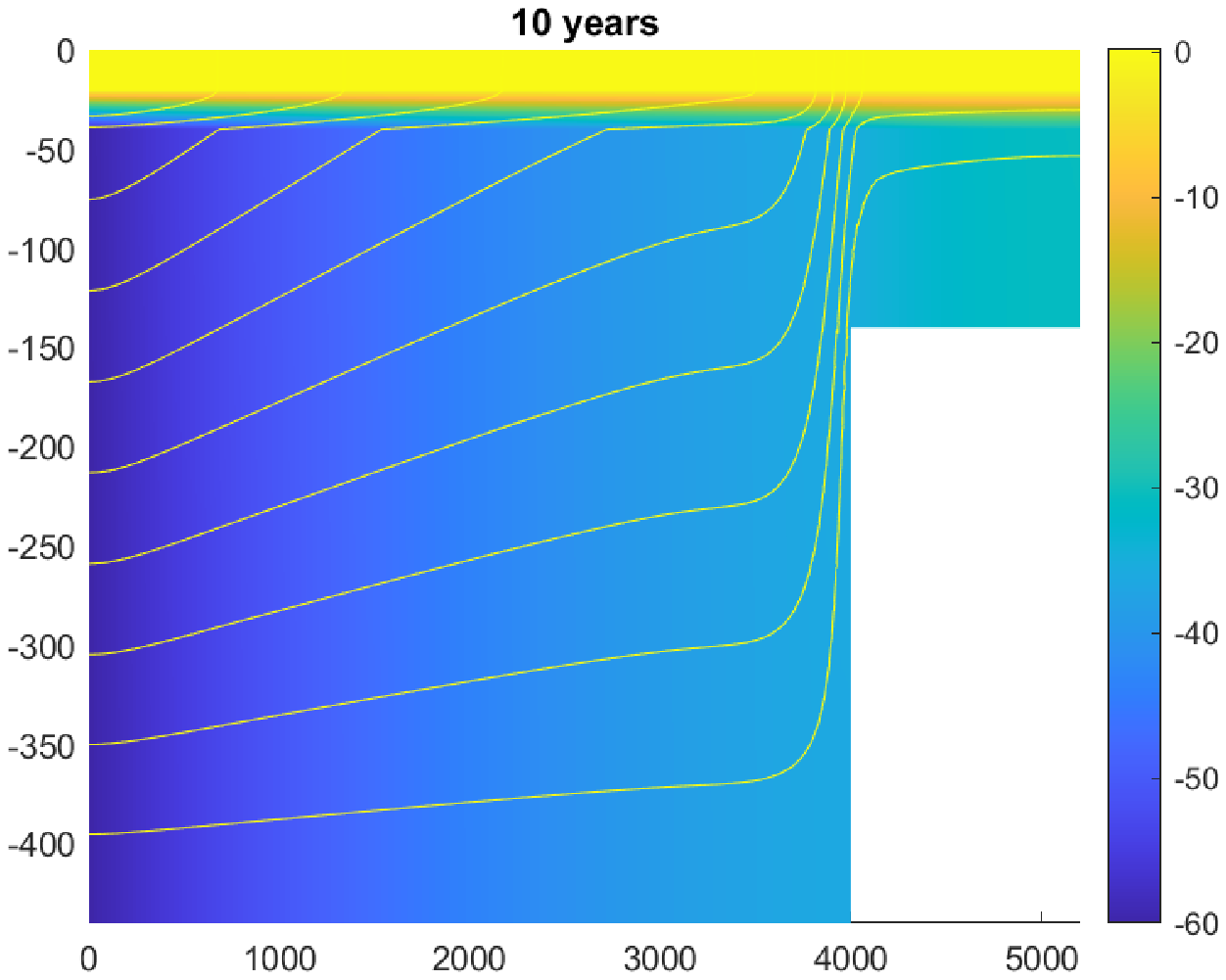}
\caption{Example~\ref{ex4}. Numerical results at 2 years (left) and 10 years (right). Surface: pressure; Contour: displacement ($|u_h|$).}
\label{ex4-sol}
\end{figure}

\section{Conclusion}
In this paper we analyzed a staggered DG method for a five-field formulation of the Biot system of poroelasticity on general polygonal meshes. The method is proved to converge in optimal rates for all the variables in their natural norms. In addition, the error estimates are independent of the storativity coefficient $c_0$ and are valid even for $c_0=0$, which is also confirmed by our numerical simulation. Several numerical experiments including cantilever bracket benchmark problem are presented, which further confirms the locking-free property of our method. In addition, the accuracy of our method is slightly influenced by the shape of the grid. Numerical results illustrate that our method is a good option for solving problems on anisotropic meshes. Moreover, our method can handle nonmatching meshes straightforwardly, which is advantageous in solving problem with local features. In the future we will extend this method to solve nonlinear poroelasticity.

\section*{Acknowledgment}

The research of Eric Chung is partially supported by the Hong
 Kong RGC General Research Fund (Project numbers 14304719 and 14302018)
 and CUHK Faculty of Science Direct Grant 2019-20 and NSFC/RGC Joint
Research Scheme (Project number HKUST620/15). The research of
Eun-Jae Park was supported by the National Research Foundation of
Korea (NRF) grant funded by the Ministry of Science and ICT
(NRF-2015R1A5A1009350 and NRF-2019R1A2C2090021).

\bibliographystyle{plain}
\bibliography{reference}

\begin{thebibliography}{10}

\bibitem{Ambartsumyan20}
I.~Ambartsumyan, E.~Khattatov, and I.~Yotov.
\newblock A coupled multipoint stress -- multipoint flux mixed finite element
  method for the {Biot} system of poroelasticity.
\newblock {\em arXiv:2001.04582}, 2020.

\bibitem{Apel20}
T.~Apel, A.~Linke, and C.~Merdon.
\newblock A nonconforming pressure-robust finite element method for the
  {Stokes} equations on anisotropic meshes.
\newblock {\em arXiv:2002.12127}, 2020.

\bibitem{Biot41}
M.~A. Biot.
\newblock General theory of three‐dimensional consolidation.
\newblock {\em J. Appl. Phys.}, 12(2):155--164, 1941.

\bibitem{Boffi16}
D.~Boffi, M.~Botti, and D.~A.~Di Pietro.
\newblock A nonconforming high-order method for the {Biot} problem on general
  meshes.
\newblock {\em SIAM J. Sci. Comput.}, 38(3):A1508--A1537, 2016.

\bibitem{Braess07}
D.~Braess.
\newblock {\em Finite {Elements}: {Theory}, {Fast Solvers}, and {Applications
  in Solid Mechanics}}.
\newblock Cambridge University Press, 2007.

\bibitem{Brenner03}
S.~C. Brenner.
\newblock Poincar\'{e}-{Friedrichs} inequalities for piecewise {$H^1$}
  functions.
\newblock {\em SIAM J. Numer. Anal.}, 41(1):306--324, 2003.

\bibitem{Brezzi91}
F.~Brezzi and R.~S. Falk.
\newblock Stability of higher-order {Hood-Taylor} methods.
\newblock {\em SIAM J. Numer. Anal.}, 28(3):581--590, 1991.

\bibitem{Cheung15}
S.~Cheung, E.~T. Chung, H.~H. Kim, and Y.~Qian.
\newblock Staggered discontinuous {Galerkin} methods for the incompressible
  {Navier–Stokes} equations.
\newblock {\em J. Comput. Phys.}, 302(1):251--266, 2015.

\bibitem{Chung14}
E.~T. Chung, B.~Cockburn, and G.~Fu.
\newblock The staggered {DG} method is the limit of a hybridizable {DG} method.
\newblock {\em SIAM J. Numer. Anal.}, 52(2):915--932, 2014.

\bibitem{Chung16}
E.~T. Chung, B.~Cockburn, and G.~Fu.
\newblock The staggered {DG} method is the limit of a hybridizable {DG} method.
  {Part} {II}: The {Stokes} flow.
\newblock {\em J. Sci. Comput.}, 66(2):870--887, 2016.

\bibitem{ChungDu17}
E.~T. Chung, J.~Du, and C.~Lam.
\newblock Discontinuous {Galerkin} methods with staggered hybridization for
  linear elastodynamics.
\newblock {\em Comput. Math. Appl.}, 74(6):1198--1214, 2017.

\bibitem{ChungEngquist06}
E.~T. Chung and B.~Engquist.
\newblock Optimal discontinuous {Galerkin} methods for wave propagation.
\newblock {\em SIAM J. Numer. Anal.}, 44(5):2131--2158, 2006.

\bibitem{ChungWave09}
E.~T. Chung and B.~Engquist.
\newblock Optimal discontinuous {Galerkin} methods for the acoustic wave
  equation in higher dimensions.
\newblock {\em SIAM J. Numer. Anal.}, 47(5):3820--3848, 2009.

\bibitem{ChungKim13}
E.~T. Chung, H.~H. Kim, and O.~B. Widlund.
\newblock Two-level overlapping {Schwarz} algorithms for a staggered
  discontinuous {Galerkin} method.
\newblock {\em SIAM J. Numer. Anal.}, 51(1):47--67, 2013.

\bibitem{ChungLee11}
E.~T. Chung and C.~Lee.
\newblock A staggered discontinuous {Galerkin} method for the curl–curl
  operator.
\newblock {\em IMA J. Numer. Anal.}, 32(3):1241--1265, 2012.

\bibitem{ChungQiu17}
E.~T. Chung and W.~Qiu.
\newblock Analysis of an {SDG} method for the incompressible {Navier--Stokes}
  equations.
\newblock {\em SIAM J. Numer. Anal.}, 55(2):543--569, 2017.

\bibitem{Ciarlet02}
P.~G. Ciarlet.
\newblock {\em The Finite Element Method for Elliptic Problems}.
\newblock Society for Industrial and Applied Mathematics, 2002.

\bibitem{Du18}
J.~Du and E.~T. Chung.
\newblock An adaptive staggered discontinuous {Galerkin} method for the steady
  state convection–diffusion equation.
\newblock {\em J. Sci. Comput.}, 77(3):1490--1518, 2018.

\bibitem{Fu19}
G.~Fu.
\newblock A high-order {HDG} method for the {Biot’s} consolidation model.
\newblock {\em Comput. Math. Appl.}, 77(1):237--252, 2019.

\bibitem{Gaspar03}
F.~J. Gaspar, F.~J. Lisbona, and P.~N. Vabishchevich.
\newblock A finite difference analysis of {Biot's} consolidation model.
\newblock {\em Appl. Numer. Math.}, 44(4):487--506, 2003.

\bibitem{Guzman18}
J.~Guzm\'{a}n and L.~R. Scott.
\newblock The {Scott}-{Vogelius} finite elements revisited.
\newblock {\em Math. Comp.}, 88(316):515--529, 2019.

\bibitem{HuMu18}
X.~Hu, L.~Mu, and X.~Ye.
\newblock Weak {Galerkin} method for the {Biot’s} consolidation model.
\newblock {\em Comput. Math. Appl.}, 75(6):2017--2030, 2018.

\bibitem{KimChungStokes13}
H.~H. Kim, E.~T. Chung, and C.~Lee.
\newblock A staggered discontinuous {Galerkin} method for the {Stokes} system.
\newblock {\em SIAM J. Numer. Anal.}, 51(6):3327--3350, 2013.

\bibitem{Korsawe05}
J.~Korsawe and G.~Starke.
\newblock A least-squares mixed finite element method for {Biot's}
  consolidation problem in porous media.
\newblock {\em SIAM J. Numer. Anal.}, 43(1):318--339, 2005.

\bibitem{Leake97}
S.~A. Leake and P.~A. Hsieh.
\newblock Simulation of deformation of sediments from decline of ground-water
  levels in an aquifer underlain by a bedrock step.
\newblock {\em U.S. Geological Survey Open File Report: 97-47}, 1997.

\bibitem{Lee16}
J.~J. Lee.
\newblock Robust error analysis of coupled mixed methods for {Biot’s}
  consolidation model.
\newblock {\em J. Sci. Comput.}, 69(2):610--632, 2016.

\bibitem{Lee18}
J.~J. Lee.
\newblock Robust three-field finite element methods for {Biot’s}
  consolidation model in poroelasticity.
\newblock {\em BIT Numer. Math}, 58(2):347--372, 2018.

\bibitem{LeeKim16}
J.~J. Lee and H.~H. Kim.
\newblock Analysis of a staggered discontinuous {Galerkin} method for linear
  elasticity.
\newblock {\em J. Sci. Comput.}, 66(2):625--649, 2016.

\bibitem{LeeMardal17}
J.~J. Lee, K.-A. Mardal, and R.~Winther.
\newblock Parameter-robust discretization and preconditioning of {Biot's}
  consolidation model.
\newblock {\em SIAM J. Sci. Comput.}, 39(1):A1--A24, 2017.

\bibitem{Mikelic13}
A.~Mikeli\'{c} and M.~F. Wheeler.
\newblock Convergence of iterative coupling for coupled flow and geomechanics.
\newblock {\em Comput. Geosci.}, 17(3):455--461, 2013.

\bibitem{Murad92}
M.~A. Murad and A.~F.~D. Loula.
\newblock Improved accuracy in finite element analysis of {Biot's}
  consolidation problem.
\newblock {\em Comput. Methods Appl. Mech. Engrg.}, 95(3):359--382, 1992.

\bibitem{Murad94}
M.~A. Murad and A.~F.~D. Loula.
\newblock On stability and convergence of finite element approximations of
  {Biot's} consolidation problem.
\newblock {\em Internat. J. Numer. Methods Engrg.}, 37(4):645--667, 1994.

\bibitem{Murad96}
M.~A. Murad, V.~Thomée, and A.~F.~D. Loula.
\newblock Asymptotic behavior of semidiscrete finite-element approximations of
  {Biot’s} consolidation problem.
\newblock {\em SIAM J. Numer. Anal.}, 33(3):1065--1083, 1996.

\bibitem{Niu20}
C.~Niu, H.~Rui, and X.~Hu.
\newblock A stabilized hybrid mixed finite element method for poroelasticity.
\newblock {\em Comput. Geosci.}, 2020.

\bibitem{Nordbotten16}
J.~M. Nordbotten.
\newblock Stable cell-centered finite volume discretization for {Biot}
  equations.
\newblock {\em SIAM J. Numer. Anal.}, 54(2):942--968, 2016.

\bibitem{Oyarzua16}
R.~Oyarzúa and R.~Ruiz-Baier.
\newblock Locking-free finite element methods for poroelasticity.
\newblock {\em SIAM J. Numer. Anal.}, 54(5):2951--2973, 2016.

\bibitem{Phillips07}
P.~J. Phillips and M.~F. Wheeler.
\newblock A coupling of mixed and continuous {Galerkin} finite element methods
  for poroelasticity {I}: the continuous in time case.
\newblock {\em Comput. Geosci.}, 11(2):131--144, 2007.

\bibitem{Phillips072}
P.~J. Phillips and M.~F. Wheeler.
\newblock A coupling of mixed and continuous {Galerkin} finite element methods
  for poroelasticity {II}: the discrete-in-time case.
\newblock {\em Comput. Geosci.}, 11(2):145--158, 2007.

\bibitem{Phillips09}
P.~J. Phillips and M.~F. Wheeler.
\newblock Overcoming the problem of locking in linear elasticity and
  poroelasticity: an heuristic approach.
\newblock {\em Comput. Geosci.}, 13(1):5--12, 2009.

\bibitem{Rodrigo16}
C.~Rodrigo, F.~J. Gaspar, X.~Hu, and L.T. Zikatanov.
\newblock Stability and monotonicity for some discretizations of the {Biot’s}
  consolidation model.
\newblock {\em Comput. Methods Appl. Mech. Engrg.}, 298(1):183--204, 2016.

\bibitem{Showalter00}
R.~E. Showalter.
\newblock Diffusion in poro-elastic media.
\newblock {\em J. Math. Anal. Appl.}, 251(1):310--340, 2000.

\bibitem{Showalter10}
R.~E. Showalter.
\newblock Nonlinear degenerate evolution equations in mixed formulation.
\newblock {\em SIAM J. Numer. Anal.}, 42(5):2114--2131, 2010.

\bibitem{Showalter13}
R.~E. Showalter.
\newblock {\em Monotone operators in Banach space and nonlinear partial
  differential equations}.
\newblock American Mathematical Society, 2013.

\bibitem{Wheeler14}
M.~Wheeler, G.~Xue, and I.~Yotov.
\newblock Coupling multipoint flux mixed finite element methods with continuous
  {Galerkin} methods for poroelasticity.
\newblock {\em Comput. Geosci.}, 18(1):57--75, 2014.

\bibitem{Yi13}
S.-Y. Yi.
\newblock A coupling of nonconforming and mixed finite element methods for
  {Biot’s} consolidation model.
\newblock {\em Numer. Methods Partial Differential Equations},
  29(5):1749--1777, 2013.

\bibitem{Yi14}
S.-Y. Yi.
\newblock Convergence analysis of a new mixed finite element method for
  {Biot's} consolidation model.
\newblock {\em Numer. Methods Partial Differential Equations},
  30(4):1189--1210, 2014.

\bibitem{Yi17}
S.-Y. Yi.
\newblock A study of two modes of locking in poroelasticity.
\newblock {\em SIAM J. Numer. Anal.}, 55(4):1915--1936, 2017.

\bibitem{ZhaoChungLam20}
L.~Zhao, E.~T. Chung, and M.~Lam.
\newblock A new staggered {DG} method for the {Brinkman} problem robust in the
  {Darcy} and {Stokes} limits.
\newblock {\em Comput. Methods Appl. Mech. Engrg.}, 364(1), 2020.

\bibitem{LinaEricPark20}
L.~Zhao, E.~T. Chung, E.-J. Park, and G.~Zhou.
\newblock Staggered {DG} method for coupling of the {Stokes} and
  {Darcy-Forchheimer} problems.
\newblock {\em SIAM J. Numer. Anal.}, 2020, to appear.

\bibitem{LinaKim20}
L.~Zhao, D.~Kim, E.-J. Park, and E.~Chung.
\newblock Staggered {DG} method with small edges for {Darcy} flows in fractured
  porous media.
\newblock {\em arXiv:2005.10955}, 2020.

\bibitem{LinaParkdiffusion18}
L.~Zhao and E.-J. Park.
\newblock A staggered discontinuous {Galerkin} method of minimal dimension on
  quadrilateral and polygonal meshes.
\newblock {\em SIAM J. Sci. Comput.}, 40(4):A2543--A2567, 2018.

\bibitem{ZhaoPark20}
L.~Zhao and E.-J. Park.
\newblock A lowest-order staggered {DG} method for the coupled {Stokes–Darcy}
  problem.
\newblock {\em IMA J. Numer. Anal.}, 40(4):2871--2897, 2020.

\bibitem{LinaParkelasticity20}
L.~Zhao and E.-J. Park.
\newblock A staggered cell-centered {DG} method for linear elasticity on
  polygonal meshes.
\newblock {\em SIAM J. Sci. Comput.}, 42(4):A2158--A2181, 2020.

\bibitem{ZhaoParkShin19}
L.~Zhao, E.-J. Park, and D.~w.~Shin.
\newblock A staggered {DG} method of minimal dimension for the {Stokes}
  equations on general meshes.
\newblock {\em Comput. Methods Appl. Mech. Engrg.}, 345(1):854--875, 2019.

\end{thebibliography}

\end{document}